\documentclass[11pt, twoside, leqno]{article}

\usepackage{amsmath}
\usepackage{amssymb}
\usepackage{titletoc}
\usepackage{mathrsfs}
\usepackage{amsthm}
\usepackage{color}
\usepackage{latexsym,amssymb}
\usepackage{indentfirst}
\usepackage{color}
\usepackage{txfonts}

\allowdisplaybreaks

\def\bint{{\ifinner\rlap{\bf\kern.30em--}
\int\else\rlap{\bf\kern.35em--}\int\fi}\ignorespaces}

\def\sbint{{\ifinner\rlap{\bf\kern.32em--}
\hspace{0.078cm}\int\else\rlap{\bf\kern.45em--}\int\fi}\ignorespaces}

\def\nn{{\mathbb N}}
\def\zz{{\mathbb Z}}

\def\XXint#1#2#3{{\setbox0=\hbox{$#1{#2#3}{\int}$ }
\vcenter{\hbox{$#2#3$ }}\kern-.6\wd0}}

\def\\det{\text{det}}

\def\.5{\frac{1}{2}}

\DeclareMathOperator{\dist}{dist}
\newtheorem{theorem}{Theorem}[section]
\newtheorem{definition}{Definition}[section]
\newtheorem{lemma}{Lemma}[section]
\newtheorem{remark}{Remark}[section]
\newtheorem{proposition}{Proposition}[section]
\newtheorem{corollary}{Corollary}[section]

\numberwithin{equation}{section}

\newcommand{\RN}[1]{%
  \textup{\uppercase\expandafter{\romannumeral#1}}%
}

\renewcommand{\epsilon}{\varepsilon}
\newcommand{\va}{\ensuremath{\varepsilon}}
\newcommand{\ptl}{\ensuremath{\partial}}
\newcommand{\om}{\ensuremath{\Omega}}

\newcommand{\R}{\ensuremath{\mathbb{R}}}
\newcommand{\bt}{\ensuremath{\beta}}

\newcommand{\lam}{\ensuremath{\lambda}}
\newcommand{\alp}{\ensuremath{\alpha}}
\newcommand{\dt}{\ensuremath{\delta}}

\newcommand{\wdt}{\ensuremath{\widetilde}}
\newcommand{\ud}{\mathrm{d}}

\newcounter{marnote}

\renewcommand{\appendix}{\par
   \setcounter{section}{0}%
   \setcounter{subsection}{0}%
   \setcounter{subsubsection}{0}%
   \gdef\thesection{\@Alph\c@section}%
   \gdef\thesubsection{\@Alph\c@section.\@arabic\c@subsection}%
   \gdef\theHsection{\@Alph\c@section.}%
   \gdef\theHsubsection{\@Alph\c@section.\@arabic\c@subsection}%
   \csname appendixmore\endcsname
 }

\begin{document}

\arraycolsep=1pt

\title{\bf\Large Gradient estimates for elliptic systems from composite materials with
 closely spaced stiff $C^{1,\gamma}$ inclusions
\footnotetext{
\hspace{-0.25cm} 2020 {\it Mathematics Subject Classification}. Primary: 74G70;
Secondary: 35B45 35B65 74B05.
\endgraf {\it Key words and phrases.}  Stress concentration, Schauder estimates,
gradient estimates, H\"{o}lder semi-norm, elliptic systems.
}}
\date{}
\author{Yan Li\footnote{Corresponding author,
Yan Li
E-mail: \texttt{yanli@mail.bnu.edu.cn}}}
\maketitle

\vspace{-0.8cm}

\begin{center}
\begin{minipage}{13cm}
{\small {\bf Abstract}\quad
This paper is devoted to establishing the
 pointwise upper and lower bounds estimates of the gradient
of the solutions to a class of general elliptic systems
with H\"{o}lder continuous coefficients in a narrow region
where the upper and lower boundaries is $C^{1,\gamma},$ $0<\gamma<1$,
weaker than the previous $C^{2,\gamma}$ assumption.
These estimates play a key role in the damage analysis of composite materials.
From our results, the damage may initiate from the narrowest place.}
\end{minipage}
\end{center}

\vspace{0.2cm}

\tableofcontents

\vspace{0.1cm}

\section{Introduction\label{s1}}

\subsection{Background and problem formulation}
Damage analysis of composite materials is of great significance in engineering,
which is also an important application of gradient estimates of second-order elliptic systems
 in partial differential equations \cite{dong,llby,ln}.
Babu\v{s}ka et al. \cite{basl} computationally
analyzed the damage and fracture in fiber composite materials,
and some numerical results of stress concentration are given,
which plays a key role in the problem of stress concentration
in materials.
At present, a lot of progress has been made in precisely
studying this field concentration phenomenon, see e.g.
\cite{ackly,bll,CL,dong,kly0,ky,LHLY,ln}.

From the viewpoint of mathematic, Li and Vogelius \cite{lv} described this stress concentration
 by using  the gradient of the solution
to a specific class of  elliptic systems
with partially degenerated coefficients.
Before studying the gradient estimates of elliptic systems,
Bonnetier et al. \cite{bv} considered the simplified scalar equation:
\begin{equation}\label{es}
\nabla \cdot \big((1+(a-1)\chi_{\cup_{i=1}^{N}D_{i}})\nabla v\big)=0,\quad \hbox{in}~~\Omega,
\end{equation}
to model a problem of  electric conduction,
where $a\ne1,$ $\Omega\subset\mathbb{R}^{n}$ represents a bounded domain,
$N\in \mathbb{Z}_{+}$ represents the number of the inclusions,
$D_i\subset\Omega$  represents the inclusions in the matrix material
and  close to each other.
Bonnetier and Vogelius  \cite{bv} proved  rigorously that the gradient
of the solution of \eqref{es} is indeed bounded,
when $N=2$ and $D_{1}, D_{2}$ is
two touching disks with comparable radii
in $\mathbb{R}^{2}$.
Li and Vogelius \cite{lv} extended the result  to
a large class of divergence form second order
 elliptic equations with piecewise
 H\"{o}lder continuous coefficients in $\mathbb{R}^{2}$,
 when $N\geq 2$ and the inclusions is $C^{1,\gamma},
 (0<\gamma<1)$ (see Definition \ref{defnbdd} below).
Li and Nirenberg \cite{ln} further extended
 to general divergence form
elliptic systems with H\"{o}lder continuous coefficients
satisfying the strong elliptic condition.

When the coefficients  degenerate to infinity in $D_{i},$
the gradient of the solution is no longer bounded but blows up.
For the scalar case, we call it perfect conductivity problem.
Let $\va$ be the distance between the two inclusions.
The blow-up rate of $|\nabla u|$ is respectively
$\va^{-1/2}$ in  two dimensions, $(\varepsilon|\ln \varepsilon|)^{-1}$
 in three dimensions and $\varepsilon^{-1}$ in four dimensions and high dimensions.
 See Bao, Li and Yin \cite{bly1},   Ammari, Kang and Lim \cite{akl},
 Ammari, Kang, Lee, Lee and Lim \cite{aklll}, and Yun \cite{y1}.
 There have been many papers on the problem and related ones: see e.g.
 \cite{ackly,adkl,akllz,bly2,bt2,kly0,llby,ly2} and the references therein.

However, when considering the gradient estimates to the solution of
linear elasticity problem, namely Lam\'{e} system,
the method of scalar equation is no longer suitable for using.
Under the assumption that the smoothness of the inclusion boundary is
$C^{2, \gamma}$ $(0<\gamma<1),$
Bao, Li and Li \cite{bll,bll2} applied an energy method and an iteration technique, which was first used in
\cite{llby}, to obtained pointwise upper bound of $|\nabla u|$ in the narrow region between inclusions.
 Kang and Yu \cite{ky} proved that the blow up rate $\va^{-1/2}$ is optimal in some two-dimensional cases
 when the smoothness of  inclusion boundary is of $C^{3,\gamma}.$
Ju, Li and Xu \cite{jlx} established the pointwise upper and lower bounds
of the gradient of solutions to a class of general elliptic systems in the narrow region between two $C^{2,\gamma}$ inclusions. For more work on elliptic equations and systems related to the study of composites, see \cite{bjl,bv,dong,dongzhang,JJ,GoBe1,GoBe2,kly0,LHLY,V1} and the references therein.

Under a weaker smoothness assumption on the inclusion boundary, namely, $C^{1,\gamma},$
Chen and Li \cite{CL} proved that the blow up rate of gradient for the Lam\'{e} system of linear
elasticity with partially infinite coefficients is $\va^{-1/(1+\gamma)}$ in two dimensions
and $\va^{-1}$ in $n\geq 3$ dimensions.

Contrary to the case where the smoothness of the inclusion boundary is $C^{2,\gamma}$ or higher,
less is known on such blow up phenomenon for the
case of weaker smoothness $C^{1,\gamma}.$
Based on the classical elliptic theory, a natural question is whether it is possible to obtain gradient estimates of the solutions to a class of general elliptic systems (see Definition \ref{defn1} below),
under a weaker smoothness assumption on the inclusions, namely, $C^{1,\gamma}.$
In addition, we want to obtain more information about the dependence of $|\nabla u|,$ which
play an important role in the study of the perfect conductivity problem  and Lam\'e system
with partially infinite coefficients.

In this paper, we investigate  the gradient estimates
of the solutions to a class of general elliptic systems with
H\"{o}lder continuous  coefficients  in a general narrow region between the two $C^{1,\gamma}$ inclusions.
This is a generalization of the stress concentration problem
 in two-phase high-contrast elastic composites with
densely packed $C^{1,\gamma}$ inclusions.
This estimates have a wide range of applications and
play a key role in the damage analysis of composite materials.
When we apply  this results to the Lam\'{e} systems of linear elasticity,
 under the assumptions of the $C^{1, \gamma}$-regularity  boundary,
the results can present more dependency information about gradient.
Our results show that the damage may initiate from the narrowest place.

Before state our results, we first introduce some definitions and notations, as well as fix our domain.

Let $U$ be any domain in $ \mathbb{R}^{n}.$
Denote by the \emph{symbol} $C(U)$ the
set of all continuous functions on $U.$
For every $0<\gamma\leq1,$
a function $u \in C(U)$ is said to be \emph{H\"{o}lder continuous with exponents $\gamma$} if $
|u(x)-u(y)|\leq C|x-y|^{\gamma},$ $x,y\in U$ for some  constant $C.$
If $u: U \rightarrow \mathbb{R}$ is bounded and continuous, we write
$\|u\|_{C(U)}:=\sup _{x \in U}|u(x)|.$
We use the \emph{symbol} $C^{k}(U)$ to
denote the set of all $k$-th continuously
differentiable functions on $U$ for  integer $k\geq0$.
For every $0<\gamma\leq1$, the $\gamma^{\text {th }}$-H\"{o}lder semi-norm of $u: U \rightarrow \mathbb{R}$ is
\begin{align}\label{semi}
[u]_{C^{0, \gamma}(U)}:=\sup _{\substack{x, y \in U\\  x \neq y}}
\Big\{\frac{|u(x)-u(y)|}{|x-y|^{\gamma}}\Big\}.
\end{align}
\begin{definition}
Let $0<\gamma\leq1$, $k\in \mathbb{Z}_{+}$,
and $\alp=(\alp_1,\cdots,\alp_{n})\in \mathbb{Z}^{n}$ be
a multiindex of order $|\alp|=\alp+\cdots+\alp_{n}$.
The \emph{H\"{o}lder space} $C^{k, \gamma}(U)$ is defined to be the
set of all  $k$-th  continuous differentiable real valued functions
satisfying the $k$-th order derivatives are H\"{o}lder continuous with exponents $\gamma$
and
$$
\|u\|_{C^{k, \gamma}(U)}:=\sum_{|\alpha| \leq k}\left\|\ptl^{\alpha} u\right\|_{C(U)}+\sum_{|\alpha|=k}\left[\ptl^{\alpha} u\right]_{C^{0, \gamma}(U)}< \infty,
$$
where $\partial^{\alp}u:=\frac{\ptl^{\alp_1}}{\ptl x_1^{\alp_1}}
\cdots\frac{\ptl^{\alp_n}}{\ptl x_n^{\alp_n}} u.$
In particular, for $0<\gamma<1$, we often use the \emph{symbol} $C^{\gamma}(U)$ to denote $C^{0,\gamma}(U).$
\end{definition}

\begin{definition}\label{defnbdd}
Let $U$ be any domain in $ \mathbb{R}^{n}$, the integer $k>0$ and $0<\gamma<1,$
the boundary $\ptl U$ is said to be $C^{k}$ or $C^{k, \gamma}$
if for each point $x_{0}\in \ptl U$ there exist $r>0$ and a
$ C^{k}$ or $C^{k, \gamma}$  function
$T: \mathbb{R}^{n-1} \rightarrow \mathbb{R}$
such that-upon relabeling and reorienting
the coordinates axes if necessary-we have
$$
U \cap B_{r}(x_{0})=\left\{x \in B_{r}(x_{0}) \mid x_{n}>T\left(x_{1}, \ldots, x_{n-1}\right)\right\}.
$$
Furthermore, the inclusion is said to
be $C^{k}(U)$ or $C^{k, \gamma}(U)$ if its boundary is $C^{k}(U)$ or $C^{k, \gamma}(U)$.
\end{definition}

\begin{definition}
Let $ 1 \leq p \leq \infty$, $k\in \mathbb{Z}_{+}$ and
$\alp=(\alp_1,\cdots,\alp_{n})\in \mathbb{Z}^{n}$ be
a multiindex of order $|\alp|=\alp+\cdots+\alp_{n}$.
The \emph{Sobolev space} $W^{k, p}(U)$ is defined to be the set of
all locally integrable functions
 $u: U \rightarrow \mathbb{R}$ such that for each multiindex $\alpha\in \mathbb{Z}_{+}^{n}$
 with $|\alpha| \leq k,$ $\ptl_{\alpha} u$ exists in the weak sense and belongs to the standard Lebesgue spaces $L^{p}(U)$ and
 $$
\|u\|_{W^{k, p}(U)}:= \begin{cases}\Big(\sum\limits_{|\alpha| \leq k} \int_{U}|\ptl^{\alpha} u|^{p} \ud x\Big)^{1 / p} & (1 \leq p<\infty) \\ \sum\limits_{|\alpha| \leq k} \operatorname{ess} \sup _{U}\left|\ptl^{\alpha} u\right| & (p=\infty)\end{cases}
< \infty.
 $$
 \end{definition}
Denote by the  \emph{symbol} $C_{\mathrm{c}}^{\infty}(U)$  the set of all infinitely continuously
differentiable functions on $U$ with compact support.
We denote by $W_{0}^{k, p}(U)$ the closure of $C_{c}^{\infty}(U)$ in $W^{k, p}(U)$.

Next, we fix our domain.
Let $D$ be a domain in $\mathbb{R}^{n}$, and $D_{1}$, $D_{2}$
be a pair of convex subdomains of $D\subset \mathbb{R}^n.$ Let the distance
 between $D_{1}$ and $D_{2}$ be $\varepsilon>0$ (sufficiently small positive number).
Denote $P_{1}:=(\vec{0}_{n-1},\varepsilon/2),$
 $P_{2}:=(\vec{0}_{n-1},-\varepsilon/2)$ the nearest points between $\partial D_{1}$ and $\partial D_{2}$ such that
\begin{align}\label{vare}
\dist(P_{1},P_{2})=\dist(\partial D_{1},\partial D_{2})=\varepsilon,
\end{align}
where for any $x' \in \mathbb{R}^{n-1},$ $x:=(x',x_n)\in\mathbb{R}^{n}.$
Let $B'_{r}$ be the ball with $\vec{0}_{n-1}$ as the center and $0< r\leq 1$ as the radius in $\R^{n-1}.$

Let $\va$ be as in \eqref{vare}, $h_{1}, h_{2}\in C^{1,\gamma}(B'_{1}),$ $0<\gamma<1$  and satisfy
\begin{align}
-\frac{\varepsilon}{2}+h_{2}(x')<\frac{\varepsilon}{2}+h_{1}(x'),&\quad\hbox{for}~~|x'|\leq 1,\label{1.42}\\
h_{1}(\vec{0}_{n-1})=h_{2}(\vec{0}_{n-1})=0,&\quad \nabla h_{1}(\vec{0}_{n-1})=\nabla h_{2}(\vec{0}_{n-1})=0,\label{h}
\end{align}
and there exist some constants $0<\kappa_{0}<\kappa_{1}$  such that
\begin{align}\label{h'}
\kappa_{0}|x'|^{\gamma}\leq|\nabla h_{1}(x')|,\quad|\nabla h_{2}(x')|\leq\kappa_{1}|x'|^{\gamma},\quad\hbox{for}~~|x'|\leq 1.
\end{align}

To be precise, we define general narrow region in $\mathbb{R}^{n}$: for $r\leq1$,
\begin{align}\label{Omega}
\Omega_{r}:=\Big\{(x',x_n)\in D:-\frac{\varepsilon}{2}+h_2(x')<x_{n}<\frac{\varepsilon}{2}+h_{1}(x'),~|x'|\leq r\Big\}.
\end{align}

We here assume that $\ptl D_1$ and $\ptl D_2$ are
$C^{1, \gamma},$ $0<\gamma<1$ as in Definition \ref{defnbdd}, and
the top and bottom boundaries of the narrow region $\Omega_1$ between $\ptl D_1$  and $\ptl D_2$ satisfy
\begin{equation}
\begin{aligned}\label{h1h2'}
&\{(x',x_n)\in \mathbb{R}^n: x_n=\frac{\va}{2}+h_1(x'),\, |x'|\leq 1\}\subset \ptl D_{1},\\
&\{(x',x_n)\in \mathbb{R}^n: x_n=-\frac{\va}{2}+h_2(x'),\, |x'|\leq 1\}\subset \ptl D_{2},\\
\hbox{and}&\\
&\{(x',x_n)\in \mathbb{R}^n:  -\frac{\va}{2}+h_2(x')<x_{n}<\frac{\va}{2}+h_1(x'),
\,|x'|\leq 1 \}
\cap (\ptl D_{1}\cup\ptl D_{2})=\emptyset.
\end{aligned}
\end{equation}
Furthermore, we denote the top and bottom boundaries of $\Omega_{r}$ as
\begin{align}\label{boundary}
\begin{aligned}
\Gamma_{r}^{+}:&=\{x\in \partial D_{1}:\,x_n=\frac{\varepsilon}{2}+h_{1}(x'),\,|x'|\leq r\},\\
\Gamma_{r}^{-}:&=\{x\in  \partial D_{2}:\, x_n=-\frac{\varepsilon}{2}+h_{2}(x'),\,|x'|\leq r\}.
\end{aligned}
\end{align}

We now introduce the definition of
general elliptic system with
H\"{o}lder continuous  coefficients in a narrow region $\Omega_{1}$ in this paper.

\begin{definition}\label{defn1}
Let $0<\gamma<1$, $m,n\in \mathbb{Z}_{+}$, $A_{ij}^{\alp\bt}, B_{ij}^{\alp}, C_{ij}^{\bt},
D_{ij}\in C^{\gamma}(\Omega_1)$ for any integer $0\leq\alp,\bt\leq n,$ $0\leq i,j \leq m,$ and
the matrix of coefficients $(A_{ij}^{\alpha\beta})^{1\leq\alpha,\beta\leq n}_{1\leq i,j\leq m}$  satisfy the  \emph{strong ellipticity condition} in $\Omega_1$,  namely,
there exists a constant $\lambda>0$
such that
\begin{align}\label{LHC}
\sum_{\alp,\bt,i,j}A_{ij}^{\alpha\beta}(x)\xi_{\alpha}\xi_{\beta}\eta_{i}
\eta_{j}\geq\lambda|\xi|^2|\eta|^{2},\quad \forall\, \xi\in\mathbb{R}^{n},\,
\eta\in\mathbb{R}^{m},\,
 x\in \Omega_1.
\end{align}
Let
\begin{align}
\begin{aligned}\label{1.51}
\varphi:&=\left(\varphi^{(1)},\varphi^{(2)},\cdots,\varphi^{(m)}\right)
\in C^{1, \gamma}(\Gamma_{1}^{+};\mathbb{R}^{m}),\\
\psi:&=\left(\psi^{(1)},\psi^{(2)},\cdots,\psi^{(m)}\right)
\in C^{1,\gamma}(\Gamma_{1}^{-};\mathbb{R}^{m}).
\end{aligned}
\end{align}
Then the systems as follows
\begin{align}\label{equ1}
\begin{aligned}
\left\{
  \begin{array}{ll}
    \sum_{\alp,\bt,j}\partial_{\alpha}\left(A_{ij}^{\alpha\beta}\partial_{\beta}u^{(j)}
    +B_{ij}^{\alpha}u^{(j)}\right)
+C_{ij}^{\beta}\partial_{\beta}u^{(j)}+D_{ij}u^{(j)}=0 & \hbox{in}~~\Omega_1,\\
    {\sf u}=\varphi, & \hbox{on}~~\Gamma_{1}^{+},\\
    {\sf u}=\psi, &\hbox{on} ~~\Gamma_{1}^{-},
  \end{array}
\right.
\end{aligned}
\end{align}
is called a \textit{general elliptic system},
where $\Gamma_{1}^{+}$ and $\Gamma_{1}^{-}$ are defined as in \eqref{boundary}.
\end{definition}
A function
${\sf u}:=(u^{(1)},u^{(2)},\cdots, u^{(m)})\in W^{1,2}(\Omega_1\subset \R^{n};\mathbb{R}^{m})$
 is said to
be a  \itshape{weak solution} of  the \emph{general elliptic systems} defined as in Definition \ref{defn1} if,
 for every vector-valued function $\phi\in W_{0}^{1,2}(\Omega_1\subset \R^{n};\R^{m})$,
\begin{align}
\begin{aligned}\label{1.54}
\sum_{\alp,\bt,j}\int_{\Omega_1}&\left(A_{ij}^{\alpha\bt}(x)\ptl_{\beta}u^{(j)}(x)
+B_{ij}^{\alp}(x)u^{(j)}(x)\right)\ptl_{\alp}\phi^{(i)}(x)\\
&-C_{ij}^{\bt}(x)\ptl_{\bt}u^{(j)}(x)\phi^{(i)}(x)-D_{ij}(x)u^{(j)}(x)\phi^{(i)}(x)\,dx=0
\end{aligned}
\end{align}
holds true, for any integer $i=1,\cdots,m$.

\begin{remark}
It is clear that hypotheses \eqref{LHC} and \eqref{AUP} are satisfied by
 the  Lam\'{e} system as follow, (see \cite{osy}),
$$
\lambda_1 \Delta {\sf u}+(\lambda_1+\mu_1) \nabla(\nabla \cdot {\sf u})=0,
$$
where Lam\'e constant $(\lambda_1, \mu_1)$ satisfies  ellipticity conditions:
$\mu_1>0,$ $\lambda_1+\mu_1>0$.
Therefore, the gradient estimates results in this paper include the case of Lam\'{e} systems.
\end{remark}

Now we are going to present our main result about pointwise gradient
estimates of the weak solutions to the \textit{general elliptic systems}, which is:
\begin{theorem} \label{Them1}
Let $\va$ be as in \eqref{vare},
 $0<\gamma<1$,  $h_{1},h_{2}\in C^{1,\gamma}(B'_{1})$ satisfy \eqref{1.42}-\eqref{1.50},
and $\Omega_{r},$ $r\in \{1/2, 1\}$ be as in \eqref{Omega}.
Let $\varphi$ and $\psi$ be as in \eqref{1.51} and
${\sf u}\in W^{1,2}({\Omega_1\subset \R^{n};\R^{m}})$ be a weak
solution of \emph{general elliptic system}  as in Definition \ref{defn1}.
Then there exists a positive constant $C$ independent of $\va$, such that,
for any $x=(x',x_{n})\in \Omega_{1/2}$,
\begin{align*}%\label{1.43}
|\nabla {\sf u}(x',x_n)|
\leq&\frac{C}{\va+|x'|^{1+\gamma}}\left|\varphi(x',\varepsilon/2+h_{1}(x'))
-\psi(x',-\varepsilon/2+h_2(x'))\right|\nonumber\\
&+C\left(\|\varphi\|_{C^{1,\gamma}(\Gamma_{1}^+)}+\|\psi\|_{C^{1,\gamma}(\Gamma_1^-)} +\|{\sf u}\|_{L^{2}(\Omega_1)}\right),
\end{align*}
where $\Gamma_{1}^{+}$ and $\Gamma_{1}^{-}$ are  as in \eqref{boundary}.

Moreover, if $\varphi^{(\ell)}(\vec{0}_{n-1},\va/2)\ne\psi^{(\ell)}(\vec{0}_{n-1},-\va/2)$
for some integer $\ell\in \{1,\cdots,m\},$ then there exists a positive constants $C$
independent of $\va$, such that,
for any $x_n\in (-\va/2,\va/2)$,
\begin{align*}
|\nabla {\sf u}(\vec{0}_{n-1},x_n)
|\geq C\frac{|\varphi^{(\ell)}(\vec{0}_{n-1},\va/2)-\psi^{(\ell)}
(\vec{0}_{n-1}, -\va/2)|}{\va}.
\end{align*}
\end{theorem}

For the convenience of further applications, we list the analog result about
the Lam\'{e} systems in the narrow region, which is:
\begin{corollary}\label{cor1}
Let $\va$ be as in \eqref{vare},
$\Omega_{r},$ $r\in\{\frac{1}{2},1\}$ be as in \eqref{Omega},
 $\Gamma_{1}^{+},$ $\Gamma_{1}^{-}$ be as in \eqref{boundary},
and $\varphi\in C^{1, \gamma}(\Gamma_{1}^{+};\mathbb{R}^{n}),$ $\psi\in C^{1, \gamma}(\Gamma_{1}^{-};\mathbb{R}^{n})$ be as in \eqref{1.51}.
Let $\lambda_{1},$ $\mu_{1}\in \mathbb{R}$ be the pair of Lam\'{e} constants which satisfies
the strong ellipticity conditions:   $\mu_1>0,$ $\lambda_1+\mu_1>0$. Let
${\sf u}=(u^{(1)}, \cdots, u^{(n)})\in W^{1,2}(\Omega_1)$ be the weak solution of
\begin{align}\label{lame}
\begin{aligned}
\left\{
  \begin{array}{ll}
  \mathcal{L}_{\lambda_{1},\mu_{1}} {\sf u}:=\nabla \cdot (\mathbb{C}^{0}e({\sf u}))=
    \lambda_1 \Delta {\sf u}+(\lambda_1+\mu_1) \nabla(\nabla \cdot {\sf u})=0,
     & \,\hbox{in}~~\Omega_1,\\
    {\sf u}=\varphi, & \,\hbox{on}~~\Gamma_{1}^{+},\\
    {\sf u}=\psi, &\,\hbox{on} ~~\Gamma_{1}^{-},
  \end{array}
\right.
\end{aligned}
\end{align}
where $e(\sf u)=\frac{1}{2}(\nabla u+(\nabla u)^{T})$ and
the  elastic tensor $\mathbb{C}^{0}=\mathbb{C}(\lambda_{1},\mu_{1})$ consists of elements
\begin{align}\label{cijkl}
C_{ijkl}(\lambda_{1},\mu_{1})=\lambda_{1}\delta_{ij}\delta_{kl}
+\mu_{1}(\delta_{ik}\delta_{jl}+\delta_{il}\delta_{jk}),\quad i,j,k,l=1,2,\cdots,n,
\end{align}
 and $\delta_{ij}$ is Kronecker symbol: $\delta_{ij}=0$ for $i\ne j,$ $\delta_{ij}=1$ for $i=j$.
Then there exists a positive constant $C$ independent of $\va$,
 for any $x=(x_1,x_2)\in \Omega_{1/2}$,
\begin{align*}
|\nabla {\sf u}(x',x_n)|
\leq&\frac{C}{\va+|x'|^{1+\gamma}}\left|\varphi(x',\varepsilon/2+h_{1}(x'))
-\psi(x',-\varepsilon/2+h_2(x'))\right|\nonumber\\
&+C\left(\|\varphi\|_{C^{1,\gamma}(\Gamma_{1}^+)}+\|\psi\|_{C^{1,\gamma}(\Gamma_1^-)} +\|{\sf u}\|_{L^{2}(\Omega_1)}\right).
\end{align*}
Moreover, if $\varphi^{(\ell)}(\vec{0}_{n-1},\va/2)\ne\psi^{(\ell)}(\vec{0}_{n-1},-\va/2)$
for some integer $\ell\in \{1,\cdots,n\}$, then there exists a positive constants $C$
independent of $\va$,  such that,
for any $x_n\in (-\va/2,\va/2)$,
\begin{align*}
|\nabla {\sf u}(\vec{0}_{n-1},x_n)
|\geq C\frac{\left|\varphi^{(\ell)}(\vec{0}_{n-1},\va/2)-\psi^{(\ell)}
(\vec{0}_{n-1}, -\va/2)\right|}{\va}.
\end{align*}
\end{corollary}

Indeed, the proof of this result is nontrivial.
Under the assumption that the smoothness of the inclusions
 boundaries is $C^{1,\gamma},$ we need to estimate the
 H$\mathrm{\ddot{o}}$lder semi-norm of the gradient of the
 constructed auxiliary function when we use the iteration method to prove that the
 gradient of the auxiliary function is the major singular terms of the gradient of the solution to \eqref{equ1}.
In this paper, we consider the general elliptic system as in Definition \ref{defn1}, so the complexity of the constructed auxiliary function makes it more cumbersome to deal with some parameters in estimating  the H$\mathrm{\ddot{o}}$lder semi-norm, see Proposition \ref{prop1} below.
In addition, compared with \cite{CL}, the coefficients of the elliptic systems here are no longer constant, and the right hand side term is no longer in divergence form, which leads to the complexity of the iteration process,
see Lemmas \ref{lem2}, \ref{lemmabddlocal}, and \ref{lem1} below.

To be precise, this article is organized as follows.

In Section \ref{sc1}, we devoted to proving the  Theorem \ref{Them1}.
Firstly, we give the $C^{1,\gamma}$ estimates and $W^{1,p}$ estimates
required in the iteration process. Then, we obtain the solutions ${\sf v}_{\ell},$ $\ell=1,\cdots,m$
(see \eqref{1.49} below) of $m$ elliptic systems
with relatively simple boundary conditions (see \eqref{vl} below)
by decomposing the solution ${\sf u}$ to \emph{general elliptic system}.
Next, we use a  scalar auxiliary function $\bar{u}$ (see \eqref{1.52} below)
to generate a family of vector-valued auxiliary functions $\wdt{\sf u}_{\ell}$,
whose values is same as ${\sf v}_{\ell}$ on $\Gamma_{1}^{+}$ and $\Gamma_{1}^{-}$ (see \eqref{ul} below).
In order to prove that
the  $\nabla\wdt{\sf u}_{\ell}$ is
the major singular terms of $\nabla {\sf v}_{\ell}$ by using iteration method in \cite{bll,bll2,llby},
it is very important to consider the estimates of the H\"{o}lder semi-norm
of $\nabla \wdt{\sf u}_{\ell}$ in a small region (see Proposition \ref{prop1} below).
Finally,
by using semi-norm estimates,
$C^{1,\gamma}$ estimates and $W^{1,p}$ estimates,
we complete the proof of Theorem \ref{Them1}.

In Section \ref{sc3}, our main task is to prove  Corollary \ref{cor1}.
When the general elliptic systems are simplified to the Lam\'{e} system,
 the pointwise estimates of  the gradient
 is obtained under the assumptions in
Definition \ref{defn1}.

In Section \ref{sec5}, we show the proofs of the Theorem \ref{C1Gamma} and Theorem \ref{them2}
which play a key role in the proof of Theorem \ref{Them1}.
Different from the Theorem 2.3 and 2.4 in \cite{CL},
the  elliptic systems  considered in this paper are
 no longer simple constant coefficients,
  and the right hand side of the systems is no longer in divergence form.
  This section can be regarded as a generalization of \cite[Theorem 2.3 and 2.4]{CL}.
To prove Theorem \ref{C1Gamma}, we first give the interior $C^{1,\gamma}$  estimates
(see Lemma \ref{them5.14} below)
of the solution to the elliptic systems with non divergence
 form at the right hand side,
with the help of the Campanato's approach, Schauder estimates in \cite[Theorem 5.14]{Gia}.
Then we can obtain the boundary $C^{1,\gamma}$ estimates
 (see Corollary \ref{them3} below)
 on half space by using the
 \cite[Theorem 5.21]{Gia} and Lemma \ref{them5.14}.
 Finally, we use them to obtain the estimates
 near the $C^{1,\gamma}$ boundary $\Gamma$
 by using the technology of
locally flattening the boundary.
We prove the $W^{1,p}$ estimates
by applying  the interior $W^{1,p}$ estimates
 and the boundary $W^{1,p}$ estimates
of the upper half space.

Finally, we make some conventions on notation.
Let $\nn:=\{1,2,\ldots\}$ and $\zz_+:=\nn\cup\{0\}$.
 We always denote by $C$ a \emph{positive constant}
which is independent of the main parameters, but it
may vary from line to line.
If $E$ is a subset of $\mathbb{R}^{n}$, we denote by $\chi_{E}$
its characteristic function.
Let $U$ and $V$ be the
open subsets of $\mathbb{R}^n$,
we write $V\subset\subset U$ if
$V\subset\bar{V}\subset U$ and $\bar{V}$ is
compact, and say $V$ is compactly contained in $U$.
The \emph{symbol} $\partial{D}_{i}$ denotes the boundary of $D_{i}$.
For any $x_{0}\in \mathbb{R}^n$ and $r>0$, the \emph{symbol} $B_{r}(x_0)$ denotes
 the open ball with center $x_{0}$, radius $r$ in $\mathbb{R}^n$.

Throughout this article, let
$\Lambda>0$  satisfy for $\alpha,\beta=1,\cdots, n,\, i,j=1,\cdots, m,$
\begin{align}\label{AUP}
\left|A_{ij}^{\alpha\beta}(x)\right|\leq \Lambda,\quad
\forall\, x\in \Omega_1,
\end{align}
and $\kappa_{2},$ $\kappa_{3}>0$ such that
\begin{align}
&\|h_1\|_{C^{1,\gamma}(B'_1)}+\|h_{2}\|_{C^{1,\gamma}(B'_{1})}\leq\kappa_{2},\label{1.50}\\
&\|{\sf A}\|_{C^{\gamma}(\Omega_1)}+\|{\sf B}\|_{C^{\gamma}(\Omega_1)}+\|{\sf C}\|_{C^{\gamma}(\Omega_1)}+\|{\sf D}\|_{C^{\gamma}(\Omega_1)}\leq\kappa_{3},\label{BCD}
\end{align}
where \begin{align*}
\|{\sf A}\|_{C^{\gamma}(\Omega_1)}&:=\max_{ \alp,\bt,i,j}\|A_{ij}^{\alp\bt}(\cdot)\|_{C^{\gamma}(\Omega_1)}, \quad
\|{\sf B}\|_{C^{\gamma}(\Omega_1)}:=\max_{\alp,i,j}\|B_{ij}^{\alp}(\cdot)\|_{C^{\gamma}(\Omega_1)},\\
\|{\sf C}\|_{C^{\gamma}(\Omega_1)}&:=\max_{\bt,i,j}\|C_{ij}^{\bt}(\cdot)\|_{C^{\gamma}(\Omega_1)},\quad
\|{\sf D}\|_{C^{\gamma}(\Omega_1)}:=\max_{i,j}\|D_{ij}(\cdot)\|_{C^{\gamma}(\Omega_1)}.
\end{align*}

\section{The gradient estimates for the general elliptic systems}\label{sc1}
In this section, our main task is to prove Theorem \ref{Them1}.
By using the idea of iteration in  \cite{bll,bll2,llby} and the treatment of $C^{1,\gamma}$
boundary in \cite{CL}, we establish the pointwise upper and lower bounds estimates of the gradient of the solution to general elliptic systems, under the assumption that the smoothness of  partial boundary of the region
is $C^{1,\gamma}.$
Before that, we first give the $C^{1,\gamma}$ estimates and $W^{1,p}$ estimates
required in the iteration process.

\subsection{$C^{1,\gamma}$ estimates and $W^{1,p}$ estimates}
Firstly, we give the definition of \emph{general elliptic type system} in a more general region.
\begin{definition}\label{defn2}
Let $0<\gamma<1$ and  $Q$ be a bounded domain in $\mathbb{R}^{n},$ $n\geq2$, with a $C^{1,\gamma}$
 boundary portion $\Gamma\subset \partial Q$.
Let ${A}_{ij}^{\alp\bt}, {B}_{ij}^{\alpha}$,
${C}_{ij}^{\beta},{D}_{ij},{F}_{i}^{(\alp)}\in C^{\gamma}(Q)$
and ${H}^{(i)}\in L^{\infty}(Q)$
for any
$\alp,\bt=1,\cdots, n,$  $i,j=1,\cdots,m$.
Let the matrix of coefficients $({A}_{ij}^{\alp\bt})$
satisfy  \eqref{LHC} and \eqref{AUP},
and ${B}_{ij}^{\alpha}$,
${C}_{ij}^{\beta},{D}_{ij}$ satisfy \eqref{BCD}.
Then the systems as follows
\begin{align}\label{system}
\left\{
  \begin{array}{ll}
\sum\limits_{\alp,\bt,j}\partial_{\alpha}\left({A}_{ij}^{\alpha\beta}\partial_{\beta}{w}^{(j)}
+{B}_{ij}^{\alpha}{w}^{(j)}\right)
+{C}_{ij}^{\beta}\partial_{\beta}
{w}^{(j)}+{D}_{ij}{w}^{(j)}={H}^{(i)}
-\partial_{\alpha} {F}_{i}^{(\alpha)}, & \hbox{in}~Q, \\
    {\sf w}:=({w}^{(1)},\cdots,{w}^{(m)})={\sf 0 }, & \hbox{on}~\Gamma,
  \end{array}
\right.
\end{align}
is called a \emph{general elliptic type system}.
\end{definition}

 The $C^{1,\gamma}$ estimates for the general elliptic type system as in Definition \ref{defn2} is:
\begin{theorem}\label{C1Gamma}($C^{1,\gamma}$ estimates)
Let $\kappa_{3}$ be as in \eqref{BCD} and ${\sf w}=({w}^{(1)},\cdots,{w}^{(m)})
\in W^{1,2}(Q\subset\mathbb{R}^n;\mathbb{R}^{m})\cap C^{1}(Q \cup \Gamma;\mathbb{R}^{m})$ be the
weak solution to the \emph{general elliptic type system} as in
 Definition \ref{defn2}. Then, for any
$Q'\subset\subset Q\cup\Gamma$,
there exists a positive constant $C$ depending on $n, m, \kappa_3, \gamma,
Q',Q,$ such that,
\begin{align}\label{WC1gamma}
\|{\sf w}\|_{C^{1,\gamma}(Q')}\leq C\left(\|{\sf w}\|_{L^{\infty}(Q)}+[{\sf F}]_{\gamma, Q}
+   \|{\sf H}\|_{L^{\infty}(Q)} \right).
\end{align}
\end{theorem}
Here and thereafter,
the H$\mathrm{\ddot{o}}$lder semi-norm of
the matrix-valued function ${\sf F}=({F}_{i}^{(\alpha)})_{i, \alpha}$
be defined as
$
[{\sf F}]_{\gamma,Q}:=
\max_{i,\alpha}[{F}_{i}^{(\alpha)}]_{\gamma, Q},
$
the $L^{\infty}$ norm of the vector-valued function ${\sf H}:=({H}^{(1)},\cdots,{H}^{(m)})$
be defined as
$\|{\sf H}\|_{L^{\infty}(Q)}:=\max_{i}\|{H}^{(i)}\|_{L^{\infty}(Q)}.$

The $W^{1,p}$ estimates for the general elliptic type system as in Definition \ref{defn2} is:
\begin{theorem}\label{them2}($W^{1,p}$ estimates)
Let $0<\gamma<1$ and ${\sf w}=({w}^{(1)},\cdots,{w}^{(m)})
\in W^{1,2}(Q\subset\mathbb{R}^n;\mathbb{R}^{m})\cap C^{1}(Q \cup \Gamma;\mathbb{R}^{m})$ be the
weak solution of the \emph{general elliptic  type system} as in Definition \ref{defn2}.
 Then, for any $2\leq p<\infty$ and any domain $Q^\prime \subset\subset Q \cup \Gamma$,
there exists a positive constant $C$ depending on $\lam, \kappa_3, p, Q', Q$, such that,
\begin{equation}\label{1.39}
\|{\sf w}\|_{W^{1, p}(Q^\prime)}\leq C(\|{\sf w}\|_{L^{2}(Q)}+[{\sf F}]_{\gamma, Q}
+\|{\sf H}\|_{L^{\infty}(Q)}).
\end{equation}
In particular, if $p> n$, there exists a positive constant $C$
depending on $\lam, \tau, \kappa_3, p, Q', Q$, such that, for any $0<\tau\leq 1-n/p$,
\begin{equation*}
\|{\sf w}\|_{C^{\tau}(Q')}\leq C(\|{\sf w}\|_{L^{2}(Q)}
+[{\sf F}]_{\gamma, Q}+\|{\sf H}\|_{L^{\infty}(Q)}).
\end{equation*}
\end{theorem}

For readers' convenience, the proofs of Theorem \ref{C1Gamma} and Theorem \ref{them2} are given in Section \ref{sec5}.

\subsection{The proof of Theorem \ref{Them1}}

\begin{definition}
Let
$A_{ij}^{\alp\bt}, B_{ij}^{\alp}, C_{ij}^{\bt},D_{ij},$ $\alp,\bt=1,\cdots,n,$ $i,j=1,\cdots, m$, be as in
Definition \ref{defn1} and $\varphi,$ $\psi$ be as in \eqref{1.51}.
Let
$${\sf v}_{\ell}=(v_{\ell}^{(1)}, v_{\ell}^{(2)},\cdots,v_{\ell}^{(m)}),~~~~
\ell=1,2,\cdots, m$$
 with $v_{\ell}^{(j)}=0$ for $j\ne \ell,$
$v_{\ell}^{(j)}=u^{(\ell)}$ for $j=\ell$,
 and be a weak solution of the following
boundary value problem:
\begin{align}\label{vl}
\left\{
  \begin{array}{ll}
    \sum_{\alp,\bt,j}\partial_{\alpha}\left(A_{ij}^{\alpha\beta}\partial_{\beta}v_{\ell}^{(j)}
    +B_{ij}^{\alpha}v_{\ell}^{(j)}\right)
+C_{ij}^{\beta}\partial_{\beta}v_{\ell}^{(j)}+D_{ij}v_{\ell}^{(j)}=0,&\,\hbox{in}~\Omega_1,
 \\
   {\sf v}_{\ell}=\left( 0,\cdots,0,\varphi^{(\ell)},0,\cdots,0 \right), &\,\hbox{on}~\Gamma_{1}^{+}, \\
   {\sf v}_{\ell}=\left( 0,\cdots,0,\psi^{(\ell)},0,\cdots,0 \right), &\, \hbox{on}~\Gamma_{1}^{-}.
  \end{array}
\right.
\end{align}
\end{definition}
It follows from Definition \ref{defn1} that
for the solution ${\sf u}=(u^{(1)},\cdots,u^{(m)})$ of
\emph{general elliptic system} as in Definition \ref{defn1}, one has
\begin{align}
\sf{u}={v_1}+{v_2}+\cdots+{v_{m}}\quad{and}\quad
\nabla {\sf u}=\sum\limits_{\ell=1}^m\nabla {\sf v}_{\ell}~~ \mbox{in}\, \Omega_{1}.\label{1.49}
\end{align}

In order to estimate $|\nabla {\sf v}_{\ell}|,$ $\ell=1,\cdots, m$,
we introduce a scalar function $\bar {u}\in C^{1,\gamma}(\mathbb{R}^{n})$ such that
$\bar {u}=1$ on $\Gamma_{1}^{+},$  $\bar{u}=0$ on $\Gamma_{1}^{-}$ and
\begin{align}\label{1.52}
\bar{u}(x):=\frac{x_{n}-h_{2}(x')+\varepsilon/2}{\varepsilon+h_1(x')-h_2(x')},\quad x\in\Omega_{1},
\end{align}
where $h_{1},$ $h_{2}\in C^{1,\gamma}(B'_1)$ satisfy \eqref{1.42}-\eqref{h'}.

By a direct calculation, we obtain that for $x:=(x',x_{n})\in \Omega_{1}$,
\begin{align}\label{baru}
|\partial_{\alpha} \bar{u}(x)|\leq \frac{C|x'|^{\gamma}}{\va+|x'|^{1+\gamma}},\,\, \alpha=1,2,\cdots,n-1,\quad
\partial_{n}\bar{u}(x)=\frac{1}{\delta(x')},
\end{align}
where
\begin{align}\label{delta}
\delta(x'):=\varepsilon+h_1(x')-h_2(x'),\quad\hbox{in}\,\Omega_1.
\end{align}
Using $\bar{u}$ to define a family of vector-valued auxiliary functions, for $\ell=1,2,\cdots,m$, we define
\begin{align}\label{ul}
 \wdt{\sf u}_{\ell}:=\Big(0,\cdots,0,\varphi^{(\ell)}(x',\frac{\varepsilon}{2}+h_{1}(x'))\bar{u}(x)
+\psi^{(\ell)}(x',-\frac{\varepsilon}{2}+h_2(x'))(1-\bar{u}(x)),0,\cdots,0\Big).
\end{align}
It is obvious that $\wdt{\sf u}_{\ell}=( 0,\cdots,0,\varphi^{(\ell)},0,\cdots,0)$ on
$\Gamma_{1}^{+}$
and $\wdt{\sf u}_{\ell}=( 0,\cdots,0,\psi^{(\ell)},0,\cdots,0)$ on $\Gamma_{1}^{-}$.

In view of \eqref{baru} and \eqref{ul}, for any $x\in \Omega_{1}$ and  $\alpha=1,2,\cdots,n-1$,
\begin{align}\label{par_1}
\left|\partial_{\alpha}\tilde{\sf u}_{\ell}(x)\right|&\leq\frac{C|x'|^{\gamma}}{\va+|x'|^{1+\gamma}}\left|\varphi^{(\ell)}(x',\varepsilon/2+h_{1}(x'))
-\psi^{(\ell)}(x',-\varepsilon/2+h_2(x'))\right|\nonumber\\
&+C\left(\|\nabla\varphi^{(\ell)}\|_{L^{\infty}(\Gamma_1^+)}+\|\nabla\psi^{(\ell)}\|_{L^{\infty}(\Gamma_1^-)}\right),
\end{align}
and
\begin{align}
\begin{aligned}\label{par_n}
&\frac{\left|\varphi^{(\ell)}(x',\varepsilon/2+h_{1}(x'))
-\psi^{(\ell)}(x',-\varepsilon/2+h_2(x'))\right|}{C(\va+|x'|^{1+\gamma})}\\
&\leq|\ptl_n\wdt{\sf u}_{\ell}(x)|\leq\frac{C\left|\varphi^{(\ell)}(x',\varepsilon/2+h_{1}(x'))
-\psi^{(\ell)}(x',-\varepsilon/2+h_2(x'))\right|}{\va+|x'|^{1+\gamma}}.
\end{aligned}
\end{align}
Next, we estimate  $|\nabla{\sf v}_{\ell}|,$ $\ell=1,\cdots,m$.
Let
\begin{align}\label{1.53}
{\sf w}_{\ell}={\sf v}_{\ell}-{\wdt{\sf u}}_{\ell},\quad \ell=1,2,\cdots,m,
\end{align}
then, by \eqref{vl} and \eqref{ul} we have ${\sf w}_{\ell}=(w_{\ell}^{(1)}, w_{\ell}^{(2)},\cdots,w_\ell^{(m)})$ satisfies
\begin{align}\label{equ_w}
\begin{aligned}
\left\{
  \begin{array}{ll}
\sum\limits_{\alp,\bt,j}
\left(\partial_{\alpha}\Big(A_{ij}^{\alpha\beta}\partial_{\beta}w_{\ell}^{(j)}
+B_{ij}^{\alpha}w_{\ell}^{(j)}\Big)
+C_{ij}^{\beta}\partial_{\beta}w_{\ell}^{(j)}+D_{ij}w_{\ell}^{(j)}
\right)=H^{(i)}-\sum\limits_{\alp}\partial_{\alpha} F_{i}^{\alpha}, & \,\hbox{in}\,\Omega_{1} \\
   { \sf w}_{\ell}=0, &\,\hbox{on}\,\Gamma_{1}^{\pm},
  \end{array}
\right.
\end{aligned}
\end{align}
where for any $1\leq \alp\leq n,$ $1\leq i\leq m$,
\begin{align}\label{fh}
\begin{aligned}
F_{i}^{(\alpha)}(x):
&=\sum\limits_{\bt,j}\left(A_{ij}^{\alpha\beta}(x)\partial_{\beta}\widetilde{u}_{\ell}^{(j)}(x)
+B_{ij}^{\alpha}(x)\widetilde{u}_{\ell}^{(j)}(x)\right),\\
 H^{(i)}(x):&=\sum\limits_{\bt,j}\left(-C_{ij}^{\beta}\partial_{\beta}\widetilde{u}_{\ell}^{(j)}(x)
 -D_{ij}(x)\widetilde{u}_{\ell}^{(j)}(x)\right).
\end{aligned}
\end{align}

\begin{lemma}\label{lem2}
Let ${\sf v}_{\ell}\in W^{1,2}
(\Omega_{1}\subset\mathbb{R}^{n}; \mathbb{R}^{m}),$ $1\leq\ell\leq m$, be a weak solution of \eqref{vl}
in the sense of \eqref{1.54},
thus ${\sf w}_{\ell}$  as in \eqref{1.53} satisfies \eqref{equ_w}.
Let $\varphi,$ $\psi$ be as in \eqref{1.51}, $\Gamma_{1}^{+}$, $\Gamma_{1}^{-}$
be as in \eqref{boundary} and $\Omega_{r},$ $r\in\{1/2,1\}$ be as in \eqref{Omega}.
Then exists a positive constant $C$ independent of $\va,$ $\ell$,
such that, for any $\ell=1,\cdots,m$,
\begin{align*}
\int_{\Omega_{1/2}}\left|\nabla {\sf w}_{\ell}\right|^{2}\,dx
\leq
 C\left(
 \left\|{\sf w}_{\ell}\right\|_{L^{2}(\Omega_{1})}^{2}
+\|\varphi^{(\ell)}\|^2_{C^{1,\gamma}(\Gamma_{1}^{+})}
+\|\psi^{(\ell)}\|_{C^{1,\gamma}(\Gamma_1^{-})}^2\right).
\end{align*}
\end{lemma}

\begin{proof}
For simplicity, we assume that $\psi\equiv0$.
Multiply \eqref{equ_w} by ${\sf w}_{\ell}$, making use of the integration by parts in $\Omega_{1/2}$,
\begin{align*}
&\sum\limits_{\alp\bt,i,j}
\int_{\Omega_{1/2}}
\left(A_{ij}^{\alpha,\beta}(x)\partial_{\beta}w_{\ell}^{(j)}\right)
\partial_{\alpha}w_{\ell}^{(i)}\,dx\\
=
&\sum\limits_{\alp,\bt,i,j}
\left(-\int_{\Omega_{1/2}}B_{ij}^{\alpha}(x)w_{\ell}^{(j)}\partial_{\alpha}w_{\ell}^{(i)}\,dx
+\int_{\Omega_{1/2}} C_{ij}^{\beta}(x)\partial_{\beta}w_{\ell}^{(j)}w_{\ell}^{(i)}\,dx
\right)
\nonumber\\
\quad
&+\sum\limits_{\alp, \bt,i,j}\left(
\int_{\Omega_{1/2}}D_{ij}(x)w_{\ell}^{(j)}w_{\ell}^{(i)}\,dx
-\int_{\Omega_{1/2}}H^{(i)}w_{\ell}^{(i)}
+\int_{\Omega_{1/2}}\partial_{\alpha}F_{i}^{\alpha}w_{\ell}^{(i)}\,dx\right)\nonumber\\
\quad
&+\sum\limits_{\alp,\bt,i,j}\left(
\int_{\substack{|x'|=1/2,\\-\frac{\varepsilon}{2}+h_{2}(x')<x_{n}<\frac{\varepsilon}{2}+h_{1}(x')}}
\left(A_{ij}^{\alpha\beta}(x)\partial_{\beta}w_{\ell}^{(j)}+B_{ij}^{\alpha}(x)w_{\ell}^{(j)}\right)w_{\ell}^{(i)}\nu^{(\alp)}\,ds
\right),
\end{align*}
where $\vec{\nu}:=(\nu^{(1)},\cdots,\nu^{(n)})\in \mathbb{R}^{n}$ is the  unit outer normal vector of
vertical boundary on both sides of $\Omega_{1/2}$.

From the strong ellipticity condition  \eqref{LHC} and Cauchy's inequality, we have
\begin{align}\label{1.1}
\lambda\int_{\Omega_{1/2}}\left|\nabla {\sf w}_{\ell}\right|^{2}\,dx&\leq
\sum\limits_{\alp,\bt,i,j}\left(
\int_{\Omega_{1/2}}\left(A_{ij}^{\alpha\beta}(x)\partial_{\beta}w_{\ell}^{(j)}\right)\partial_{\alpha}w_{\ell}^{(i)}\,dx
\right)
\nonumber\\
&\leq \frac{\lambda}{4}\int_{\Omega_{1/2}}|\nabla {\sf w}_{\ell}|^2\,dx+C\int_{\Omega_{1/2}}|{\sf w}_{\ell}|^2\,dx
+\sum\limits_{i}\left|\int_{\Omega_{1/2}}H^{(i)}w_{\ell}^{(i)}\,dx\right|\nonumber\\
&\quad+C\int_{\substack{|x'|=1/2,\\-\frac{\varepsilon}{2}+h_{2}(x')<x_{n}<\frac{\varepsilon}{2}+h_{1}(x')}}
\left(|\nabla {\sf w}_{\ell}|^2+|{\sf w}_{\ell}|^2\right)\,dx
+\sum\limits_{\alp,i}\left|\int_{\Omega_{1/2}}\partial_{\alpha} F_{i}^{\alpha}w_{\ell}^{(i)}\,dx\right|.
\end{align}
By using Cauchy's inequality again, \eqref{ul}, and \eqref{fh}, we have
\begin{align}\label{1.2}
\sum\limits_{i}\left|\int_{\Omega_{1/2}}H^{(i)}w_{\ell}^{(i)}\,dx\right|\leq&
C\int_{\Omega_1}\left(\left|\nabla\wdt{\sf u}_{\ell}\right||{\sf w}_{\ell}|+|\wdt{\sf u}_{\ell}||{\sf w}_{\ell}|\right)\,dx\nonumber\\
\leq & C\left(\int_{\Omega_1}(|\nabla\wdt{\sf u}_{\ell}|^2+|\wdt{\sf u}_{\ell}|^2)\,dx
+\int_{\Omega_1}|{\sf w}_{\ell}|^2\,dx\right)\nonumber\\
\leq& C\left( \left\|{\sf w}_{\ell} \right\|^2_{L^{2}(\Omega_{1})}
+\|\varphi^{(\ell)}\|^2_{C^{1,\gamma}(\Gamma_{1}^{+})}
\right).
\end{align}
Note that  $\Omega_{2/3}\backslash\overline{\Omega_{1/3}} \subset\subset\Omega_{1}\backslash\overline{\Omega_{1/4}}$,
 applying Theorem \ref{C1Gamma} and \eqref{1.39} in Theorem \ref{them2} for \eqref{equ_w} with \eqref{fh},  one has
\begin{align*}
\|\nabla {\sf w}_{\ell}\|_{L^{\infty}(\Omega_{2/3}\backslash\overline{\Omega_{1/3}})}
\leq
\|{\sf w}_{\ell}\|_{C^{1,\gamma}(\Omega_{2/3}\backslash\overline{\Omega_{1/3}})}
\leq &
C\left(\|{\sf w}_{\ell}\|_{L^{\infty}(\Omega_{1}\backslash\overline{\Omega_{1/4}})}
+\|\varphi^{(\ell)}\|_{C^{1,\gamma}(\Gamma_{1}^{+})}
\right)\nonumber\\
\leq&
C\left(\|{\sf w}_{\ell}\|_{L^{2}(\Omega_1)}+\|\varphi^{(\ell)}\|_{C^{1,\gamma}(\Gamma_{1}^{+})}
\right).
\end{align*}
This implies that for $x=(x',x_{n})\in \Omega_{2/3}\backslash\overline{\Omega_{1/3}},$
\begin{align*}
  \left|{\sf w}_{\ell}(x',x_n)\right|\leq C\left(\|{\sf w}_{\ell}\|_{L^{2}(\Omega_{1})}
  +\|\varphi^{(\ell)}\|_{C^{1,\gamma}(\Gamma_{1}^{+})}  \right),
\end{align*}
it follows that
\begin{align}\label{1.35}
\int_{\substack{|x'|=1/2,\\-\frac{\varepsilon}{2}+h_{2}(x')<x_{n}<\frac{\varepsilon}{2}+h_{1}(x')}}
\left(|\nabla {\sf w}_{\ell}|^2+|{\sf w}_{\ell}|^2\right)\,dx
\leq C\left(\|{\sf w}_{\ell}\|^2_{L^{2}(\Omega_{1})}+\|\varphi^{(\ell)}\|^2_{C^{1,\gamma}(\Gamma_{1}^{+})}  \right).
\end{align}
By \eqref{par_1}, we deduce
\begin{align}\label{par_n-1}
&\sum\limits_{\alp=1}^{n-1}
\int_{\Omega_{1}}|\ptl_{\alp}\wdt{{\sf u}}_{\ell}|^2\,dx\nonumber\\
\leq&
C\int_{|x'|<1}\delta(x')
\left(\frac{|x'|^{2\gamma}}{(\va+|x'|^{1+\gamma})^{2}}
\left|\varphi^{\ell}(x',\frac{\varepsilon}{2}+h_{1}(x'))\right|
^2+\|\nabla\varphi^{(\ell)}\|^2_{L^{\infty}(\Gamma_1^{+})}
\right)\,dx\nonumber\\
\leq& C\left\|\varphi^{(\ell)}\right\|^2_{C^{1,\gamma}(\Gamma_{1}^{+})}.
\end{align}
In the view of $\partial_{nn}\wdt{\sf u}_{\ell}=0$ and  \eqref{1.35}, \eqref{par_n-1},
by applying the integration by parts in $\Omega_{1/2}$ we have
\begin{align*}
\sum\limits_{\alp,i}\left|\int_{\Omega_{1/2}}\partial_{\alpha} F_{i}^{(\alpha)}w_{\ell}^{(i)}\,dx\right|
\leq&
\sum\limits_{\alp,i}\left|\int_{\Omega_{1/2}} F_{i}^{(\alpha)}\partial_{\alpha}w_{\ell}^{(i)}\,dx\right|
+\left|\int_{\substack{|x'|=1/2,\\-\frac{\varepsilon}{2}+h_{2}(x')<x_{n}<\frac{\varepsilon}{2}+h_{1}(x')}}
F_{i}^{(\alpha)} w_{\ell}^{(i)}\nu^{\alp}\,dx \right|\nonumber\\
\leq&
C\int_{\Omega_{1/2}}
\left(\sum\limits_{\alp=1}^{n-1}\left|\ptl_{\alp}\wdt{\sf u}_{\ell}\right| |\nabla {\sf w}_{\ell}|+|\wdt{\sf u}_{\ell}||\nabla {\sf w}_{\ell}|\right)\,dx\nonumber\\
\quad &
+C \int_{\substack{|x'|=1/2,\\-\frac{\varepsilon}{2}+h_{2}(x')<x_{n}<\frac{\varepsilon}{2}+h_{1}(x')}}
\left( \sum\limits_{\alp=1}^{n-1}\left|\ptl_{\alp}\wdt{\sf u}_{\ell}\right||{\sf w}_{\ell}|+|\wdt{\sf u}_{\ell}||{\sf w}_{\ell}|\right)\,ds\nonumber\\
\leq&
\frac{\lambda}{4}\int_{\Omega_{1/2}}|\nabla {\sf w}_{\ell}|^2\,dx
+C\left(\|\varphi^{(\ell)}\|^2_{C^{1,\gamma}(\Gamma_{1}^{+})}+ \|{\sf w}_{\ell}\|^2_{L^{2}(\Omega_{1})}\right)
,
\end{align*}
this, combining with  \eqref{1.1}, \eqref{1.2}, and \eqref{1.35} we obtain
\begin{align*}
\int_{\Omega_{1/2}}\left|\nabla {\sf w}_{\ell}\right|^{2}\,dx
\leq C
\left(\left\|{\sf w}_{\ell}\right\|_{L^{2}(\Omega_{1})}^{2}
+
\|\varphi^{(\ell)}\|^2_{C^{1,\gamma}(\Gamma_{1}^{+})}
\right).
\end{align*}
We have completed the proof of Lemma \ref{lem2}.
\end{proof}
Let $1\leq s\leq 1/2$ and $\Omega_{1/2}$ be as in \eqref{Omega}.
Let $h_1, h_2\in C^{1,\gamma}(B'_{1})$ satisfy \eqref{1.42}-\eqref{1.50}
and $\va$ be as in \eqref{vare}, define set as
\begin{align}\label{hatomega}
\widehat{\Omega}_{s}(z):=\{(x',x_{n})\in \Omega_{1/2}\,|-\frac{\varepsilon}{2}+h_{2}(x')<x_{n}<\frac{\varepsilon}{2}+h_{1}(x'),
\,|x'-z'|<s\}.
\end{align}
In order to iteration as in  \cite{bll,CL} to prove that
the gradients of the auxiliary functions $\wdt{\sf u}_{\ell}$  are the major singular terms of $|\nabla v_{\ell}|$,
we need the following estimates:
namely, for a fixed  point
\begin{align}\label{z}
z=(z',z_{n})\in \Omega_{1/2},
\end{align}
we consider the H$\mathrm{\ddot{o}}$lder semi-norm estimates
of $\nabla\wdt{\sf u}_{\ell}$ in $\widehat{\Omega}_{s}(z)$.
In the following, we always assume that $\varepsilon$ as in  \eqref{vare} and $|z'|$ are sufficiently small.

\begin{proposition}\label{prop1}
Let $h_{1},$ $h_{2}$ satisfy  \eqref{1.42}--\eqref{1.50},
$\Gamma_{1}^{+},$ $\Gamma_{1}^{-}$ be as in \eqref{boundary} and $\varphi,$ $\psi$
be as in \eqref{1.51}.
Let $z=(z',z_{n})$ be as in \eqref{z}, $\delta(z')$ be as in \eqref{delta},
and $\widehat{\Omega}_{s}(z)$ be as in \eqref{hatomega} for $0<s\leq C\delta(z')$.
Then there exists a positive constants $C$
independent of $ \va$, such that, for any $\ell=1,\cdots, m$,
\begin{align}
\begin{aligned}\label{1.55}
\left[\nabla\wdt{\sf u}_{\ell}\right]_{\gamma,\widehat{\Omega}_{s}(z)}
\leq& C\left|\varphi^{(\ell)}(z', \frac{\varepsilon}{2}+h_1(z'))-\psi^{(\ell)}(z', -\frac{\varepsilon}{2}+h_2(z'))
\right|\\
&\quad\left(\delta(z')^{-1-\frac{1}{1+\gamma}}s^{1-\gamma}
+ \delta(z')^{-\gamma-\frac{1}{1+\gamma}} \right)
+C\left(\|\varphi^{(\ell)}\|_{C^{1,\gamma}(\Gamma_1^+)}+\|\psi^{(\ell)}\|_{C^{1,\gamma}(\Gamma_1^-)}\right)\\
&\quad\left(\delta(z')^{-1-\frac{1}{1+\gamma}} s^{2-\gamma}
+\delta(z')^{-1}s^{1-\gamma}
+ \delta(z')^{-\gamma-\frac{1}{1+\gamma}} s
+\delta(z')^{-\gamma}\right).
\end{aligned}
\end{align}
\end{proposition}
\begin{proof}
Since
$
s\leq C\delta(z')$ and $ |z'|\leq C\delta(z')^{\frac{1}{1+\gamma}},
$
 for any $x=(x',x_n)\in \widehat{\Omega}_{s}(z)$,
\begin{align*}
|x'|\leq|x'-z'|+|z'|<s+|z'|\leq C\delta(z')^{\frac{1}{1+\gamma}},
\end{align*}
and combining with  mean value theorem we have
\begin{align}\label{diff_hi}
 |h_i(x')-h_i(\tilde{x}')|\leq C|x'_{\theta_i}|^{\gamma}|x'-\tilde{x}'|\leq
C \delta(z')^{\frac{\gamma}{1+\gamma}}|x'-\tilde{x}'|,\quad i=1,2,
\end{align}
for any $x',\tilde{x}'\in\widehat{\Omega}_{s}(z),$ $x'\neq\tilde{x}' $.
In view of \eqref{delta}, we obtain
\begin{align}\label{diff_delta_}
|\delta(x')-\delta(\tilde{x}')|\leq |h_1(x')-h_1(\tilde{x}')|+|h_2(x')-h_2(\tilde{x}')|\leq
C\delta(z')^{\frac{\gamma}{1+\gamma}}|x'-\tilde{x}'|.
\end{align}
In particular, taking $\tilde{x}'=z'$, and recalling that $|x'-z'|< s\leq C\delta(z')$, we have
\begin{align} \label{upper_delta}
\delta(x')\leq \delta(z')+|h_1(x')-h_1(z')|+|h_2(x')-h_2(z')|\leq C\delta(z'),
\end{align}
and for sufficiently small $\varepsilon$ and $|z'|$,
\begin{align}\label{low_delta}
\delta(x')\geq\delta(z')-|h_1(x')-h_1(z')|-|h_2(x')-h_2(z')|\geq \frac{1}{2}\delta(z').
\end{align}

Next we estimate $\left|\partial_{\alpha}\bar{u}(x)
-\partial_{\alpha}\bar{u}(\tilde{x})\right|,$ $x,\tilde{x}\in\widehat{\Omega}_{s}(z),$ $x\neq\tilde{x}$, for $ \alpha=1,2,\cdots,n$.
For $\alpha=1,2,\cdots,n-1$, recalling  \eqref{1.52}, we have
\begin{align}
\partial_{\alpha}\bar{u}(x',x_n)&=\frac{-\partial_{\alpha}h_2(x')}{\delta(x')}+\frac{-x_n}{\delta^2(x')}+
\frac{(h_2(x')-{\varepsilon/2})\partial_{\alpha}\delta(x')}{\delta^{2}(x')}\nonumber\\
&=:\Pi_1(x)+\Pi_2(x)+\Pi_3(x).\label{chi1}
\end{align}
Since $h_1(x'),$ $h_2(x')\in C^{1,\gamma}(B'_1)$, \eqref{diff_delta_}, and \eqref{low_delta}, we obtain
\begin{align}\label{chi2}
|\Pi_{1}(x)-\Pi_{1}(\tilde{x})|&
\leq\frac{|\partial_{\alpha}h_2(x')-\partial_{\alpha}h_2(\tilde{x}'|)}{\delta(x')}
+\left|\partial_{\alpha}h_2(\tilde{x}')\right|\left|\frac{1}{\delta(\tilde{x}')}-\frac{1}{\delta(x')}\right|\nonumber\\
&\leq C\left(\delta(z')^{-1}|x'-\tilde{x}'|^{\gamma}+\delta(z')^{-\frac{2}{1+\gamma}} |x'-\tilde{x}'|\right),
\end{align}
for any $x,\tilde{x}\in\widehat{\Omega}_{s}(z),$ $x\neq\tilde{x}$.
By using $|\tilde{x}_n|\leq \delta(z')$, we can obtain
\begin{align}\label{chi3}
|\Pi_{2}(x)-\Pi_2(\tilde{x})|&\leq \frac{|\partial_{\alpha}\delta(x')|}{\delta^{2}(x')}|x_n-\tilde{x}_n|
+\frac{|\tilde{x}_n|}{\delta^{2}(x')}|\partial_{\alpha}\delta(x')-\partial_{\alpha}\delta(\tilde{x}')|\nonumber\\
&\leq C\left(\delta(z')^{-1-\frac{1}{1+\gamma}}|x_n-\tilde{x}_{n}|+\delta(z')^{-1}|x'-\tilde{x}'|^{\gamma}\right).
\end{align}
It follows from \eqref{diff_hi}, \eqref{upper_delta}, and \eqref{low_delta} that
\begin{align*}
|\Pi_{3}(x)-\Pi_{3}(\tilde{x})|\leq&
\frac{|\partial_{\alpha}\delta(x')|}{\delta(x')^2}|h_{2}(x')-h_{2}(\tilde{x}')|
+\frac{\delta(\tilde{x}')}{\delta(x')^{2}}|\partial_{\alpha}\delta(x')
-\partial_{\alpha}\delta(\tilde{x}')|\nonumber\\
&+\delta(\tilde{x}')|\partial_{\alpha}\delta(\tilde{x}')|\left|\frac{1}{\delta(x')^{2}}-
\frac{1}{\delta(\tilde{x}')^2}\right|\nonumber\\
&\leq C\left(\delta(z')^{-\frac{2}{1+\gamma}}|x'-\tilde{x}'|+\delta(z')^{-1}|x'-\tilde{x}'|^{\gamma}\right).
\end{align*}
This, combining with \eqref{chi1}, \eqref{chi2}, and \eqref{chi3}, which implies
that for $\alpha=1,2,\cdots,n-1,$
\begin{align}\label{1.7}
&\left|\partial_{\alpha}\bar{u}(x)-\partial_{\alpha}\bar{u}(\tilde{x})\right|\nonumber\\
\leq & C\left(\delta(z')^{-\frac{2}{1+\gamma}} |x'-\tilde{x}'|+
 \delta(z')^{-1}|x'-\tilde{x}'|^{\gamma}
+ \delta(z')^{-1-\frac{1}{1+\gamma}}|x_n-\tilde{x}_{n}| \right),
\end{align}
and for $\alpha=n$, by using \eqref{diff_delta_} and \eqref{low_delta}, we obtain
\begin{align}\label{par_nbaru}
\left|\partial_{n}\bar{u}(x)-\partial_{n}\bar{u}(\tilde{x})\right|\leq C\left|\frac{1}{\delta(x')}-\frac{1}{\delta(\tilde{x}')}\right|
\leq C\delta(z')^{-1-\frac{1}{1+\gamma}}|x'-\tilde{x}'|.
\end{align}

Now we prove \eqref{1.55}. Take the case when $\psi\equiv0$ for instance.
Firstly, for $\alpha=n$, since $\varphi(x)\in C^{1, \gamma}(\Gamma_1^+)$, \eqref{baru}, \eqref{low_delta}, and \eqref{par_nbaru}, we have
\begin{align*}
\left| \partial_{n}{\wdt{\sf u}}_{\ell}(x)-\partial_n{\wdt{\sf u}}_{\ell}(\tilde{x})\right|
\leq&
\left|\varphi^{(\ell)}(x', \frac{\varepsilon}{2}+h_1(x'))\right|\left| \partial_{n}\bar{u}(x)-\partial_{n}\bar{u}(\tilde{x}) \right|\nonumber\\
&+|\partial_{n}\bar{u}(\tilde{x})|\left|\varphi^{(\ell)}(x', \frac{\varepsilon}{2}+h_1(x'))-\varphi^{(\ell)}(\tilde{x}',\frac{\varepsilon}{2}+h_1(\tilde{x}')\right|
\nonumber\\
\leq&
C\left|\varphi^{(\ell)}(z', \frac{\varepsilon}{2}+h_1(z'))\right|\delta(z')^{-1-\frac{1}{1+\gamma}}|x'-\tilde{x}'|\nonumber\\
&+C\|\nabla\varphi^{(\ell)}\|_{L^{\infty}(\Gamma_1^+)}\delta(z')^{-1}|x'-\tilde{x}'| \left(
\delta(z')^{-\frac{1}{1+\gamma}}s
+1\right),
\end{align*}
where we used the fact that $|x'-z'|<s$. Similarly, by using $|x'-\tilde{x}'|\leq s$ we obtain
\begin{align}\label{1.10}
[\partial_{n}\wdt{\sf u}_{\ell}]_{\gamma,\widehat{\Omega}_{s}(z')}
\leq &C\left|\varphi^{(\ell)}(z', \frac{\varepsilon}{2}+h_1(z'))\right|
 \delta(z')^{-1-\frac{1}{1+\gamma}} s^{1-\gamma}\nonumber\\
\quad
&+C\|\nabla\varphi^{(\ell)}\|_{L^\infty(\Gamma_1^+)}
(\delta(z')^{-1-\frac{1}{1+\gamma}}s^{2-\gamma}
+\delta(z')^{-1}s^{1-\gamma}).
\end{align}
For $\alpha=1,2,\cdots,n-1$,
\begin{align}\label{pardecom}
\left| \partial_{\alpha}{\wdt{\sf u}}_{\ell}(x)-\partial_{\alpha}{\wdt{\sf u}}_{\ell}(\tilde{x})\right|
\leq &\left|\partial_{\alpha}\varphi^{(\ell)}(x', \frac{\varepsilon}{2}+h_1(x'))\bar{u}(x)-\partial_{\alpha}\varphi^{(\ell)}(\tilde{x}', \frac{\varepsilon}{2}+h_1(\tilde{x}'))\bar{u}(\tilde{x})\right|\nonumber\\
&+\left| \varphi^{(\ell)}(x', \frac{\varepsilon}{2}+h_1(x'))\partial_{\alpha}\bar{u}(x)- \varphi^{(\ell)}(x', \frac{\varepsilon}{2}+h_1(x'))\partial_{\alpha}\bar{u}(\tilde{x})  \right|\nonumber\\
=&:\mathrm{I}+\mathrm{II}.
\end{align}
Using \eqref{low_delta} and \eqref{par_nbaru}, we obtain
\begin{align*}
\mathrm{I}&\leq C \|\nabla\varphi^{(\ell)} \|_{L^{\infty}(\Gamma_1^{+})} \left( \frac{|x_{n}-\tilde{x}_{n}|}{\delta(x')}
+|\tilde{x}_{n}||\frac{1}{\delta(x')}-\frac{1}{\delta(\tilde{x}')}|\right)
+C\|\nabla\varphi^{(\ell)}\|_{C^{0,\gamma}(\Gamma_1^+)}|x'-\tilde{x}'|^{\gamma}\nonumber\\
&\leq C\|\nabla\varphi^{(\ell)}\|_{C^{0,\gamma}(\Gamma_1^+)}\left(\delta(z')^{-1}|x_{n}-\tilde{x}_{n}|
+\delta(z')^{-\frac{1}{1+\gamma}}|x'-\tilde{x}'|+|x'-\tilde{x}'|^{\gamma}\right).
\end{align*}
In view of $|x'-\tilde{x}'|<s$ and $|x_{n}-\tilde{x}_{n}|\leq 2\delta(z')$, we obtain
\begin{align}\label{1.8}
\sup\limits_{x,\tilde{x}\in\widehat{\Omega}_{s}(z'),x\neq\tilde{x}}\frac{\mathrm{I}}{|x-\tilde{x}|^{\gamma}}
\leq C\left\|\nabla\varphi^{(\ell)}\right\|_{C^{0,\gamma}(\Gamma_1^+)}
\left(\delta(z')^{-\gamma}+\delta(z')^{-\frac{1}{1+\gamma}}s^{1-\gamma}\right).
\end{align}
It follows from \eqref{baru} and \eqref{1.7} that
\begin{align*}
\mathrm{II}
\leq& C\left(|\varphi^{(\ell)}(z', \frac{\varepsilon}{2}+h_1(z'))|+ s\|\nabla\varphi^{(\ell)}\|_{L^\infty(\Gamma_{1}^+)} \right)\nonumber\\
&\left(\delta(z')^{-\frac{2}{1+\gamma}} |x'-\tilde{x}'|+
 \delta(z')^{-1}|x'-\tilde{x}'|^{\gamma}
+ \delta(z')^{-1-\frac{1}{1+\gamma}}|x_n-\tilde{x}_{n}| \right)\nonumber\\
&+C\|\nabla\varphi^{(\ell)}\|_{L^\infty(\Gamma_1^+)}\delta(z')^{-1+\frac{\gamma}{1+\gamma}}
|x'-\tilde{x}'|.
\end{align*}
Obviously,
\begin{align*}
\sup\limits_{x, \tilde{x}\in\widehat{\Omega}_{s}(z'),x\neq\tilde{x}}&\frac{\mathrm{II}}{|x-\tilde{x}|^{\gamma}}
\leq
C\left|\varphi^{(\ell)}(z', \frac{\varepsilon}{2}+h_1(z'))\right|
\left(\delta(z')^{-\frac{2}{1+\gamma}} s^{1-\gamma}
+ \delta(z')^{-\gamma-\frac{1}{1+\gamma}} \right)\nonumber\\
&
+C\|\nabla\varphi^{(\ell)}\|_{L^\infty(\Gamma_1^+)}
\left(
\delta(z')^{-\frac{2}{1+\gamma}} s^{2-\gamma}
+\delta(z')^{-\frac{1}{1+\gamma}} s^{1-\gamma}
+ \delta(z')^{-\gamma-\frac{1}{1+\gamma}} s
\right)
.
\end{align*}
This, combining with \eqref{pardecom} and \eqref{1.8}, we have for $\alpha=1,2,\cdots, n-1$,
\begin{align*}
[\partial_{\alpha}{\wdt{\sf u}}_{\ell}]&_{\gamma,\widehat{\Omega}_{s}(z')}
\leq
C\left|\varphi^{(\ell)}(z', \frac{\varepsilon}{2}+h_1(z'))\right|
\left(\delta(z')^{-\frac{2}{1+\gamma}} s^{1-\gamma}
+ \delta(z')^{-\gamma-\frac{1}{1+\gamma}}
\right)
\nonumber\\
&
+C\|\varphi^{(\ell)}\|_{C^{1,\gamma}(\Gamma_1^+)}
\left(\delta(z')^{-\frac{2}{1+\gamma}} s^{2-\gamma}
+\delta(z')^{-\frac{1}{1+\gamma}}s^{1-\gamma}
+ \delta(z')^{-\gamma-\frac{1}{1+\gamma}} s
+\delta(z')^{-\gamma}
\right)
.
\end{align*}
From \eqref{1.10} and above formula, we have proved Proposition \ref{prop1} immediately.
\end{proof}
Let $z=(z',z_{n})$ be as in \eqref{z}, $\delta(z')$ be as in \eqref{delta},
and $\widehat{\Omega}_{s}(z)$ be as in \eqref{hatomega} for $0<s\leq C\delta(z')$.
Let $A_{ij}^{\alp\bt}$ be as in Definition \ref{defn1},
$\wdt{\sf u}_{\ell}$ be as in \eqref{ul} for $\ell=1,\cdots, m$,
 and $|\widehat{\Omega}_{s}(z)|$
be the volume of region $\widehat{\Omega}_{s}(z)$. For
$i=1, 2, \cdots, m$ and $\alpha=1,\cdots,n$, we define
\begin{align}\label{1.12}
\mathcal{M}_{i}^{(\alpha)}:=(\mathfrak{a}_{i}^{\alpha})_{j\beta}=
\Big(\frac{1}{|\widehat{\Omega}_{s}(z)|}\int_{\widehat{\Omega}_{s}(z)}
\sum\limits_{\bt,j}
\big(A_{ij}^{\alpha\beta}(y)\partial_{\beta}\tilde{u}_{\ell}^{(j)}(y)
\big)\,dy\Big)_{j\beta}.
\end{align}
Clearly, by using \eqref{equ_w} we have ${\sf w}_{\ell},$ $\ell=1,\cdots, m$ as in \eqref{1.53} satisfies
\begin{align}
\begin{aligned}\label{Matrix}
\sum\limits_{\alp,\bt,i,j}
\partial_{\alpha}\big(A_{ij}^{\alpha\beta}\partial_{\beta}w_{\ell}^{(j)}+B_{ij}^{\alpha}w_{\ell}^{(j)}\big)
+C_{ij}^{\beta}\partial_{\beta}w_{\ell}^{(j)}+&D_{ij}w_{\ell}^{(j)}\\
=&H^{(i)}-\sum\limits_{\alp}\partial_{\alpha} (F_{i}^{(\alpha)}-\mathcal{M}_{i}^{(\alpha)}),~~\hbox{in}~~\Omega_{1}.
\end{aligned}
\end{align}

\begin{lemma}\label{lemmabddlocal}
Let $z=(z',z_{n})\in \Omega_{1/2}$ as in \eqref{Omega}, $\delta(z')$ be as in \eqref{delta}
and $\widehat{\Omega}_{\delta(z')}(z)$ be as in \eqref{hatomega}.
Let $h_{1},$ $h_{2}$ satisfy  \eqref{1.42}-\eqref{1.50},
$\Gamma_{1}^{+},$ $\Gamma_{1}^{-}$ be as in \eqref{boundary} and $\varphi,$ $\psi$
be as in \eqref{1.51}.
Let $\va$ be as in \eqref{vare} and ${\sf w}_{\ell}$ be as in
\eqref{Matrix}, $\ell=1,\cdots, m$. Then there exists a positive constant $C$ independent of
$\va,\ell$, such that,
for $0\leq|z'|\leq \va^{\frac{1}{1+\gamma}}$,
\begin{align}\label{1.46}
\int_{\widehat{\Omega}_{\delta(z')}(z)}\left|\nabla {\sf w}_{l}\right|^{2}\ dx
\leq&
C\va^{n-\frac{2}{1+\gamma}}
\left(
\left|\varphi^{(\ell)}(z', \frac{\varepsilon}{2}+h_1(z'))
-\psi^{(\ell)}(z',-\frac{\va}{2}+h_2(z'))\right|^2\right)\nonumber\\
    \quad&+C\va^{n-\frac{2}{1+\gamma}+2}
    \left(
    \left\|{\sf w}_{\ell}\right\|^2_{L^{2}(\Omega_1)}
    +\left\|\varphi^{(\ell)}\right\|_{C^{1,\gamma}(\Gamma_{1}^{+})}^2
+\left\|\psi^{(\ell
)}\right\|_{C^{1,\gamma}(\Gamma_{1}^{-})}^2\right),
\end{align}
and for $\va^{\frac{1}{1+\gamma}}<|z'|\leq 1/2$,
\begin{align}\label{1.47}
\int_{\widehat{\Omega}_{\delta(z')}(z)}\left|\nabla {\sf w}_{\ell}\right|^{2}\ dx
\leq&
C|z'|^{(1+\gamma)(n-\frac{2}{1+\gamma})}\left(\left|\varphi^{(\ell)}(z', \frac{\varepsilon}{2}+h_1(z'))
-\psi^{(\ell)}(z',-\frac{\va}{2}+h_2(z'))\right|^2\right)\nonumber\\
\quad&
+C|z'|^{(1+\gamma)(n-\frac{2}{1+\gamma}+2)}
\left(
\left\|{\sf w}_{\ell}\right\|^2_{L^{2}(\Omega_1)}
+\left\|\varphi^{(\ell)}\right\|_{C^{1,\gamma}(\Gamma_{1}^{+})}^2
+\left\|\psi^{(\ell)}\right\|_{C^{1,\gamma}(\Gamma_{1}^{-})}^2
\right).
\end{align}
\end{lemma}
\begin{proof}
For simplicity, we assume that $\psi\equiv0$.
Indeed, for $0 < t < s < 1/2$, let $\eta$ be a cut-off function satisfying $0\leq\eta(x')\leq 1$,
\begin{align*}
\eta(x')=
\left\{
  \begin{array}{ll}
    1, & \hbox{if}~|x'-z'|<t  \\
    0, & \hbox{if}~|x'-z'|>s
  \end{array}
\right. ,\quad |\nabla\eta(x')|\leq \frac{2}{s-t}.
\end{align*}
Multiplying \eqref{Matrix} by $\eta^{2}{\sf w}_{\ell}$ and using the integration by parts, one has
\begin{align}
\begin{aligned}\label{ibp}
&\sum\limits_{\alp,\bt,i,j}\int_{\widehat{\Omega}_{s}(z)}\eta^2 A_{ij}^{\alpha\beta}(x)\partial_{\beta}w_{\ell}^{(j)}\partial_{\alpha}w_{\ell}^{(i)}\,dx\\
=&-\sum\limits_{\alp,\bt,i,j}\int_{\widehat{\Omega}_{s}(z)} A_{ij}^{\alpha\beta}(x) \partial_{\beta}w_{\ell}^{(j)} w_{\ell}^{(i)} 2\eta\partial_{\alpha}\eta\,dx
-\sum\limits_{\alp,i,j}\int_{\widehat{\Omega}_{s}(z)} B_{ij}^{\alpha}(x)w_{\ell}^{(j)}\partial_{\alpha}(\eta^2 w_{\ell}^{(i)})\\
&-\sum\limits_{\bt,i,j}\int_{\widehat{\Omega}_{s}(z)}C_{ij}^{\beta}(x)\partial_{\beta}w_{\ell}^{(j)}(\eta^{2}w_{\ell}^{(i)})
-\sum\limits_{i,j}\int_{\widehat{\Omega}_{s}(z)}D_{ij}(x)w_{\ell}^{(j)}(\eta^2 w_{\ell}^{(i)})\,dx\\
&+\sum\limits_{i}\int_{\widehat{\Omega}_{s}(z)}H^{(i)}(\eta^2w_{\ell}^{(i)})\,dx
+\sum\limits_{\alp,i}\int_{\widehat{\Omega}_{s}(z)} (F_{i}^{(\alpha)}-\mathcal{M}_{i}^{(\alpha)})\partial_{\alpha}(\eta^2 w_{\ell}^{(i)})
\,dx.
\end{aligned}
\end{align}
Now we can bound with \eqref{AUP}, Young's inequality $2ab\leq\zeta a^2+\frac{b^2}{\zeta},$ and the properties of $\eta$,
\begin{align}\label{A1}
&\sum\limits_{\alp,\bt,i,j}\left|  \int_{\widehat{\Omega}_{s}(z)}
 A_{ij}^{\alpha\beta}(x) \partial_{\beta}w_{\ell}^{(j)}
 w_{\ell}^{(i)} 2\eta\partial_{\alpha}\eta\,dx\right|\nonumber\\
\quad\leq&
\zeta_{1}\Lambda\int_{\widehat{\Omega}_{s}(z)}\eta^2|\nabla {\sf w}_{\ell}|^2\,\ud x
+\frac{4\Lambda}{\zeta_1(s-t)^2}\int_{\widehat{\Omega}_{s}(z)}|{\sf w}_{\ell}|^2\,\ud x,
\end{align}
combining with \eqref{BCD}, we deduce
\begin{align}\label{B1}
\sum\limits_{\alp,i,j}\left| \int_{\widehat{\Omega}_{s}(z)} B_{ij}^{\alpha}(x)w_{\ell}^{(j)}\partial_{\alpha}(\eta^2 w_{\ell}^{(i)})\,dx\right|
=&\sum\limits_{\alp,i,j}\left| \int_{\widehat{\Omega}_{s}(z)} B_{ij}^{\alpha}(x)w_{\ell}^{(j)}(2\eta\partial_{\alpha}\eta w_{\ell}^{(i)}+\eta^2\ptl_{\alp}w_{\ell}^{(i)})\,dx\right|
\nonumber\\
\leq& \frac{\zeta_2\kappa_{3}}{2}\int_{\widehat{\Omega}_{s}(z)}\eta^2|\nabla {\sf w}_{\ell}|^2\,dx
+\frac{C}{\zeta_{2}(s-t)^2}\int_{\widehat{\Omega}_{s}(z)}|{\sf w}_{\ell}|^2\,dx.
\end{align}
Similarly, using Young's inequality again, we have
\begin{align}\label{C1}
\sum\limits_{\bt,i,j}\left|
\int_{\widehat{\Omega}_{s}(z)}C_{ij}^{\beta}(x)\partial_{\beta}w_{\ell}^{(j)}\eta^{2}w_{\ell}^{(i)}\,dx  \right|
\leq&\frac{\zeta_{3}\kappa_{3}}{2}\int_{\widehat{\Omega}_{s}(z)}\eta^2|\nabla {\sf w}_{\ell}|^2\,dx+\frac{C}{\zeta_{3}(s-t)^2}
\int_{\widehat{\Omega}_{s}(z)}|{\sf w}_{\ell}|^2\,dx.
\end{align}
Noticing $0<s-t<1,$ we obtain
\begin{align}\label{D1}
\sum\limits_{i,j}\left|\int_{\widehat{\Omega}_{s}(z)}D_{ij}(x)w_{\ell}^{(j)}\eta^2 w_{\ell}^{(i)}\,dx\right|\leq\frac{C}{(s-t)^2}\int_{\widehat{\Omega}_{s}(z)}|{\sf w}_{\ell}|^2\,dx.
\end{align}
Recalling \eqref{ul} and \eqref{fh}, we also have
\begin{align}\label{h1}
\sum\limits_{i}\left|\int_{\widehat{\Omega}_{s}(z)}H^{(i)}\eta^2w_{\ell}^{(i)}\,dx\right|
\leq
\frac{C}{(s-t)^2}\int_{\widehat{\Omega}_{s}(z)}|{\sf w}_{\ell}|^2\,dx
+(s-t)^2\sum\limits_{i}\int_{\widehat{\Omega}_{s}(z)}|H^{(i)}|^2\,dx.
\end{align}
For the last term in the right hand side of \eqref{ibp}, using the Young's inequality again we obtain
\begin{align}\label{f1}
&\sum\limits_{\alp,i}\left|\int_{\widehat{\Omega}_{s}(z)}
(F_{i}^{(\alpha)}-\mathcal{M}_{i}^{(\alpha)})\partial_{\alpha}(\eta^2 w_{\ell}^{(i)})
\,dx  \right|\nonumber\\
\leq&\frac{\zeta_{4}}{2}\int_{\widehat{\Omega}_{s}(z)}\eta^{2}|\nabla {\sf w}_{\ell}|^2\,dx
+\frac{C}{\zeta_{4}}\sum\limits_{\alp,i}\int_{\widehat{\Omega}_{s}(z)} |F_{i}^{(\alpha)}-\mathcal{M}_{i}^{(\alpha)}|^2\,dx
+\frac{C}{\zeta_{4}(s-t)^2}\int_{\widehat{\Omega}_{s}(z)}|{\sf w}_{\ell}|^2\,dx.
\end{align}
Choosing $0<\zeta_{i}<1, i=1,2,3,4$ satisfy
\begin{align*}
\Lambda\zeta_{1}+\frac{\kappa_{3}}{2}\zeta_2+\frac{\kappa_{3}}{2}\zeta_{3}+\frac{1}{2}\zeta_{4}<\frac{\lambda}{2}.
\end{align*}
In view of the strong ellipticity condition \eqref{LHC},  we obtain
\begin{align*}
\lambda\int_{\widehat{\Omega}_{s}(z)}\eta^{2}|\nabla {\sf w}_{\ell}|^{2}\,dx
\leq&
\sum\limits_{\alp,\bt,i,j}
\int_{\widehat{\Omega}_{s}(z)}\eta^{2} A_{ij}^{\alpha\beta}(x)\partial_{\beta} w_{\ell}^{(j)}\partial_{\alpha}w_{\ell}^{(i)}\,dx,
\end{align*}
this, combining with \eqref{ibp}-\eqref{f1}, yields
\begin{align}\label{1.4}
\int_{\widehat{\Omega}_{t}(z)} |\nabla {\sf w}_{\ell}|^2\,dx
\leq& \frac{C}{(s-t)^{2}}\int_{\widehat{\Omega}_{s}(z)}|{\sf w}_{\ell}|^2\,dx
+C(s-t)^2\sum\limits_{i}\int_{\widehat{\Omega}_{s}(z)}|H^{(i)}|^2\,dx\nonumber\\
\quad&+C\sum\limits_{\alp,i}\int_{\widehat{\Omega}_{s}(z)} |F_{i}^{(\alpha)}-\mathcal{M}_{i}^{(\alpha)}|^2\,dx.
\end{align}
It follows from \eqref{par_1} and \eqref{par_n} that
\begin{align}\label{1.44}
\sum\limits_{i}\int_{\widehat{\Omega}_{s}(z)}|H^{(i)}|^2\,dx
\leq&
 C\left|\varphi^{(\ell)}(z',\varepsilon/2+h_{1}(z'))\right|^2 \int_{|x'-z'|<s}
 \left(\frac{1}{\va+|x'|^{1+\gamma}}\right)\,dx'\nonumber\\
\quad&+ \|\nabla\varphi^{(\ell)}\|_{L^{\infty}(\Gamma_1^+)}^2 \int_{|x'-z'|<s}
\left(\frac{|x'-z'|^2}{\va+|x'|^{1+\gamma}}\right)\,dx'.
\end{align}
\textbf{Case 1.} For $|z'|\leq \varepsilon^{\frac{1}{1+\gamma}}$ and $0< s<\varepsilon^{\frac{1}{1+\gamma}}$, we have
$\varepsilon\leq \delta(z')\leq C\varepsilon$. By a direct calculation, we have
\begin{align}\label{1.5}
\int_{\widehat{\Omega}_{s}(z)}|{\sf w}_{\ell}|^{2}\,dx
=\int_{\widehat{\Omega}_{s}(z)}\left|
\int_{-\frac{\varepsilon}{2}+h_2(x')}^{x_{n}}\partial_{n}w_{\ell}(x',x_{n})\,dx_{n}
\right|^2\,dx
\leq C\varepsilon^2\int_{\widehat{\Omega}_{s}(z)}|\nabla {\sf w}_{\ell}|^2\,dx,
\end{align}
and from \eqref{1.44}, we have
\begin{align*}
\sum\limits_{i}\int_{{\widehat{\Omega}_{s}(z)}}|H^{(i)}|^2\,dx
\leq
 C\left|\varphi^{(\ell)}(z',\frac{\va}{2}+h_{1}(z'))\right|^2 \frac{s^{n-1}}{\va}
 +C\|\nabla\varphi^{(\ell)}\|_{L^{\infty}(\Gamma_1^+)}^2\frac{s^{n+1}}{\va}:=G_{11}(s).
\end{align*}
By using \eqref{BCD}, \eqref{fh}, and \eqref{1.12} for any $\alp=1,\cdots,n,$ $i=1,\cdots, m$, we have
\begin{align*}
\left|F_{i}^{(\alpha)}-\mathcal{M}_{i}^{(\alpha)}\right|^2
\leq&\frac{C\kappa_{3}}{|\widehat{\Omega}_{s}(z)|^2}
\left([\nabla \wdt{\sf u}_{\ell}]_{\gamma;\widehat{\Omega}_{s}(z)}
\int_{\widehat{\Omega}_{s}(z)}
|x-y|^{\gamma}\,dy
+\delta(z')^{-1}\int_{\widehat{\Omega}_{s}(z)} |x-y|^{\gamma}\,dy
\right)^2+C\kappa_{3}\|\varphi^{(\ell)}\|_{L^{\infty}(\Gamma_{1}^{+})}
 \nonumber\\
\leq& C  \left( [\nabla \wdt{\sf u}_{\ell}]^2_{\gamma;\widehat{\Omega}_{s}(z)}
 +\delta(z')^{-2}\right)\left(\delta(z')^{2\gamma}+s^{2\gamma}\right),
 \end{align*}
thus, from Proposition \ref{prop1} we have
\begin{align}\label{1.6}
\sum\limits_{\alp,i}
\int_{\widehat{\Omega}_{s}(z)}
\left|F_{i}^{\alpha}-\mathcal{M}_{i}^{\alpha}\right|^2\,dx
\leq&
 C\left(\left|\varphi^{(\ell)}(z', \frac{\varepsilon}{2}+h_1(z'))\right|^2
+s^2\|\varphi^{(\ell)}\|_{C^{1,\gamma}(\Gamma_{1}^{+})}^2\right)\nonumber\\
&\quad \left(
\frac{s^{n+1}}{\varepsilon^{1+\frac{2}{1+\gamma}}}
+\frac{s^{n-1}}{\varepsilon^{\frac{2}{1+\gamma}-1}}
+\frac{s^{n+1-2\gamma}}{\varepsilon^{1+\frac{2}{1+\gamma}-2\gamma}}
+\frac{s^{n-1+2\gamma}}{\varepsilon^{1+\frac{2\gamma^{2}}{1+\gamma}}}
\right)
=:G_{12}(s).
\end{align}
Denote
$
F(t):=\int_{\widehat{\Omega}_{t}(z_1)}|\nabla{\sf w}_{\ell}|^{2} \ dx.
$
It follows from \eqref{1.4}, \eqref{1.5}, and \eqref{1.6} that
\begin{align}\label{1.13}
F(t)\leq \left(\frac{c_{1}\varepsilon}{s-t}\right)^2 F(s)+C(s-t)^2G_{11}(s)+CG_{12}(s).
\end{align}
Similarly as in \cite{CL}, let $k=(4c_{1}\varepsilon^{\frac{\gamma}{1+\gamma}})^{-1}$ and $t_{\tau}
=\delta(z')+2c_{1}\tau\varepsilon$, $\tau=0, 1, 2, \cdots, k$. It is easy to see from the definition of $G_{11}(s)$ and $G_{12}(s)$ that
\begin{align*}
G_{11}(t_{\tau+1})\leq
C\va^{n-2}\left(
\left|\varphi^{(\ell)}(z',\varepsilon/2+h_{1}(z'))\right|^2
+\va^{2}\|\nabla\varphi^{(\ell)}\|_{L^{\infty}(\Gamma_1^+)}^2
\right)(\tau+1)^{n+1},
\end{align*}
and
\begin{align*}
G_{12}(t_{\tau+1})\leq C\va^{n-\frac{2}{1+\gamma}}\left(\left|\varphi^{(\ell)}(z', \frac{\varepsilon}{2}+h_1(z'))\right|^2
+\va^{2}\left\|\varphi^{(\ell)}\right\|_{C^{1,\gamma}(\Gamma_{1}^+)}^2\right)(\tau+1)^{n+3} .
\end{align*}
Taking $s=t_{i+1}$ and $t=t_{i}$ in \eqref{1.13}, we have the following iteration formula
\begin{align*}
F(t_{\tau})\leq&\frac{1}{4}F(t_{\tau+1})
+C\va^{n}\left(
\left|\varphi^{(\ell)}(z',\varepsilon/2+h_{1}(z'))\right|^2
+\va^{2}\|\nabla\varphi^{(\ell)}\|_{L^{\infty}(\Gamma_1^+)}^2
\right)(\tau+1)^{n+1}
\nonumber\\
&\quad+C\va^{n-\frac{2}{1+\gamma}}\left(\left|\varphi^{(\ell)}(z', \frac{\varepsilon}{2}+h_1(z'))\right|^2
+\va^{2}\left\|\varphi^{(\ell)}\right\|_{C^{1,\gamma}(\Gamma_{1}^+)}^2
\right)(\tau+1)^{n+3}\\
\leq& C\va^{n-\frac{2}{1+\gamma}}\left(|\varphi^{(\ell)}(z', \frac{\varepsilon}{2}+h_1(z'))|^2
+\va^{2}\left\|\varphi^{(\ell)}\right\|_{C^{1,\gamma}(\Gamma_{1}^+)}^2
\right)(\tau+1)^{n+3},
\end{align*}
after $k$ iterations, and by virtue of Lemma \ref{lem2}, we have
\begin{align*}
F(t_{0})\leq (\frac{1}{4})^{k}F(t_{k})
&+C\va^{n-\frac{2}{1+\gamma}}\left(|\varphi^{(\ell)}(z', \frac{\varepsilon}{2}+h_1(z'))|^2+\va^{2}\left\|\varphi^{(\ell)}\right\|_{C^{1,\gamma}(\Gamma_{1}^+)}^2
\right)
 \sum_{i=0}^{k-1}(\frac{1}{4})^i(i+1)^{n+3} \nonumber\\
\leq& C\va^{n-\frac{2}{1+\gamma}}\left(|\varphi^{(\ell)}(z', \frac{\varepsilon}{2}+h_1(z'))|^2
+\va^{2}(\left\|\varphi^{(\ell)}\right\|_{C^{1,\gamma}(\Gamma_{1}^+)}^2+\left\|{\sf w}_{l}\right\|^2_{L^{2}(\Omega_1)})
\right).
\end{align*}
This implies Lemma \ref{lemmabddlocal} with $|z'|\leq \varepsilon^{\frac{1}{1+\gamma}}$ .

{\bf Case 2.} For $\varepsilon^{\frac{1}{1+\gamma}}\leq|z'|\leq\,\frac{1}{2}$ and $0<s<|z'|$, we have $\frac{1}{C}|z'|^{1+\gamma}\leq\delta(z')\leq\,C|z'|^{1+\gamma}$. Estimates \eqref{1.5} and \eqref{1.6}  become, respectively,
\begin{align*}%\label{1.14}
\int_{\widehat{\Omega}_{s}(z)}|{\sf w}_{\ell}|^{2}\ dx
\leq&\,C|z'|^{2(1+\gamma)}\int_{\widehat{\Omega}_{s}(z')}|\nabla{{\sf w}_{\ell}}|^{2}\ dx,
 \quad\mbox{if}~\,0<s<\frac{2}{3}|z'|,\\
\sum\limits_{i}\int_{{\widehat{\Omega}_{s}(z)}}|H^{(i)}|^2\,dx
\leq&
C\left|\varphi^{(\ell)}(z',\frac{\va}{2}+h_{1}(z'))\right|^2 \frac{s^{n-1}}{|z'|^2}
 +C\|\nabla\varphi^{(\ell)}\|_{L^{\infty}(\Gamma_1^+)}^2\frac{s^{n+1}}{|z'|^2}:=G_{21}(s),
\end{align*}
and
\begin{align*}
\sum\limits_{\alp,i}\int_{\widehat{\Omega}_{s}(z)}
\left|F_{i}^{(\alpha)}-\mathcal{M}_{i}^{(\alpha)}\right|^2\,dx
\leq & C\left(\big|\varphi^{(\ell)}(z', \frac{\varepsilon}{2}+h_1(z'))\big|^2
 +s^2\left\|\varphi^{(\ell)}\right\|_{C^{1,\gamma}(\Gamma_{1}^{+})}^2\right)\nonumber\\
 \quad&\left(
\frac{s^{n+1}}{|z'|^{3+\gamma}}
+\frac{s^{n-1}}{|z'|^{-1-\gamma}}
+\frac{s^{n+1-2\gamma}}{|z'|^{1-\gamma-2\gamma^2}}
+\frac{s^{n-1+2\gamma}}{|z'|^{\gamma-1+2\gamma^2}}
\right)
\nonumber\\
=&:G_{22}(s).
\end{align*}
Let $k=(4c_{2}|z'|^{\gamma})^{-1}$ and $t_{\tau}=\delta(z')+2c_{2}\iota\,|z'|^{1+\gamma}$, $\tau=0,1,2,\cdots,k$, one has
\begin{align*}
G_{21}(s)\leq &
C|z'|^{(1+\gamma)(n-2)}\\
\quad&\left(\left|\varphi^{(\ell)}(z',\varepsilon/2+h_{1}(z'))\right|^2
 +C|z'|^{2(1+\gamma)} \|\nabla\varphi^{(\ell)}\|_{L^{\infty}(\Gamma_1^+)}^2\right)(\tau+1)^{n+1},
\end{align*}
and
\begin{align*}
G_{22}(t_{\tau+1}) &\leq C|z'|^{(1+\gamma)(n-\frac{2}{1+\gamma})}\nonumber\\
&\quad\left(\big|\varphi^{(\ell)}(z', \frac{\varepsilon}{2}+h_1(z'))\big|^2
+C|z'|^{2(1+\gamma)}\left\|\varphi^{(\ell)}\right\|_{C^{1,\gamma}(\Gamma_1^+)}^2\right)(\tau+1)^{n+3} .
\end{align*}
Then, we obtain that, for $0<t<s<\frac{2}{3}|z'|$,
\begin{align*}
F(t_{\tau})&\leq\frac{1}{4}F(t_{\tau+1})
+ C|z'|^{(1+\gamma)(n-\frac{2}{1+\gamma})}\nonumber\\
&\quad\left(\big|\varphi^{(\ell)}(z', \frac{\varepsilon}{2}+h_1(z'))\big|^2
+C|z'|^{2(1+\gamma)}
\left\|\varphi^{(\ell)}\right\|_{C^{1,\gamma}(\Gamma_1^+)}^2\right)(\tau+1)^{n+3},
\end{align*}
after $k$ iterations, and using Lemma \ref{lem2} again, we have
\begin{align*}
F(t_{0})
&\leq C|z'|^{(1+\gamma)(n-\frac{2}{1+\gamma})}\nonumber\\
&\quad\left(\big|\varphi^{(\ell)}(z', \frac{\varepsilon}{2}+h_1(z'))\big|^2
+C|z'|^{2(1+\gamma)}(\left\|\varphi^{(\ell)}\right\|_{C^{1,\gamma}(\Gamma_{1}^+)}^2+\|{\sf w}_{\ell}\|^2_{L^{2}(\Omega_1)})\right).
\end{align*}
This implies Lemma \ref{lemmabddlocal} with $|z'|\geq\varepsilon^{\frac{1}{1+\gamma}}$.
\end{proof}

\begin{lemma}\label{lem1}

Let $z=(z',z_{n})\in \Omega_{1/2}$ as in \eqref{Omega}, $\delta(z')$ be as in \eqref{delta}
and $\widehat{\Omega}_{\delta(z')}(z)$ be as in \eqref{hatomega}.
Let $h_{1},$ $h_{2}$ satisfy  \eqref{1.42}-\eqref{1.50},
$\Gamma_{1}^{+},$ $\Gamma_{1}^{-}$ be as in \eqref{boundary} and $\varphi,$ $\psi$
be as in \eqref{1.51}.
Let $\va$ be as in \eqref{vare} and ${\sf w}_{\ell}$ be as in
\eqref{Matrix}, $\ell=1,\cdots, m$. Then there exists a positive constant $C$ independent of
$\va$, such that, for $ |z'|\leq \va^{\frac{1}{1+\gamma}}$,
\begin{align*}
\left|\nabla {\sf w}_{\ell}(z',z_n)\right|
\leq&C\va^{-\frac{1}{1+\gamma}}\left|\varphi^{(\ell)}(z', \frac{\varepsilon}{2}+h_1(z'))-\psi^{l}(z', -\frac{\varepsilon}{2}+h_2(z'))\right|
\\
&+C\va^{-\frac{\gamma}{1+\gamma}}
\left(
\|{\sf w}_{\ell}\|_{L^2(\Omega_1)}
+
\|\varphi^{(\ell)}\|_{C^{1,\gamma}(\Gamma_1^{+})}
+\|\psi^{(\ell)}\|_{C^{1,\gamma}(\Gamma_1^{-})}
\right),
\end{align*}
and for $\va^{\frac{1}{1+\gamma}}<|z'|<\frac{1}{2}$,
\begin{align*}
\left|\nabla {\sf w}_{\ell}(z',z_n)\right|
\leq& C|z'|^{-1}\left|\varphi^{(\ell)}(z', \frac{\varepsilon}{2}+h_1(z'))-\psi^{(\ell)}(z', -\frac{\varepsilon}{2}+h_2(z'))\right|
  \\
&+C|z'|^{-\gamma}
\left(
\|{\sf w}_{\ell}\|_{L^2(\Omega_1)}
+\|\varphi^{(\ell)}\|_{C^{1,\gamma}(\Gamma_1^{+})}
+\|\psi^{(\ell)}\|_{C^{1,\gamma}(\Gamma_1^{-})}
\right).
\end{align*}
Consequently, by \eqref{par_1} and \eqref{par_n}, we have for sufficiently small $\va$ and $z\in \Omega_{1/2}$,
\begin{align}\label{1.48}
|\nabla {\sf v}_{\ell}(z)|\leq&
\frac{C\left|\varphi^{(\ell)}(z', \frac{\varepsilon}{2}+h_1(z'))-\psi^{(\ell)}(z', -\frac{\varepsilon}{2}+h_2(z'))\right|}
{\va+|z'|^{1+\gamma}}
\nonumber  \\
&+C\left(
\|{\sf v}_{\ell}\|_{L^2(\Omega_1)}
+\|\varphi^{(\ell)}\|_{C^{1,\gamma}(\Gamma_1^{+})}
+\|\psi^{(\ell)}\|_{C^{1,\gamma}(\Gamma_1^{-})}\right).
\end{align}
Moreover, if $\varphi^{(\ell)}(\vec{0}_{n-1},\frac{\va}{2})\ne \psi^{(\ell)}(\vec{0}_{n-1},-\frac{\va}{2})$, then
there exists a positive constant $C$ independent of $\va$, such that,
for any $ z_{n}\in (-\frac{\va}{2},\frac{\va}{2})$,
\begin{align*}
\left|\nabla{\sf v}_{\ell}(\vec{0}_{n-1}, z_{n})\right|
\geq\frac{|\varphi^{(\ell)}(\vec{0}_{n-1},z_n)-\psi^{(\ell)}(\vec{0}_{n-1},z_n)|}{C\va}.
\end{align*}
\end{lemma}
\begin{proof}
For simplicity, we assume that $\psi\equiv0$. Given $z=(z',z_n)\in\Omega_{1/2}$, making the following change of variables on $\widehat{\Omega}_{\delta(z')}(z)$, as in \cite{CL},
\begin{equation*}
\left\{
\begin{aligned}
&x'-z'=\delta(z') y',\\
&x_n=\delta(z') y_n,
\end{aligned}
\right.
\end{equation*}
then $\widehat{\Omega}_{\delta(z')}(z)$ becomes $\mathcal{Q}_{1}$ of nearly unit size, where
\begin{align}
\begin{aligned}\label{LQ}
\mathcal{Q}_r=\Big\{y\in \R^n :& -\frac{\va}{2\dt(z')} +\frac{1}{\dt(z')}h_2(\dt(z') y' +z') \\
&  < y_n < \frac{\va}{2\dt(z')} +\frac{1}{\dt(z')}h_1(\dt(z') y' +z'), \, |y' |< r  \Big\},
\end{aligned}
\end{align}
for $r\leq1$, and the top and bottom boundaries become
\[
\widehat{\Gamma}^+_r:=\left\{ y\in \R^n \, : \,
y_n=\frac{\varepsilon}{2\delta(z')}+\frac{1}{\delta(z')}h_{1}(\delta(z') y'+z'),\quad|y'|<r\right\},\]
and
\[\widehat{\Gamma}^-_r:=\left\{y\in \R^n \, : \,y_n
=-\frac{\va}{2\delta(z')}+\frac{1}{\delta(z')}h_{2}(\delta(z') y'+z'), \quad |y'|<r \right\},\]
respectively. Let
\begin{align*}
\widehat{\sf w}_{\ell}(y',y_n):={\sf w}_{\ell}(\delta(z') y'+z', \delta(z') y_{n}),\,\,
 \widehat{\sf u}_{\ell}(y',y_n):=\widetilde{\sf u}_{\ell}(\delta(z') y'+z', \delta(z')y_{n}),\quad
 \mbox{for}\, (y', y_n)\in \mathcal{Q}_{1}.
\end{align*}
 It follows from \eqref{equ_w} that $\widehat{\sf w}_{\ell}(y)$ satisfies
\begin{align}\label{equ_hatw}
\left\{
  \begin{array}{ll}
\partial_{\alpha}\left(\widehat{A}_{ij}^{\alpha\beta}\partial_{\beta}\widehat{w}_{\ell}^{(j)}
+\widehat{B}_{ij}^{\alpha}\widehat{w}_{\ell}^{(j)}\right)
+\widehat{C}_{ij}^{\beta}\partial_{\beta}\widehat{w}_{\ell}^{(j)}
+\widehat{D}_{ij}\widehat{w}_{\ell}^{(j)}=\widehat{H}^{(i)}-\partial_{\alpha} \widehat{F}_{i}^{\alpha}, & \,\hbox{in}\,\mathcal{Q}_{1}, \\
    \widehat{\sf w}_{\ell}=0, & \,\hbox{on}\,\widehat{\Gamma}_{1}^{\pm},
  \end{array}
\right.
\end{align}
where
\begin{align}
\widehat{A}_{ij}^{\alpha\beta}(y):=A_{ij}^{\alpha\beta}(\delta y'+z', \delta y_{n}),&\quad
\widehat{B}_{ij}^{\alpha}(y):=\delta B_{ij}^{\alpha}(\delta y'+z', \delta y_{n}),\nonumber\\
\widehat{C}_{ij}^{\beta}(y):=\delta C_{ij}^{\beta}(\delta y'+z', \delta y_{n}),&\quad
\widehat{D}_{ij}(y):=\delta^2 D_{ij}(\delta y'+z', \delta y_{n}),\label{1.25}\\
\widehat{F}_{i}^{\alpha}(y):=
 \widehat{A}_{ij}^{\alpha\beta}(y)\partial_{\beta}\widehat{u}_{\ell}^{(j)}
+\widehat{B}_{ij}^{\alpha}(y)\widehat{u}_{\ell}^{(j)},&\quad
\widehat{H}^{(i)}(y):=\widehat{C}_{ij}^{\beta}(y)\partial_{\beta}\widehat{u}_{\ell}^{(j)}
-\widehat{D}_{ij}(y)\widehat{u}_{\ell}^{(j)}.\nonumber
\end{align}
It follows from Theorem \ref{them2} that
\begin{align}\label{1.26}
\|\widehat{\sf w}_{l}\|_{L^{\infty}(\mathcal{Q}_{1/2})}\leq C
\left(\|\widehat{\sf w}_{l}\|_{L^{2}(\mathcal{Q}_{1})}
+ [ \widehat{\sf F}]_{\gamma, \mathcal{Q}_1}+\|\widehat{\sf H}\|_{L^{\infty}(\mathcal{Q}_{1})}
\right).
\end{align}
Applying Theorem \ref{C1Gamma} for \eqref{equ_hatw} with \eqref{1.25} on $\mathcal{Q}_{1/2}$, we have
\begin{align*}
   \|\widehat{\sf w}_{\ell}\|_{C^{1,\gamma}(\mathcal{Q}_{1/4})}\leq C\left( \|\widehat{\sf w}_{\ell}\|_{L^{\infty}(\mathcal{Q}_{1/2})}
   + [ \widehat{\sf F}]_{\gamma, \mathcal{Q}_1}+\|\widehat{\sf H}\|_{L^{\infty}(\mathcal{Q}_{1})}
   \right).
\end{align*}
This, combining with \eqref{1.26} and using the Poincar\'{e} inequality, yields
\begin{align*}
\|\nabla {\widehat{\sf w}}_{\ell}\|_{L^{\infty}(\mathcal{Q}_{1/4})}\leq  C
\left(  \|\nabla{\widehat{\sf w}}_{l\ell}\|_{L^2(\mathcal{Q}_{1})}
+ [ \widehat{\sf F}]_{\gamma, \mathcal{Q}_1}+\|\widehat{\sf H}\|_{L^{\infty}(\mathcal{Q}_{1})}
\right).
\end{align*}
In the following proofs, we will briefly refer to $\delta(z')$ as $\delta$. Recalling back to the original region $\widehat{\om}_{\dt}(z)$,
we have
\begin{align*}
&\|\nabla {\widehat{\sf w}}_{\ell}\|_{L^{\infty}(\mathcal{Q}_{1/4})}
=\delta\|\nabla {{\sf w}}_{\ell}\|_{L^{\infty}(\widehat{\om}_{\dt/4}(z))},
\quad
\|\nabla{\widehat{\sf w}}_{\ell}\|_{L^2(\mathcal{Q}_{1})}
=\dt^{1-\frac{n}{2}}\|\nabla {\sf w}_{\ell}\|_{L^{2}(\widehat{\om}_{\dt}(z))},
\\
&\|\nabla{ \widehat{\sf u}}_{\ell}\|_{L^{\infty}(\mathcal{Q}_{1})}
=\dt\|\nabla{\widetilde{\sf u}}_{\ell}\|_{L^{\infty}(\widehat{\om}_{\dt}(z))},
\quad
[\nabla {\widehat{\sf u}}_{\ell}]_{\gamma, \mathcal{Q}_1}
=\delta^{1+\gamma}[\nabla{\widetilde{\sf u}}_{\ell}]_{\gamma,\widehat{\om}_{\dt}(z)},
\end{align*}
thus,
\begin{align*}
[\widehat{\sf F}]_{\gamma, \mathcal{Q}_1}
\leq C\delta^{1+\gamma}
\left(
\|\nabla\wdt{\sf u}_{\ell}\|_{L^{\infty}(\widehat{\om}_{\dt}(z))}
+[\nabla \wdt{\sf u}_{\ell}]_{\gamma,\widehat{\om}_{\dt}(z)}
+\|\wdt{\sf u}_{\ell}\|_{L^{\infty}(\widehat{\om}_{\dt}(z))}
\right),
\end{align*}
and
\begin{align*}
\|\widehat{\sf H}\|_{L^{\infty}(\mathcal{Q})}\leq C\delta^2
\left(
\|\nabla\wdt{\sf u}_{\ell}\|_{L^{\infty}(\widehat{\om}_{\dt}(z))}
+\|\wdt{\sf u}_{\ell}\|_{L^{\infty}(\widehat{\om}_{\dt}(z))}
\right).
\end{align*}
It follows that
\begin{align*}
&\|\nabla {\sf w}_{\ell}\|_{L^{\infty}(\widehat{\om}_{\dt/4}(z))}\\
\leq&
C\delta^{-\frac{n}{2}}
\|\nabla {\sf w}_{\ell}\|_{L^{2}(\widehat{\om}_{\dt}(z))}
+C\delta^{\gamma}\left(
[\nabla\widetilde{\sf u}_{\ell}]_{\gamma,\widehat{\om}_{\dt}(z)}
+\|\nabla\widetilde{\sf u}_{\ell}\|_{L^{\infty}(\widehat{\om}_{\dt}(z))}
+\|\wdt{\sf u}_{\ell}\|_{L^{\infty}(\widehat{\om}_{\dt}(z))}
\right).
\end{align*}
\textbf{Case 1.} For $0\leq |z'|\leq \va^{\frac{1}{1+\gamma}}$.

By \eqref{1.46} and Proposition \ref{prop1}, we have
\begin{align*}
\delta^{-\frac{n}{2}}\|\nabla{\sf w}\|_{L^{2}(\widehat{\om}_{\dt(z')}(z))}
\leq
\frac{C}{\va^{\frac{1}{1+\gamma}}}
\left|\varphi^{(\ell)}(z', \frac{\varepsilon}{2}+h_1(z'))\right|
+C(\left\|\varphi^{(\ell)}\right\|_{C^{1,\gamma}(\Gamma_{1}^{+})}
+\left\|{\sf w}_{\ell}\right\|_{L^{2}(\Omega_1)}),
\end{align*}
and
\begin{align*}
\delta^{\gamma}[\nabla\widetilde{\sf u}_{\ell}]_{\gamma,\widehat{\om}_{\dt}(z)}\leq
\frac{C}{\va^{\frac{1}{1+\gamma}}}
\left|\varphi^{(\ell)}(z', \varepsilon/2+h_1(z'))\right|
+C\left\|\varphi^{(\ell)}\right\|_{C^{1,\gamma}(\Gamma_{1}^{+})}
.
\end{align*}
By using \eqref{ul}, \eqref{par_1}, and \eqref{par_n}, we can obtain
\begin{align*}
\delta^{\gamma}\|\wdt{\sf u}_{\ell}\|_{L^{\infty}(\widehat{\om}_{\dt}(z))}
\leq \|\varphi^{(\ell)}\|_{L^{\infty}(\Gamma_{1^+})},
\end{align*}
and
\begin{align*}
\delta^{\gamma}\|\nabla \wdt{\sf u}_{\ell}\|_{L^{\infty}(\widehat{\om}_{\dt}(z))}
\leq
\frac{C}{\va^{1-\gamma}}
\left|\varphi^{(\ell)}(z',\varepsilon/2+h_{1}(z'))\right|
+C\|\nabla\varphi^{(\ell)}\|_{L^{\infty}(\Gamma_1^+)}.
\end{align*}
Therefore,
\begin{align*}
\left\|\nabla {\sf w}_{\ell}\right\|_{L^{\infty}\widehat{\om}_{\dt(z')/4}(z)}
\leq
\frac{C}{\va^{\frac{1}{1+\gamma}}}
\left|\varphi^{(\ell)}(z', \varepsilon/2+h_1(z'))\right|
+C(\left\|\varphi^{(\ell)}\right\|_{C^{1,\gamma}(\Gamma_{1}^{+})}
+\left\|{\sf w}_{\ell}\right\|_{L^{2}(\Omega_1)}).
\end{align*}

\textbf{Case 2.} For $\va^{\frac{1}{1+\gamma}}<|z'|<1/2$.
By \eqref{1.47} and Proposition \ref{prop1}, we have
\begin{align*}
\delta^{-\frac{n}{2}}\|\nabla{\sf w}_{\ell}\|_{L^{2}(\widehat{\om}_{\dt(z')}(z))}
\leq
\frac{C}{|z'|}
\left|\varphi^{(\ell)}(z', \frac{\varepsilon}{2}+h_1(z'))\right|
+C(\left\|\varphi^{(\ell)}\right\|_{C^{1,\gamma}(\Gamma_{1}^{+})}
+\left\|{\sf w}_{\ell}\right\|_{L^{2}(\Omega_1)}),
\end{align*}
and
\begin{align*}
\delta^{\gamma}[\nabla\widetilde{\sf u}_{\ell}]_{\gamma,\widehat{\om}_{\dt}(z)}\leq
\frac{C}{|z'|}
\left|\varphi^{(\ell)}(z', \varepsilon/2+h_1(z'))\right|
+C\left\|\varphi^{(\ell)}\right\|_{C^{1,\gamma}(\Gamma_{1}^{+})}
.
\end{align*}
By using \eqref{ul}, \eqref{par_1}, and \eqref{par_n}, we can obtain
\begin{align*}
\delta^{\gamma}\|\wdt{\sf u}_{\ell}\|_{L^{\infty}(\widehat{\om}_{\dt}(z))}
\leq \|\varphi^{(\ell)}\|_{L^{\infty}(\Gamma_{1^+})},
\end{align*}
and
\begin{align*}
\delta^{\gamma}\|\nabla \wdt{\sf u}_{\ell}\|_{L^{\infty}(\widehat{\om}_{\dt}(z))}
\leq
\frac{C}{|z'|^{1-\gamma^2}}
\left|\varphi^{(\ell)}(z',\varepsilon/2+h_{1}(z'))\right|
+C\|\nabla\varphi^{(\ell)}\|_{L^{\infty}(\Gamma_1^+)}.
\end{align*}
It follows that
\begin{align*}
\left\|\nabla {\sf w}_{\ell}\right\|_{L^{\infty}(\widehat{\om}_{\dt(z')/4}(z))}
\leq
\frac{C}{|z'|}
\left|\varphi^{(\ell)}(z', \varepsilon/2+h_1(z'))\right|
+C\left(\left\|\varphi^{(\ell)}\right\|_{C^{1,\gamma}(\Gamma_{1}^{+})}
+\left\|{\sf w}_{\ell}\right\|_{L^{2}(\Omega_1)}\right).
\end{align*}
Noticed that $|\nabla {\sf v}_{\ell}|\leq |\nabla {\sf w}_{\ell}|+|\nabla \wdt{\sf u}_{\ell}|$. By \eqref{par_1}, \eqref{par_n}, \eqref{1.46}, and \eqref{1.47}, we obtain \eqref{1.48}.

It is clear that if $\varphi^{(\ell)}\ne 0$, then
\begin{align*}
|\nabla {\sf v}_{\ell}(\vec{0}_{n-1},x_n)|\geq
|\nabla\wdt{\sf u}_{\ell}(\vec{0}_{n-1},x_n)|
-|\nabla{\sf w}_{\ell}(\vec{0}_{n-1},x_n)|
\geq\frac{|\varphi^{(\ell)}(\vec{0}_{n-1},x_n)|}{C\va}.
\end{align*}
The Lemma \ref{lem1} is proved with $\psi\equiv0$.
\end{proof}

\begin{proof}[Proof of Theorem \ref{Them1}]
By Lemma \ref{lem2}, Lemma \ref{lemmabddlocal}, and  Lemma \ref{lem1},
we have for $x\in \Omega_{1/2}$ as in \eqref{Omega},
\begin{align*}
|\nabla {\sf u}(x)|\leq \sum_{\ell=1}^{m}|\nabla{\sf v}_{\ell}(x)|
\leq&\frac{C\left|\varphi(x',\varepsilon/2+h_{1}(x'))
-\psi(x',-\varepsilon/2+h_2(x'))\right|}{\va+|x'|^{1+\gamma}}\\
+&C\left(\|\varphi\|_{C^{1,\gamma}(\Gamma_{1}^+)}+\|\psi\|_{C^{1,\gamma}(\Gamma_1^-)} +\|{\sf u}\|_{L^{2}(\Omega_1)}\right).
\end{align*}
If $\varphi^{(\ell)}(\vec{0}_{n-1},\va/2)\ne \psi^{(\ell)}(\vec{0}_{n-1},-\va/2)$
 for some integer $\ell$, then by Lemma \ref{lem1}, we can obtain
\begin{align*}
|\nabla{\sf u}(\vec{0}_{n-1},x_n)|\geq\frac{\left|\varphi^{(\ell)}(\vec{0}_{n-1},\va/2)
-\psi^{(\ell)}(\vec{0}_{n-1},-\va/2) \right|}{C\va}\quad
\forall\,x_{n}\in (-\frac{\va}{2},\frac{\va}{2}).
\end{align*}
The proof of Theorem \ref{Them1} is completed.
\end{proof}

\section{Proof of Corollary \ref{cor1}}\label{sc3}
Since the Lam\'{e} systems as in \eqref{lame} has many more applications in practice, such as
 shear modulus in high contrast linear elastic composites.
 This section establishes
 the gradient estimates of Lam\'{e} systems under the assumptions in Definition \ref{defn1}.
 We give a sketched proof of Corollary \ref{cor1} and only list its main ingredients.
In order to make the proof more clear and concise, we will prove the case when $n=2$.

Let ${\sf v}_1=(v_1^{(1)}, v_{1}^{(2)})=(u^{(1)},0)$ be a weak solution of
\begin{align}\label{Lv1}
\left\{
  \begin{array}{ll}
     \mathcal{L}_{\lambda_1,\mu_{1}} {\sf v}_{1}
     =\nabla \cdot(\mathbb{C}^{0} e({\sf v}_{1}))=0,&\,\hbox{in}\,\Omega_1,
 \\
   {\sf v}_{1}=(\varphi^{(1)},0 ), &\, \hbox{on}\,\Gamma_{1}^{+}, \\
   {\sf v}_{1}=(\psi^{(1)},0), &\, \hbox{on}\,\Gamma_{1}^{-},
  \end{array}
\right.
\end{align}
and
${\sf v}_{2}=(v_{2}^{(1)},v_{2}^{(2)})=(0,u^{(2)})$
 be a weak solution of
\begin{align}\label{Lv2}
\left\{
  \begin{array}{ll}
     \mathcal{L}_{\lambda_1,\mu_{1}} {\sf v}_{2}
     =\nabla \cdot(\mathbb{C}^{0} e({\sf v}_{2}))=0,&\,\hbox{in}\,\Omega_1,
 \\
   {\sf v}_{2}=(0, \varphi^{(2)}), &\,\hbox{on}\,\Gamma_{1}^{+}, \\
   {\sf v}_{2}=( 0,\psi^{(1)}), &\,\hbox{on}~\Gamma_{1}^{-}.
  \end{array}
\right.
\end{align}
It obvious that
\begin{align}\label{LD}
\sf{u}={v_1}+{v_2}\quad \mbox{and}\quad \nabla{\sf u}=\nabla{\sf v}_1+\nabla{\sf v}_2.
\end{align}
Then we still construct the auxiliary function $\wdt{\sf u}_{\ell}$ for $\ell =1,2$
 as shown in \eqref{ul} in Section \ref{sc1}:
 \begin{align}
 \begin{aligned}\label{Lu12}
\wdt{\sf u}_{1}:&= \Big(\varphi^{(1)}(x_1,\frac{\varepsilon}{2}+h_{1}(x_1))\bar{u}(x)
+\psi^{(1)}(x_1,-\frac{\varepsilon}{2}+h_2(x_1))(1-\bar{u}(x)),~0\Big),\\
\wdt{\sf u}_{2}:&= \Big(0,~\varphi^{(2)}(x_1,\frac{\varepsilon}{2}+h_{1}(x_1))\bar{u}(x)
+\psi^{(2)}(x_1,-\frac{\varepsilon}{2}+h_2(x_1))(1-\bar{u}(x))\Big).
\end{aligned}
 \end{align}
 So $|\nabla \wdt{\sf u}_{\ell}|$ also has the gradient estimates as in \eqref{par_1}-\eqref{par_n},
 and  the H\"{o}lder semi-norm estimates of $\nabla \wdt{\sf u}_{\ell}$ as in Proposition \ref{prop1}.

Denote ${\sf w}_{\ell}={\sf v}_{\ell}-\wdt{\sf u}_{\ell}$ for any $\ell=1,2$, which satisfies the following boundary value problem:
\begin{align}\label{Lwl}
\begin{aligned}
\left\{
  \begin{array}{ll}
\mathcal{L}_{\lambda_1, \mu_{1}}{\sf w}_{\ell}
=-\mathcal{L}_{\lambda_1,\mu_1}\wdt{\sf u}_{\ell}, &\quad \hbox{in}~\Omega_{1}, \\
   {\sf w}_{\ell}=0, & \quad\hbox{on}~\Gamma_{1}^{+},\\
   {\sf w}_{\ell}=0, &\quad \hbox{on}~\Gamma_{1}^{-}.
  \end{array}
\right.
\end{aligned}
\end{align}
Because the result in Corollary \ref{cor1} independent of $\ell$,
 we might as well consider only the case of $\ell=1$.

\begin{lemma}
Under the hypotheses of Lemma \ref{lem2}, and in addition that $w_{1}$ is the weak solution of \eqref{Lwl},
then there exists a positive constant  $C$ independent of $\va$, such that,
\begin{align}\label{L1.4}
\int_{\Omega_{1/2}}|\nabla w_{1}|^2\,dx
\leq C\left(\left\|{\sf w}_{1}\right\|_{L^{2}(\Omega_{1})}^{2}
+
\|\varphi^{(1)}\|^2_{C^{1,\gamma}(\Gamma_{1}^{+})}
+
\|\psi^{(1)}\|^2_{C^{1,\gamma}(\Gamma_{1}^{-})}
\right).
\end{align}
\end{lemma}
\begin{proof}
Multiplying \eqref{Lwl} by ${\sf w}_{1}$ and making use of  the integration by parts
in $\Omega_{1/2}$, in view of ${\sf w}_{1}=0$ on $\Gamma_{1}^{\pm}$, we have
\begin{align}\label{L1.1}
&\int_{\Omega_{1/2}}
\Big(\mathbb{C}^0e({\sf w}_{1}), e({\sf w}_{1})\Big)dx
=\int_{\Omega_{1/2}}\big(\nabla\cdot(\mathbb{C}^0
e(\wdt{\sf u}_{1}))\big)\cdot {\sf w}_{1} \, dx.
\end{align}
For the right hand side of \eqref{L1.1},  noticing that $\ptl_{22} \wdt{u}_{1}^{(1)}=0$ in $\Omega_{1/2}$,
 and by using integration by parts, \eqref{par_1} and \eqref{1.35} , one has
\begin{align}
&\left|
\int_{\Omega_{1/2}}\big(\nabla\cdot(\mathbb{C}^0
e(\wdt{\sf u}_{1}))\big)\cdot {\sf w}_{1} \, dx\right|\nonumber\\
\leq&C\left|\int_{\Omega_{1/2}}
\ptl_1(\ptl_1 \wdt{u}^{(1)}_{1})w_1^{(1)}\,dx
+\int_{\Omega_{1/2}}
\ptl_2(\ptl_1 \wdt{u}_{1}^{(1)}) w_1^{(2)} \,dx\right|\nonumber\\
\leq&C\int_{\Omega_{1/2}}\left|\ptl_1\wdt{u}_{1}^{(1)}\right|\left|\nabla w_{1}\right|\,dx
+\int_{\substack{|x_1|=1/2\\ -\frac{\va}{2}+h_2(x_1)<x_2<\frac{\va}{2}+h_1(x_1)}}
\left|w_1\right|\left|\ptl_{1}\wdt{u}_{1}^{(1)}\right|dx_2\nonumber\\
\leq&
\frac{\lambda}{2}\int_{\Omega_{1/2}}|\nabla w_{1}|^2\,dx
+C\left(\left\|{\sf w}_{1}\right\|_{L^{2}(\Omega_{1})}^{2}
+
\|\varphi^{(1)}\|^2_{C^{1,\gamma}(\Gamma_{1}^{+})}
+
\|\psi^{(1)}\|^2_{C^{1,\gamma}(\Gamma_{1}^{-})}
\right).\label{L1.2}
\end{align}
For the left hand side of \eqref{L1.1}, it follows from
strong ellipticity condition as in \eqref{LHC}
there exists a positive constant $\lambda$, such that,
\begin{align}\label{L1.3}
\lambda\int_{\Omega_{1/2}}|\nabla {\sf w}_{1}|^2\,dx\leq
\int_{\Omega_{1/2}}
\Big(\mathbb{C}^0e({\sf w}_{1}), e({\sf w}_{1})\Big)dx.
\end{align}
By \eqref{L1.1}-\eqref{L1.3}, we obtain \eqref{L1.4}.
\end{proof}
From the definition of $\wdt{\sf u}_1$ , we can get
\begin{equation*}
\mathbb{C}^0e(\wdt{\sf u}_1)=\begin{pmatrix}
~(\lam_1+2\mu_1)\ptl_1\wdt{\sf u}_{1}^{(1)} &
\mu_1 \ptl_2\wdt{\sf u}_{1}^{(1)} ~\\
~~\\
~\mu_1 \ptl_2\wdt{\sf u}_{1}^{(1)} &
 \lam_1 \ptl_1\wdt{\sf u}_{1}^{(1)}~
\end{pmatrix}.
\end{equation*}
Let
\begin{align}\label{L1.9}
\mathcal{M}:=\int_{\widehat{\Omega}_{s}(z_1)}\!\!\!\!\!\!\!\!\!\!\!\!\!\!\!\!\!
 {}-{} \,\,\mathbb{C}^0e(\bar{u}_1^1(y)) \, dy:=\frac{1}{|\widehat{\Omega}_{s}(z_1)|}\int_{\widehat{\Omega}_{s}(z_1)} \mathbb{C}^0e(\bar{u}_1^1(y)) \ dy .
\end{align}
It follows that from \eqref{Lwl} that ${\sf w}_1$ satisfy
\begin{align}\label{Lwm1}
\mathcal{L}_{\lambda_1, \mu_{1}}{\sf w}_{1}
=-\mathcal{L}_{\lambda_1,\mu_1}(\wdt{\sf u}_{1}-\mathcal{M}).
\end{align}

\begin{lemma}
Under the hypotheses of Lemma \ref{lemmabddlocal} and in addition
that $w_{1}$ is the weak solution of \eqref{Lwl},
then there exists a positive constant  $C$ independent of $\va$, such that,
for $0\leq|z_1|\leq \va^{\frac{1}{1+\gamma}}$,
\begin{align}\label{L1.5}
\int_{\widehat{\Omega}_{\delta(z_1)}(z)}\left|\nabla {\sf w}_{1}\right|^{2}\ dx
\leq&
C\va^{\frac{2\gamma}{1+\gamma}}
\left(
\left|\varphi^{(1)}(z_{1}, \frac{\varepsilon}{2}+h_1(z_1))
-\psi^{(1)}(z_{1},-\frac{\va}{2}+h_2(z_{1}))\right|^2\right)\nonumber\\
    \quad&
    +C\va^{\frac{2\gamma}{1+\gamma}+2}
    \left(
    \left\|{\sf w}_{1}\right\|^2_{L^{2}(\Omega_1)}
    +\left\|\varphi^{(1)}\right\|_{C^{1,\gamma}(\Gamma_{1}^{+})}^2
+\left\|\psi^{(1)}\right\|_{C^{1,\gamma}(\Gamma_{1}^{-})}^2\right),
\end{align}
and for $\va^{\frac{1}{1+\gamma}}<|z_1|\leq 1/2$,
\begin{align}\label{L1.6}
\int_{\widehat{\Omega}_{\delta(z_1)}(z)}
\left|\nabla {\sf w}_{1}\right|^{2}\ dx
\leq&
C|z_1|^{2\gamma}\left(\left|\varphi^{(1)}(z_1, \frac{\varepsilon}{2}+h_1(z_1))
-\psi^{(1)}(z_1,-\frac{\va}{2}+h_2(z_1))\right|^2\right)\nonumber\\
\quad&
+C|z_1|^{2\gamma+2(1+\gamma)}
\left(
\left\|{\sf w}_{1}\right\|^2_{L^{2}(\Omega_1)}
+\left\|\varphi^{(1)}\right\|_{C^{1,\gamma}(\Gamma_{1}^{+})}^2
+\left\|\psi^{(1)}\right\|_{C^{1,\gamma}(\Gamma_{1}^{-})}^2
\right).
\end{align}
\end{lemma}
\begin{proof}
For simplicity, we assume that $\psi\equiv0$.
Similar to the proof of Lemma \ref{lemmabddlocal},
for a fixed point $z=(z_1,z_2)$,  consider the following cut-off function $\eta(x_1)$:
 for $0<t<s<1/2$, let $\eta$ be a cut-off function satisfying
\begin{align*}
\eta(x_1)=\begin{cases}
1 & \text{if} ~ |x_1-z_1|<t, \\
0 & \text{if} ~ |x_1-z_1|>s,\\
\end{cases}
\quad\text{and}\quad |\eta'(x_1)|\leq\frac{2}{s-t}.
\end{align*}
Multiplying  \eqref{Lwm1} by $\eta^{2}{\sf w}_{1}$ and using the integration by parts, one has
\begin{align}\label{L1.8}
\int_{\widehat{\Omega}_{s}(z)}\Big(\mathbb{C}^0e({\sf w}_1), e(\eta^2{\sf w}_1)\Big) \ dx
=- \int_{\widehat{\Omega}_{s}(z)} \Big(\mathbb{C}^0e(\wdt{\sf u}_1)-\mathcal{M},
\nabla (\eta^2{\sf w}_1)\Big) \ dx.
\end{align}
By treating \eqref{L1.8} in the same way as proposition 2.1 in \cite{CL}, we get
\begin{equation}\label{L.11}
\int_{\widehat{\Omega}_{t}(z)}
|\nabla {\sf w}_1|^{2}\ dx
\leq\,
\frac{C}{(s-t)^{2}}\int_{\widehat{\Omega}_{s}(z)}|{\sf w}_1|^{2}\ dx
+C\int_{\widehat{\Omega}_{s}(z)} |\mathbb{C}^0e(\wdt{\sf u}_{1})-\mathcal{M}|^2 \ dx.
\end{equation}
\textbf{Case 1.} For $|z_1|\leq \varepsilon^{\frac{1}{1+\gamma}}$
and $0< s<\varepsilon^{\frac{1}{1+\gamma}}$, we have
$\varepsilon\leq \delta(z_1)\leq C\varepsilon$.
By a direct calculation, we have
\begin{align}\label{L1.10}
\int_{\widehat{\Omega}_{s}(z)}|{\sf w}_{1}|^{2}\,dx
=\int_{\widehat{\Omega}_{s}(z)}\left|
\int_{-\frac{\varepsilon}{2}+h_2(x_1)}^{x_{2}}\partial_{2}{\sf w}_{1}(x_1,x_{2})\,dx_{2}
\right|^2\,dx
\leq C\varepsilon^2\int_{\widehat{\Omega}_{s}(z)}|\nabla {\sf w}_{1}|^2\,dx.
\end{align}
 In the view of Proposition \ref{prop1} and \eqref{L1.9}, one has
 \begin{align}
\int_{\widehat{\Omega}_{s}(z)} |\mathbb{C}^0e(\wdt{\sf u}_1)-\mathcal{M}|^2 \ dx \notag
\leq& C [\nabla \wdt{\sf u}_1^{(1)}]_{\gamma, \, \widehat{\Omega}_{s}(z_1)}^2  \int_{\widehat{\Omega}_{s}(z_1)}  (s^{2\gamma}+\delta(z_1)^{2\gamma}) \ dx \notag \nonumber\\
\leq&C\left(\left|\varphi^{(1)}(z_1, \frac{\varepsilon}{2}+h_1(z_1))\right|^2
+s^2\|\varphi^{(1)}\|_{C^{1,\gamma}(\Gamma_{1}^{+})}^2\right)\nonumber\\
&\quad \left(
\frac{s^{3}}{\varepsilon^{1+\frac{2}{1+\gamma}}}
+\frac{s}{\varepsilon^{\frac{2}{1+\gamma}-1}}
+\frac{s^{3-2\gamma}}{\varepsilon^{1+\frac{2}{1+\gamma}-2\gamma}}
+\frac{s^{1+2\gamma}}{\varepsilon^{1+\frac{2\gamma^{2}}{1+\gamma}}}
\right)=:G_{1}(s).\label{L1.12}
\end{align}

{\bf Case 2.} For $\varepsilon^{\frac{1}{1+\gamma}}\leq|z_1|\leq\,\frac{1}{2}$
 and $0<s<|z_1|$, we have $\frac{1}{C}|z_1|^{1+\gamma}\leq\delta(z_1)\leq\,C|z_1|^{1+\gamma}$.
 Estimates \eqref{L1.10} and \eqref{L1.12} become, respectively,
\begin{align}\label{1.14}
\int_{\widehat{\Omega}_{s}(z)}|{\sf w}_{1}|^{2}\ dx
\leq&\,C|z_1|^{2(1+\gamma)}\int_{\widehat{\Omega}_{s}(z_1)}|\nabla{{\sf w}_{1}}|^{2}\ dx, \quad\mbox{if}~\,0<s<\frac{2}{3}|z'|,
\end{align}
and
\begin{align}
\int_{\widehat{\Omega}_{s}(z)} |\mathbb{C}^0e(\wdt{\sf u}_1)-\mathcal{M}|^2 \ dx \notag
\leq&
C\left(\big|\varphi^{(\ell)}(z', \frac{\varepsilon}{2}+h_1(z'))\big|^2
 +s^2\left\|\varphi^{(\ell)}\right\|_{C^{1,\gamma}(\Gamma_{1}^{+})}^2\right)\nonumber\\
 \quad&\left(
\frac{s^{3}}{|z_1|^{3+\gamma}}
+\frac{s}{|z_1|^{-1-\gamma}}
+\frac{s^{3-2\gamma}}{|z_1|^{1-\gamma-2\gamma^2}}
+\frac{s^{2+2\gamma}}{|z_1|^{\gamma-1+2\gamma^2}}
\right).
\end{align}
Next, similar to the proof of Lemma \eqref{lemmabddlocal}, we complete the proof of \eqref{L1.5}
and \eqref{L1.6}.
\end{proof}

\begin{lemma}
Under the hypotheses of Lemma \ref{lem1} and in addition that $w_{1}$ is the weak solution of \eqref{Lwl},
then there exists a positive constant  $C$ independent of $\va$ such that,
for $ |z_1|\leq \va^{\frac{1}{1+\gamma}}$,
\begin{align}\label{L1.14}
\left|\nabla {\sf w}_{1}(z_1,z_2)\right|
\leq&C\va^{-\frac{1}{1+\gamma}}\left|\varphi^{(1)}(z', \frac{\varepsilon}{2}+h_1(z_1))
-\psi^{(1)}(z_1, -\frac{\varepsilon}{2}+h_2(z_1))\right|
\nonumber\\
&+C\va^{-\frac{\gamma}{1+\gamma}}
\left(
\|{\sf w}_{1}\|_{L^2(\Omega_1)}
+
\|\varphi^{(1)}\|_{C^{1,\gamma}(\Gamma_1^{+})}
+\|\psi^{(1)}\|_{C^{1,\gamma}(\Gamma_1^{-})}
\right),
\end{align}
and for $\va^{\frac{1}{1+\gamma}}<|z_1|<\frac{1}{2}$,
\begin{align}\label{1.15}
\left|\nabla {\sf w}_{1}(z_1,z_2)\right|
\leq& C|z_1|^{-1}\left|\varphi^{(1)}(z_1, \frac{\varepsilon}{2}+h_1(z'))
-\psi^{(1)}(z_1, -\frac{\varepsilon}{2}+h_2(z_1))\right|
  \nonumber\\
&+C|z_1|^{-\gamma}
\left(
\|{\sf w}_{1}\|_{L^2(\Omega_1)}
+\|\varphi^{(1)}\|_{C^{1,\gamma}(\Gamma_1^{+})}
+\|\psi^{(1)}\|_{C^{1,\gamma}(\Gamma_1^{-})}
\right).
\end{align}
\end{lemma}
\begin{proof}
Let $\mathcal{Q}_{1}$ be as in \eqref{LQ} and $C_{ijkl}$ be as in \eqref{cijkl}.
For any $(y_1, y_2)\in \mathcal{Q}_1$.
we denote
\begin{align*}
 \widetilde{\sf w}_1(y_1, y_2):= {\sf w}_1(\dt(z_1) y_1 + z_1 , \dt(z_1) y_2),
\quad \widehat{\sf u}_1(y_1 , y_2):=\wdt{\sf u}_1(\dt(z_1) y_1 +z_1, \dt(z_1) y_2),
\end{align*}
then, after the same coordinate transformation as in Lemma \ref{lem1},
we can obtain that $\widetilde{\sf w}_{1}$ satisfies
\begin{equation}\label{L1.13}
\left\{ \begin{aligned}
-\sum_{j,k,l}\ptl_j\big(C_{ijkl}\ptl_l \widetilde{w}_{1}^{(k)} \big) &=
\sum_{j,k,l}\ptl_j\big(C_{ijkl}\ptl_l\widehat{u}_1^{(k)}\big)\quad &\text{in}~~\mathcal{Q}_1, \\
\widetilde{\sf w}_{1}&=0 \quad &\text{on} ~~\widehat{\Gamma}_1^{\pm}.
\end{aligned}\right.
\end{equation}
Similar to the proof in Lemma \eqref{lem1},
recalling back to the original region $\widehat{\Omega}_{\delta(z_1)}(z)$,
one has
\begin{align}\label{L1.15}
\|\nabla {\sf w}_1\|_{L^\infty( \widehat{\Omega}_{\delta(z_1)/4}(z))}
\leq
C\delta(z_1)^{-1}\|\nabla {\sf w}_{1}\|_{L^2( \widehat{\Omega}_{\delta(z_1)}(z_1))}
+C\delta(z_1)^{\gamma}[\nabla \wdt{\sf u}_1]_{\gamma, \, \widehat{\Omega}_{\delta(z_1)}(z) }).
\end{align}
Therefore, by using  \eqref{L1.5} and \eqref{L1.6}, Proposition \eqref{prop1}, we proved
the \eqref{L1.14} and \eqref{L1.15}.
\end{proof}

\begin{proof}[Proof of Corollary \ref{cor1}]
Consequently, by \eqref{par_1} and \eqref{par_n}, we have for sufficiently small $\va$ and $z\in \Omega_{1/2}$,
\begin{align}\label{L1.48}
|\nabla {\sf v}_{1}(z)|\leq&
\frac{C\left|\varphi^{(1)}(z_1, \frac{\varepsilon}{2}+h_1(z_1))-\psi^{(1)}(z_1, -\frac{\varepsilon}{2}+h_2(z_1))\right|}
{\va+|z_1|^{1+\gamma}}
\nonumber  \\
&\quad+C\left(
\|{\sf v}_{1}\|_{L^2(\Omega_1)}
+\|\varphi^{(1)}\|_{C^{1,\gamma}(\Gamma_1^{+})}
+\|\psi^{(1)}\|_{C^{1,\gamma}(\Gamma_1^{-})}\right).
\end{align}
By \eqref{LD},
we have for $x\in \Omega_{1/2}$ as in \eqref{Omega},
\begin{align*}
|\nabla {\sf u}(x)|\leq& |\nabla{\sf v}_{1}(x)|+|\nabla{\sf v}_{2}(x)|\\
\leq&\frac{C\left|\varphi(x_1,\varepsilon/2+h_{1}(x_1))
-\psi(x_1,-\varepsilon/2+h_2(x_1))\right|}{\va+|x_1|^{1+\gamma}}\\
&+C\left(\|\varphi\|_{C^{1,\gamma}(\Gamma_{1}^+)}+\|\psi\|_{C^{1,\gamma}(\Gamma_1^-)} +\|{\sf u}\|_{L^{2}(\Omega_1)}\right).
\end{align*}
If $\varphi^{(\ell)}(0,\va/2)\ne \psi^{(\ell)}(0,-\va/2)$
 for some integer $\ell$, then by Lemma \ref{lem1}, we can obtain
\begin{align*}
|\nabla{\sf u}(0,x_2)|\geq\frac{\left|\varphi^{(\ell)}(0,\va/2)
-\psi^{(\ell)}(0,-\va/2) \right|}{C\va}\quad
\forall\,x_{2}\in (-\frac{\va}{2},\frac{\va}{2}).
\end{align*}
The proof of Corollary \ref{cor1} is completed.
\end{proof}

\section{Appendix:
Proof of $C^{1,\gamma}$ estimates and $W^{1,p}$ estimates }\label{sec5}

In this section, we show the proofs of the Theorem \ref{C1Gamma} and  Theorem \ref{them2},
which play a key role in the proof of Theorem \ref{Them1},
with the help of the Campanato's approach, Schauder estimates
and $L^{p}$ estimates for elliptic systems in \cite{Gia}.

\subsection{Proof of Theorem \ref{C1Gamma}}
To prove Theorem \ref{C1Gamma}, we first introduce the definition of the spaces of  Morrey and Campanato (see \cite[Chapter 5]{Gia}).

Let  $Q\subset \mathbb{R}^n$ be any domain and $\rho>0$, for any $x_{0}\in Q$
we use the \emph{symbol}
$Q(x_{0}; \rho)$ to denote the set $Q\cap B_{\rho}(x_{0})$
and  the \emph{symbol} $\dim Q$ to
denote the diameter of $Q$.
The domain $\Omega$ is said to be
a \emph{Lipschitz domain}
if $\ptl\Omega$ is Lipschitz defined as in Definition \ref{defnbdd}.
\begin{definition}
Let $Q$ be a Lipschitz domain in $\mathbb{R}^{n}$.
For every $1\leq p\leq +\infty,$ $\lambda>0$ define the Morrey space $L^{p,\lambda}(Q)$,
\begin{align*}
L^{p,\lambda}(Q):=\Big\{u\in L^{p}(Q):\sup\limits_{x_{0}\in Q,~\rho>0}\rho^{-\lambda}\int_{Q(x_{0},\rho)}|u|^{p}\,dx
<+\infty\Big\},
\end{align*}
endowed with the norm defined by
\begin{align*}
\|u\|_{L^{p,\lambda}(Q)}:
=\Big(\sup_{x_{0}\in Q,~\rho>0}\rho^{-\lambda}\int_{Q(x_{0},\rho)}
|u|^{p}\,dx\Big)^{\frac{1}{p}}.
\end{align*}
\end{definition}

\begin{definition}
Let $Q$ be a Lipschitz domain in $\mathbb{R}^{n}$.
For every $1\leq p\leq +\infty,$ $\lambda>0$ define
the Campanato space $\mathcal{L}^{p,\lambda}(Q)$,
\begin{align*}
\mathcal{L}^{p,\lambda}(Q):=\{u\in L^{p}(Q):\sup_{x_{0}\in Q, \rho>0}\rho^{-\lambda}
\int_{Q(x_{0},\rho)}|u-u_{x_{0},\rho}|^{p}\,dx
<+\infty\},
\end{align*}
endowed with the norm defined by
\begin{align}
\|u\|_{\mathcal{L}^{p,\lambda}(Q)}:&=[u]_{p,\lambda}+\|u\|_{L^{p}}\nonumber\\
:=&
\Big(\sup_{x_{0}\in Q,\rho>0}\rho^{-\lambda}
\int_{Q(x_{0},\rho)}|u-u_{x_{0},\rho}|^{p}\,dx\Big)^{\frac{1}{p}}
+\Big(\int_{Q}|u|^{p} \,dx\Big)^{\frac{1}{p}} < +\infty,\label{Csemi}
\end{align}
where $u_{x_{0},\rho}:=\frac{1}{|Q(x_{0},\rho)|}\int_{Q(x_{0},\rho)}u\,dx$.
\end{definition}

The follows lemma is just \cite[Theorem 5.5]{Gia}.
\begin{lemma} \label{prop2}
For $n<\lambda \leq n+p$ and $\gamma=\frac{\lambda-n}{p}$
we have $\mathcal{L}^{p, \lambda}(Q)=C^{0, \gamma}(\overline{Q}) .$ Moreover the H\"{o}lder semi-norm
$[u]_{0, \gamma}$ as in \eqref{semi}
is equivalent to $[u]_{p, \lambda}$ as in \eqref{Csemi}.
 If $\lambda>n+p$ and $u \in \mathcal{L}^{p, \lambda}(\Omega)$, then $u$ is constant.
\end{lemma}

Referring to \cite[Theorem 5.14]{Gia}, we can obtain the following interior estimates.
In what follows, for any domain $Q\subset \mathbb{R}^n$ we denote by the \emph{symbol $\mathcal{L}_{loc}^{p,\lambda}(Q)$}
the set of all  functions $u$ which satisfy for any $Q'\subset\subset Q$,
$\|u\|_{\mathcal{L}^{p,\lambda}(Q')}< \infty$.
\begin{lemma}\label{them5.14}
Let $Q$ is a Lipschitz domain in $\mathbb{R}^{n}$.
Let ${A}_{ij}^{\alpha\beta}$ be  constant and satisfy  \eqref{LHC}
and \eqref{AUP}.
Let $0<\gamma<1$, $\mu:=n+2\gamma-2$ and for any $\alp=1,\cdots, n,$ $i=1,\cdots, m$,
${F}_{i}^{(\alpha)}\in\mathcal{L}^{2,\mu+2}(Q)$ and ${H}^{(i)}\in L^{2,\mu}(Q)$.
Let ${\sf w}=({w}^{(1)},\cdots,{w}^{(m)})\in W^{1,2}(Q\subset \R^{n};\mathbb{R}^{m})$ be a weak solution of
\begin{align}\label{D2}
\sum_{\alp,\bt,j}\partial_{\alpha}({A}_{ij}^{\alpha\beta}\partial_{\beta}{w}^{(j)})
={H}^{(i)}-\sum\limits_{\alp}\partial_{\alpha} {F}_{i}^{(\alpha)}, ~~~\mbox{in}~ Q.
\end{align}
Then $\ptl_{\alp}{w}^{(i)}\in \mathcal{L}_{loc}^{2,\mu}(Q)$ for any $\alp=1,\cdots, n$ and $i=1,\cdots,m$,
and
there exists a positive constants $C$ depending on $n, m, \gamma, R,\lambda,\Lambda$ such that, for $B_{R}(x_{0})\subset Q$,
\begin{align}\label{1.38}
[\nabla {\sf w}]_{\gamma,B_{R/2}}:=
\max\limits_{\alp,i}[\ptl_{\alp}{w}^{(i)}]_{\gamma,B_{R/2}}
\leq C\Big( \frac{1}{R^{1+\gamma}}\|{\sf w}\|_{L^{\infty}(B_{R})}+
[{\sf F}]_{\gamma, B_{R}}+ \|{\sf H}\|_{L^{2,\mu}(B_{R})} \Big),
\end{align}
where
$
\|{\sf H}\|_{L^{2,\mu}(B_{R})}:=\max\limits_{i}\|{H}^{(i)}\|_{L^{2,\mu}(B_{R})}.
$
\end{lemma}

\begin{proof}
By Proposition \ref{prop2} we have ${F}_{i}^{(\alpha)}\in \mathcal{L}^{2,n+2\gamma}(Q)$.
For a given ball $B_{R}:=B_{R}(x_{0})\subset Q$, the decomposition of ${\sf w}$ is as follows
\begin{align}\label{0.1}
{\sf w}={\sf w}_1+{\sf w}_2,\quad \mbox{in}\,B_{R},
\end{align}
where ${\sf w}_1$ and ${\sf w}_{2}$ satisfy, respectively,
\begin{align}\label{1.29}
\left\{
\begin{array}{ll}
  \sum\limits_{\alp,\bt,j} \partial_{\alpha}(A_{ij}^{\alpha\beta}
  \partial_{\beta}{ w}_1^{(j)})=0 , & \hbox{in}\,B_{R},  \\
    {\sf w}_1={\sf w}, & \hbox{on}\,\ptl B_{R},
  \end{array}
\right.
\end{align}
and
\begin{align}\label{1.28}
\left\{
  \begin{array}{ll}
   \sum_{\alp,\bt,j} \partial_{\alpha}({A}_{ij}^{\alpha\beta}\partial_{\beta}{w}_2^{(j)})= {H}^{(i)}
-\sum_{\alp}\partial_{\alpha} ({F}_{i}^{(\alpha)}-
({F}_{i}^{(\alpha)})_{R}), &\hbox{in}\,B_{R},  \\
    {\sf w}_2=0, & \hbox{on}\,\ptl B_{R},
  \end{array}
\right.
\end{align}
where $({F}_{i}^{(\alpha)})_{R}=\frac{1}{|B_R|}\int_{B_{R}}F_{i}^{\alp}\,dx$.

By  \cite[Proposition 5.8]{Gia}, for $0<\rho<\frac{3R}{4}$ we have
\begin{align}\label{0.2}
\int_{B_{\rho}}|\nabla{\sf w}_1-(\nabla{\sf w}_1)_{\rho}|^{2}\,dx\leq
C\left(\frac{\rho}{R}\right)^{n+2}\int_{B_{3R/4}}|\nabla{\sf w}_1-(\nabla{\sf w}_{1})_{R}|^2\,dx,
\end{align}
and for ${\sf w}_2$, multiplying \eqref{1.28} by ${\sf w}_2$ and using the integration by parts, one has
\begin{align}\label{1.30}
\int_{B_{3R/4}}|\nabla {\sf w}_2|^2\,dx\leq
C R^{\mu+2}\left(
[{\sf F}]_{\mathcal{L}^{2,\mu+2}(B_{R})}
+\|{\sf H}\|_{L^{2,\mu}(B_R)}\right).
\end{align}
Consequently,
\begin{align}\label{1.36}
\int_{B_{\rho}}|\nabla{\sf w}-(\nabla{\sf w})_{\rho}|^2\,dx
\leq& \int_{B_{\rho}}|\nabla{\sf w}_1-(\nabla{\sf w}_1)_{\rho}+\nabla{\sf w}_2-(\nabla{\sf w}_2)_{\rho}|^2\,dx\nonumber\\
\leq & C_1\left(\frac{\rho}{R}\right)^n\int_{B_{3R/4}}|\nabla{\sf w}_1-(\nabla{\sf w}_1)_{R}|^2\,dx
+C_2\int_{B_{3R/4}}|\nabla{\sf w}_2-(\nabla{\sf w}_2)_{3R/4}|^2\,dx\nonumber\\
\leq& C_1\left(\frac{\rho}{R}\right)^n\int_{B_{3R/4}}|\nabla{\sf w}-(\nabla{\sf w})_{3R/4}|^2\,dx
+C_2\int_{B_{3R/4}}|\nabla{\sf w}_2|^2\,dx.
\end{align}
Inserting \eqref{1.30} in \eqref{1.36} and using \cite[Lemma 5.13]{Gia}, we obtain
\begin{align}\label{1.37}
\int_{B_{\rho}}|\nabla{\sf w}-(\nabla{\sf w})_{\rho}|^2\,dx\leq C\Big[\left( \frac{\rho}{R}\right)^{\mu+2}\int_{B_{3R/4}}
|\nabla{\sf w}|^{2}\,dx+\rho^{\mu+2}(
[{\sf F}]^2_{\mathcal{L}^{2,\mu+2}(B_{R})}
+\|{\sf H}\|^2_{L^{2,\mu}(B_R)})\Big].
\end{align}
We assert that the following inequality holds:
\begin{align}\label{1.33}
\int_{B_{3R/4}}|\nabla{\sf w}|^2\,dx\leq C
\Big(\frac{1}{R^2}\int_{B_R}|{\sf w}|^2\,dx
+ R^{\mu+2}\Big([{\sf F}]^2_{\mathcal{L}^{2,\mu+2}(B_{R})}+\|{\sf H}\|^2_{L^{2,\mu}(B_{R})}\Big)  \Big),
\end{align}
where $C$ depends on $\lambda$ and boundedness of coefficients of \eqref{D2}.

Actually,
define a cut-off function $\zeta\in C_{c}^{\infty }(Q)$ as follows, $0\leq\zeta(x)\leq 1$,
\begin{align*}
\zeta(x)=
\left\{
  \begin{array}{ll}
    1, &\, \hbox{on}~B_{3R/4},  \\
    0, & \,\hbox{on}~B_{R}\backslash B_{3R/4}
  \end{array}
\right. ,\quad |\nabla\zeta(x)|\leq \frac{8}{R},
\end{align*}
and  choose as test function ${\sf w}\zeta^2$ into \eqref{system}.
From  strong ellipticity condition \eqref{LHC}, we obtain
\begin{align*}
\lambda\int_{B_{3R/4}}\zeta^2|\nabla{\sf w}|^2\,dx
\leq&\sum\limits_{\alp,\bt,i,j}
 \int_{B_{3R/4}}
\zeta^2{A}_{ij}^{\alp\bt}\ptl_{\bt}{w}^{(j)}\ptl_{\alp}{w}^{(i)}\,dx\nonumber\\
=&-\sum\limits_{\alp,\bt,i,j}\int_{B_{3R/4}}2\zeta{w}^{(i)}
{A}_{ij}^{\alpha\beta} \partial_{\beta}{w}^{(j)}\partial_{\alpha}\zeta\,dx
-\sum\limits_{\alp,i,j}\int_{B_{3R/4}} {B}_{ij}^{\alpha}{w}^{(j)}\partial_{\alpha}\left( {w}^{(i)}\zeta^2\right)\,dx\nonumber\\
&-\sum\limits_{\bt,i,j}\int_{B_{3R/4}}{w}^{(i)}\zeta^{2}({C}_{ij}^{\beta}\partial_{\beta}{w}^{(j)}+{D}_{ij}{w}^{(j)})\,dx
+\sum\limits_{i}\int_{B_{3R/4}}{w}^{(i)}\zeta^{2}{H}^{(i)}\,dx
\nonumber\\
&+\sum\limits_{\alp,i}\int_{B_{3R/4}} ({F}_{i}^{(\alpha)}-(
{F}_{i}^{(\alp)})_{3R/4})\partial_{\alpha}\left( {w}^{(i)}\zeta^2\right)\,dx.
\end{align*}
Then by using Cauchy's inequality and the properties of $\zeta$, we proved \eqref{1.33}.

Therefore,  using \eqref{1.37} and \eqref{1.33}, we have
\begin{align*}
\frac{1}{\rho^{\mu+2}}\int_{B_{\rho}}|\nabla{\sf w}-(\nabla {\sf w})_{\rho}|^2\,dx
\leq&C
\Big(\frac{1}{R^{\mu+4}}\int_{B_R}|{\sf w}|^2\,dx
+[{\sf F}]^2_{\mathcal{L}^{2,\mu+2}(B_{R})}+\|{\sf H}\|^2_{L^{2,\mu}(B_{R})}  \Big)\nonumber\\
\leq&C
\Big(\frac{1}{R^{2+2\gamma}}\int_{B_R}|{\sf w}|^2\,dx
+[{\sf F}]^2_{\mathcal{L}^{2,\mu+2}(B_{R})}+\|{\sf H}\|^2_{L^{2,\mu}(B_{R})}  \Big),
\end{align*}
where $C$ depends on $n,\gamma$. For any $x=(x',x_n)\in B_{R/2}$ and $0<\rho\leq R/4$, we have
\begin{align*}
\frac{1}{\rho^{\mu+2}}\int_{B_{\rho}(x)\cap B_{R/2}}|\nabla{\sf w}-(\nabla {\sf w})_{B_{\rho}(x)\cap B_{R/2}}|^2\,dy
\leq&\frac{1}{\rho^{\mu+2}}\int_{B_{\rho}}|\nabla{\sf w}-(\nabla {\sf w})_{\rho}|^2\,dy\nonumber\\
\leq&C
\Big(\frac{1}{R^{2+2\gamma}}\int_{B_R}|{\sf w}|^2\,dx
+[{\sf F}]^2_{\mathcal{L}^{2,\mu+2}(B_{R})}+\|{\sf H}\|^2_{L^{2,\mu}(B_{R})} \Big).
\end{align*}
By the equivalence between the H\"{o}lder space and the Campanato space
(see Lemma \ref{prop2}), this implies that \eqref{1.38} holds.
\end{proof}
Next, we  give the boundary estimate on half space $\ptl\R^n_{+}$. Consider
\begin{align}\label{1.27}
\left\{
  \begin{array}{ll}
    \sum\limits_{\alp,\bt,j}\partial_{\alpha}({A}_{ij}^{\alpha\beta}\partial_{\beta}{w}^{(j)})
={H}^{(i)}-\sum\limits_{\alp}\partial_{\alpha} {F}_{i}^{(\alpha)}, & \hbox{in}~ \R^{n}_{+}, \\
    {\sf w}=0, & \hbox{on}~ \ptl\R_+^{n}.
  \end{array}
\right.
\end{align}

\begin{corollary}\label{them3}
In the hypothesis of Lemma \ref{them5.14}, let ${\sf w}\in W^{1,2}(\R^n_+;\mathbb{R}^{m}) $
be the solution of \eqref{1.27}, for any $x_{0}\in \ptl\R_{+}^n$ and
$ \mathcal{B}_{R}^+(x_0):={B}_{R}(x_0)\cap\ptl\R^n_+$, then there
is a constant $C$ only depended to $n, \gamma,\lambda, \Lambda$ such that,
\begin{align*}
[\nabla {\sf w}]_{\gamma,\mathcal{B}_{R/2}^+(x_0) }\leq
C\Big(\frac{1}{R^{1+\gamma}} \|{\sf w}\|_{L^{\infty}(\mathcal{B}_{R}^+(x_0))}+[{\sf F}]_{\gamma,\mathcal{B}_{R}^+(x_0)}
+\|{\sf H}\|_{L^{2,\mu}(\mathcal{B}_{R}^{+}(x_0))}
\Big).
\end{align*}
\end{corollary}

\begin{proof}
Firstly, we decompose ${\sf w}={\sf w}_1+{\sf w}_2$ as shown in \eqref{0.1},
where ${\sf w}_1$, ${\sf w}_2$ satisfy \eqref{1.29} and \eqref{1.28} in $\mathcal{B}_{R}^+(x_0)$
respectively. It follows from \cite[(5.38) in Theorem 5.21]{Gia}
that \eqref{0.2} holds for ${\sf w}_1$.
The proof of the corollary can be obtained by the method in the proof of the
  Lemma \ref{them5.14} and some elementary arguments. We omit the details.
\end{proof}

\begin{proof}[Proof of Theorem \ref{C1Gamma}.]

Since $\Gamma$ is $ C^{1,\gamma}$, then for any $x_{0}\in \Gamma$,there exists a neighbourhood $U$ of $x_{0}$ and a homeomorphism
$\Psi \in C^{1,\gamma}(U)$ such that
\begin{align*}
\Psi(U\cap Q)=\mathcal{B}_{1}^{+}:=\{y\in{B}_{1}(0), y_{n}>0\},~~~~
\Psi(U\cap\Gamma)=\partial\mathcal{B}_{1}^{+}\cap\{y\in\mathbb{R}^{n}:y_{n}=0\}.
\end{align*}
Under transformation $y=\Psi(x)=(\Psi^{(1)}(x),\cdots,\Psi^{(n)}(x)),$ we denote
\begin{align*}
\mathcal{W}(y):={\sf w}(\Psi^{-1}(y)),\quad \mathcal{J}(y):=\frac{\partial\left((\Psi^{-1})^{(1)},\cdots,(\Psi^{-1})^{(n)}\right)}{\partial(y_1,\cdots,y_n)},
\quad |\mathcal{J}(y)|:=\det{\mathcal{J}(y)},
\end{align*}
and
\begin{align*}
\mathcal{A}_{ij}^{\alpha\beta}(y):&=\sum_{\hat{\alp},\hat{\bt}}
|\mathcal{J}(y)|{A}_{ij}^{\hat{\alpha}\hat{\beta}}\left(\Psi^{-1}(y)\right)
(\partial_{\hat{\beta}}(\Psi^{-1})^{(\beta)}(y))^{-1}
\partial_{\hat{\alpha}}(\Psi)^{(\alpha)}(\Psi^{-1}(y)),\nonumber\\
\mathcal{B}_{ij}^{\alpha}(y):&=\sum_{\hat{\alp}}|\mathcal{J}(y)|
{B}_{ij}^{\hat{\alpha}}\left(\Psi^{-1}(y)\right)
\partial_{\hat{\alpha}}(\Psi)^{(\alpha)}(\Psi^{-1}(y)),\\
\mathcal{C}_{ij}^{\beta}(y):&=\sum_{\hat{\bt}}|\mathcal{J}(y)|
{C}_{ij}^{\hat{\beta}}\left(\Psi^{-1}(y)\right)
\partial_{\hat{\beta}}(\Psi)^{(\beta)}\left(\Psi^{-1}(y)\right),
\quad\mathcal{D}_{ij}(y):=|\mathcal{J}(y)|{D}_{ij}\left(\Psi^{-1}(y)\right)\nonumber,\\
\mathcal{F}_{i}^{(\alpha)}(y):&=\sum_{\hat{\alp}}|\mathcal{J}(y)|
{F}_{i}^{\hat{\alpha}}\left(\Psi^{-1}(y)\right) \partial_{\hat{\alpha}}(\Psi)^{(\alpha)}\left(\Psi^{-1}(y)\right),\quad
\mathcal{H}^{(i)}(y):=|\mathcal{J}(y)|{H}^{(i)}\left(\Psi^{-1}(y)\right)\nonumber,
\end{align*}
where $\hat{\alp},\hat{\bt}=1,\cdots,n$.
Then \eqref{system} becomes
\begin{align*}
\sum\limits_{\alp,\bt,i,j}\partial_{\alpha}(\mathcal{A}_{ij}^{\alpha\beta}\partial_{\beta}\mathcal{W}^{(j)}
+\mathcal{B}_{ij}^{\alpha}\mathcal{W}^{(j)})
+\mathcal{C}_{ij}^{\beta}\partial_{\beta}\mathcal{W}^{(j)}
+\mathcal{D}_{ij}\mathcal{W}^{(j)}
=\mathcal{H}^{(i)}
-\sum\limits_{\alp}\partial_{\alpha} \mathcal{F}_{i}^{(\alpha)}
 ~~\hbox{in}\,\mathcal{B}_{R}^{+},
\end{align*}
and $ \mathcal{W}=0  ~\hbox{on}~ \partial \mathcal{B}_{R}^{+}\cap\partial \mathbb{R}^{n}_{+}$.
Let $y_{0}=\Psi(x_{0})$. Freeze the coefficients and rewrite the above formula in the form
\begin{align*}
\sum\limits_{\alp,\bt,j}\partial_{\alpha}
(\mathcal{A}_{ij}^{\alpha\beta}(y_{0})\partial_{\beta}\mathcal{W}^{(j)}(y))
=&\sum\limits_{\alp,\bt,j}
-\partial_{\alpha}((\mathcal{A}_{ij}^{\alpha\beta}(y)
-\mathcal{A}_{ij}^{\alpha\beta}(y_{0}))\partial_{\beta}\mathcal{W}^{(j)}(y)
+\mathcal{B}_{ij}^{\alpha}(y)\mathcal{W}^{(j)}(y))\nonumber\\
&\quad+\sum\limits_{\alp,\bt,j}
(\mathcal{C}_{ij}^{\beta}(y)\partial_{\beta}\mathcal{W}^{(j)}(y)
-\mathcal{D}_{ij}(y)\mathcal{W}^{(j)}(y)
+\mathcal{H}^{(i)}(y)-\partial_{\alpha} \mathcal{F}_{i}^{(\alpha)}(y)).
\end{align*}
 Then, by  Corollary \ref{them3} we have that for $0<R<1$,
\begin{align*}
\big[\nabla \mathcal{W}\big]_{\gamma,\, \mathcal{B}^+_{R/2}}
\leq &C \Big(\frac{1}{R^{1+\gamma}}\| \mathcal{W}\|_{L^\infty( \mathcal{B}^+_{R})}+[ \mathcal{F}]_{\gamma,\, \mathcal{B}^+_{R}} \Big)\\
&+C\sum\limits_{\bt,j}\Big(\big[(\mathcal{A}_{ij}^{\alpha\beta}(y)-\mathcal{A}_{ij}^{\alpha\beta}(y_{0}))\partial_{\beta}\mathcal{W}^{(j)}
\big]_{\gamma, \, \mathcal{B}^+_{R} }
+\big[\mathcal{B}_{ij}^{\alpha}(y)\mathcal{W}^{(j)} \big]_{\gamma, \, \mathcal{B}^+_{R} }\Big)\nonumber\\
&+C\Big( \sum\limits_{\bt,j} \left\|\mathcal{C}_{ij}^{\beta}(y)\partial_{\beta}\mathcal{W}^{(j)}
-\mathcal{D}_{ij}(y)\mathcal{W}^{(j)}\right\|_{L^{2,\mu}({B}^+_{R})} +\|\mathcal{H}\|_{L^{2,\mu}({B}^+_{R})}\Big).
\end{align*}
Since $\mathcal{A}_{ij}^{\alpha\beta}(y), \mathcal{B}_{ij}^{\alpha}(y), \mathcal{C}_{ij}^{\beta}(y),
\mathcal{D}_{ij}(y)\in C^{\gamma}(\mathcal{B}_{R}^{+})$, we have
\begin{align*}
\sum\limits_{\bt,j}\big[(\mathcal{A}_{ij}^{\alpha\beta}(y)
-\mathcal{A}_{ij}^{\alpha\beta}(y_{0}))\partial_{\beta}\mathcal{W}^{(j)}
\big]_{\gamma, \, \mathcal{B}^+_{R} }\leq& C\left(  R^\gamma  [\nabla \mathcal{W}]_{\gamma,\, \mathcal{B}^+_{R}}+\|\nabla  \mathcal{W}\|_{L^\infty(\mathcal{B}_R^+)}\right),\\
\sum\limits_{j}\big[\mathcal{B}_{ij}^{\alpha}(y)\mathcal{W}^{(j)} \big]_{\gamma, \, \mathcal{B}^+_{R}}
\leq& C R^{\gamma}\|\mathcal{W}\|_{L^{\infty}( \mathcal{B}^+_{R})},
\end{align*}
and
\begin{align*}
\sum\limits_{\bt,j}\left\|\mathcal{C}_{ij}^{\beta}(y)\partial_{\beta}\mathcal{W}^{(j)}
-\mathcal{D}_{ij}(y)\mathcal{W}^{(j)}\right\|_{L^{2,\mu}({\mathcal{B}}^+_{R})}
\leq C\|\nabla \mathcal{W}\|_{L^{\infty}(\mathcal{B}^+_{R})}+\|\mathcal{W}\|_{L^{\infty}(\mathcal{B}^+_{R})}.
\end{align*}
In view of  the  interpolation inequality (\cite[Lemma 6.35]{gt}), we can obtain
\[\|\nabla  \mathcal{W}\|_{L^\infty(\mathcal{B}_R^+)} \leq R^{\gamma }[\nabla \mathcal{W}]_{\gamma, \, \mathcal{B}^+_{R} } + \frac{C}{R}\| \mathcal{W}\|_{L^\infty( \mathcal{B}^+_{R})},\]
where $C=C(n)$. Hence,
\begin{align*}
\big[\nabla \mathcal{W}\big]_{\gamma,\, \mathcal{B}^+_{R/2}}
\leq C (\frac{1}{R^{1+\gamma}}\|\mathcal{W}\|_{L^\infty( \mathcal{B}^+_{R})}+R^\gamma [\nabla \mathcal{W}]_{\gamma, \, \mathcal{B}^+_{R} }+ [\mathcal{F}]_{\gamma,\, \mathcal{B}^+_{R}}+\|\mathcal{H}\|_{L^{2,\mu}({B}^+_{R})}).
\end{align*}
Since $\Psi$ is a homeomorphism, thus, changing back to the variable $x$, we obtain
\[\big[\nabla {\sf w}\big]_{\gamma,\, \mathcal{N}^\prime} \leq C \Big
(\frac{1}{R^{1+\gamma}}\|\widetilde{\sf w}\|_{L^\infty(\mathcal{N})}+R^\gamma [\nabla {\sf w}]_{\gamma, \, \mathcal{N} }+ [{\sf F}]_{\gamma,\, \mathcal{N}}+\|\wdt{\sf H}\|_{L^{2,\mu}(\mathcal{N})}\Big),\]
where $\mathcal{N}=\Psi ^{-1}(\mathcal{B}^+_{R})$,  $\mathcal{N}^\prime=\Psi^{-1}(\mathcal{B}^+_{R/2})$ and $C=C(n, \gamma, \Psi)$. Furthermore, there exists a constant $0<\sigma<1$, independent on $R$, such that $ B_{\sigma R}(x_0)\cap Q \subset \mathcal{N}^\prime$.

Therefore, recalling that $\Gamma \subset \ptl Q$ is a boundary portion, for any domain $Q'\subset\subset Q\cup \Gamma$ and for each $x_0\in Q^\prime \cap\Gamma$,
there exist  $\mathcal{R}_{0}:=\mathcal{R}_{0}(x_0)$ and $C_0=C_0(n, \gamma, x_0)$
such that,
\begin{align}\label{nearGamma}
\big[\nabla {\sf w}\big]_{\gamma,\, B_{\mathcal{R}_{0}}(x_0)\cap Q^\prime}
\leq C_0 \Big(\mathcal{R}_{0}^\gamma [\nabla {\sf w}]_{\gamma,
 \, Q^\prime}+\frac{1}{\mathcal{R}_{0}^{1+\gamma}}\|{\sf w}\|_{L^\infty( Q)}+ [{\sf F}]_{\gamma,\, Q}
+\|{\sf H}\|_{L^{2,\mu}(Q)}\Big).
\end{align}
By using the boundary estimates \eqref{nearGamma} near $\Gamma,$
 the finite covering theorem, and
 Lemma \ref{them5.14},
we can obtain
\begin{equation*}
\big[\nabla {\sf w}\big]_{\gamma,\, Q^\prime } \leq C\left( \|{\sf w}\|_{L^\infty( Q)}
+[{\sf F}\,]_{\gamma,\, Q}
+\|{\sf H}\|_{L^{2,\mu}(Q)}
\right),
\end{equation*}
where $C=C(n, \gamma, Q^\prime, Q)$. The proof details can be referred to the proof
of \cite[Theorem 2.3]{CL}. By using the interpolation inequality
(\cite[Lemma 6.35]{gt}), we obtain \eqref{WC1gamma}.
\end{proof}

\subsection{Proof of Theorem \ref{them2}}
In this subsection, we give the proof of $W^{1,p}$ estimates to the weak
solution of the systems as in Definition \ref{defn2}.

\begin{proof}[Proof of Theorem \ref{them2}]
First, we give the $W^{1, p}$ interior estimates. For any ball
$B_{3R/4}:=B_{3R/4}(x_0) \subset Q$ with $R \leq 1$,
since ${\sf w}\neq 0$ on $\ptl B_{3R/4}$, we choose a cut-off function $\eta\in C_0^\infty(B_{3R/4})$ such that
for $0 <\rho <3R/4$,
\begin{equation*}
0 \leq \eta \leq 1 \quad \eta= 1 \quad \text{in} ~~B_{\rho}, \quad |\nabla \eta| \leq \frac{C}{R-\rho}.
\end{equation*}
We have $\eta {\sf w}$ satisfies
\begin{align*}
&\sum\limits_{\alp,\bt,i,j}\int_{B_{3R/4}}
{A}_{ij}^{\alpha\beta}(x_{0})\ptl_{\bt}(\eta {w}^{(j)})
\ptl_{\alp}\phi^{(i)}
\,dx\\
=&\sum\limits_{\alp,\bt,i,j}\int_{B_{3R/4}}
\left({A}_{ij}^{\alpha\beta}(x_{0})-{A}_{ij}^{\alpha\beta}(x)\right)\ptl_{\bt}(\eta{w}^{(j)})\ptl_\alp \phi^{(i)} \, dx\nonumber\\
&+ \sum\limits_{\alp,i}\Big(\int_{B_{3R/4}} T^{(i)} \phi^{(i)} \,dx
+ \int_{B_{3R/4}} K_{i}^{(\alp)}\ptl_\alp \phi^{(i)}\Big)
,
\quad\forall\,{\sf \phi} \in C_0^\infty(B_{3R/4}; \mathbb{R}^m),
\end{align*}
 where
\begin{align*}
T^{(i)}=&-\sum\limits_{\alp,\bt,j}\left(\big({A}_{ij}^{\alp\beta}(x)\ptl_{\beta}{w}^{(j)}
+{B}_{ij}^{\alp}(x){w}
^{(j)}\big)\ptl_{\alp}\eta
-{C}_{ij}^{\beta}(x)\big(2{w}^{(j)}\ptl_{\beta}\eta+\eta\ptl_{\beta} {w}^{(j)}\big)\right)
\nonumber\\
&\quad-\sum\limits_{j}{D}_{ij}(x)(\eta{w}^{(j)})
+{H}^{(i)}\eta
+\sum\limits_{\alp}({F}_{i}^{(\alp)}-({F}_{i}^{(\alp)})_{3R/4})\ptl_{\alp}\eta,
\\
K_{i}^{(\alp)}=&\sum\limits_{\bt,j}\left({A}_{ij}^{\alp\beta}(x){w}^{(j)}\ptl_{\beta}\eta
-{B}_{ij}^{\alp}(x){w}^{(j)}\eta
+({F}_{i}^{(\alp)}-({F}_{i}^{(\alp)})_{3R/4})\eta\right).
\end{align*}

Let ${\sf v}=(v^{(1)},\cdots,v^{(m)})\in H_0^1(B_{3R/4}; \R^m)$ be the weak solution of
\begin{align}\label{1.18}
-\Delta v^{(i)}=T^{(i)}.
\end{align}
Thus, we can obtain that $\eta{\sf w}$ satisfies
\begin{align*}
\sum\limits_{\alp,\bt,i,j}\int_{B_{3R/4}}
&{A}_{ij}^{\alpha\beta}(x_{0})\ptl_{\bt}(\eta {w}^{(j)})
\ptl_{\alp}\phi^{(i)}
\,dx&\nonumber\\
\quad=&\sum\limits_{\alp,\bt,i,j}\int_{B_{3R/4}}\left(({A}_{ij}^{\alpha\beta}(x_{0})-{A}_{ij}^{\alpha\beta}(x))\ptl_{\bt}(\eta {w}^{(j)})+{K}_{i}^{(\alpha)}\right)
\ptl_{\alp}\phi^{(i)}\,dx,
\end{align*}
where ${K}_{i}^{(\alp)}:=K_{i}^{(\alp)}+\ptl_{\alp} v^{(i)}$.

Since ${F}_{i}^{(\alp)}\in C^{\gamma}(B_{3R/4})$, then $({F}_{i}^{(\alp)}-({F}_{i}^{(\alp)})_{3R/4}))\in L^p(B_{3R/4})$ for any $n\leq p<\infty$. We firstly assume that ${\sf w}\in W^{1, q}(B_{3R/4}; \R^n)$, $q\geq 2$.
Then, combining with Sobolev embedding theorem and the boundedness of coefficients, ${H}^{(i)}\in L^{p}(B_{3R/4})$, we can get
\begin{align*}
K_{i}^{(\alp)}\in L^{\min(p,q^*)}(B_{3R/4})\quad\hbox{and}\quad T^{(i)}\in L^{\min(p,q)}(B_{3R/4}),
\end{align*}
where $q^*:=\frac{nq}{n-q}$ is the Sobolev conjugate of $q$.
To write simply, we use $a\wedge b$ to represent $\min(a,b)$.
By the \eqref{1.18},
\begin{align*}
-\Delta(\ptl_{\alp} v^{(i)})=\ptl_\alp T^{(i)}, \quad \hbox{for}~~~ \alp=1,2,\cdots,n.
\end{align*}
The \cite[Theorem 7.1]{Gia} guarantees that $\nabla (\ptl_{\alp}v^{(i)})\in L^{p\wedge q}(B_{3R/4})$ and
\begin{align*}
\|\nabla (\ptl_{\alp}v^{(i)})\|_{L^{p \wedge q}(B_{3R/4})}\leq C\|T^{(i)}\|_{L^{p\wedge q}(B_{3R/4})},
\end{align*}
where $C$ depends on $n, \lambda, p, q$. Combining with the Sobolev embedding theorem that
\begin{align*}
 \ptl_{\alp}v^{(i)}\in L^{(p\wedge q)^*}(B_{3R/4}).
\end{align*}
It follows from $p\wedge q^*\leq (p\wedge q)^* $ that ${K}_{i}^{(\alp)}\in L^{p\wedge q^*}(B_{3R/4})$.
Let $s:=p\wedge q^*$ and define ${T}: W_{0}^{1,s}(B_{3R/4})\rightarrow W_{0}^{1,s}(B_{3R/4})$ as follows,
$$
{T}({\sf V})={\sf v},\quad \hbox{for any}~ {\sf V}\in W_{0}^{1,s}(B_{3R/4}),
$$
where $v\in W_{0}^{1,s}(B_{3R/4})$ is the solution of the following elliptic system:
\begin{align*}
\sum\limits_{\alp,\bt,i,j}\int_{B_{3R/4}}
{A}_{ij}^{\alpha\beta}(x_{0})\ptl_{\bt}v^{(j)}
\ptl_{\alp}\phi^{(i)}
\,dx=\sum\limits_{\alp,\bt,i,j}\int_{B_{3R/4}}\left(({A}_{ij}^{\alpha\beta}(x_{0})-{A}_{ij}^{\alpha\beta}(x))
\ptl_{\bt}V^{(j)}+{K}_{i}^{(\alpha)}\right)
\ptl_{\alp}\phi^{(i)}\,dx.
\end{align*}
By using  \cite[Theorem 7.1]{Gia} again , we have
\begin{align}\label{1.22}
\|\nabla {\sf v}\|_{L^{s}(B_{3R/4})}\leq C
\|[{\sf A}(x_{0})-{\sf A}(x)]\nabla{\sf  V}\|_{L^{s}(B_{3R/4})}
+C\|{\sf K}\|_{L^{s}(B_{3R/4})},
\end{align}
where ${\sf A}(x)\nabla{\sf V}$ represent matrix $({A}_{ij}^{\alpha\bt}(x)\ptl_{\bt}V^{(j)})$, and $C$ depends on
$n, \lambda,  \|{\sf A}\|_{C^{\gamma}(Q)},p,q$.

When $R$ is sufficiently small, it is proved by Poincar\'{e} inequality and \eqref{1.22} that ${T}$ is a contractive mapping.
 $\eta {\sf w}$ is the only fixed point. See \cite[Theorem 7.2]{Gia} for details.
By \eqref{1.22}, we have
\begin{align*}
\|\nabla (\eta{\sf w})\|_{L^{s}(B_R)}\leq C
\|[{\sf A}(x_{0})-{\sf A}(x)]\nabla (\eta{\sf w})\|_{L^{s}(B_R)}
+C\|{\sf K}\|_{L^{s}(B_R)}.
\end{align*}
Therefore, for sufficiently small $R$, we can obtain
\begin{align*}
\|\nabla (\eta {\sf w})\|_{L^{p\wedge q^*}(B_{3R/4})}&\leq C\|{\sf K}\|_{L^{p\wedge q^*}(B_{3R/4})}\nonumber\\
&\leq \|{\sf T}\|_{L^{p\wedge q}(B_{3R/4})}+\|{\sf K}\|_{L^{p\wedge q^*}(B_{3R/4})}\nonumber\\
&\leq \frac{C}{R-\rho}
\left([{\sf F}]_{\gamma,B_{3R/4}}+\|{\sf H}\|_{L^{\infty}(B_{{3R/4}})}+\| {\sf w}\|_{W^{1,q}(B_{3R/4})}
\right),
\end{align*}
where $C$ depends on $ n, \lambda,\|{\sf A}\|_{C^{\gamma}(Q)}, p, q$. Thus
\begin{align}\label{1.23}
\|\nabla {\sf u}\|_{L^{p\wedge q^*}(B_{\rho})}\leq
\frac{C}{R-\rho}\left([{\sf F}]_{\gamma,B_{3R/4}}+\|{\sf H}\|_{L^{\infty}(B_{3R/4})}+\| {\sf w}\|_{W^{1,q}(B_{3R/4})}
\right).
\end{align}
Next, we prove that $\nabla {\sf w}\in L^p(B_{R/2})$. Similar to the proof of \cite[Theorem 2.4]{CL},
choose a series of balls with radii
\begin{equation*}
\frac{R}{2}<\cdots<R_k<\cdots<R_2<R_1<\frac{3R}{4}.
\end{equation*}
First, let $\rho=R_1, q=2$ in \eqref{1.23}, then
\begin{align*}
\|\nabla {\sf u}\|_{L^{p\wedge 2^*}(B_{R_1})}\leq
\frac{C}{R-R_1}\left([{\sf F}]_{\gamma,B_{3R/4}}+\|{\sf H}\|_{L^{\infty}(B_{3R/4})}+\|{\sf w}\|_{W^{1,2}(B_{3R/4})}
\right).
\end{align*}
If $p\leq 2^*$, it can be obtained by interpolation inequality (see \cite[Theorem 5.8]{ada}) that
\begin{align*}
\|{\sf w}\|_{L^p(B_{R_1})}\leq C\|{\sf w}\|_{W^{1,2}(B_{R_1})}^{\theta}\|{\sf w}\|_{L^{2}(B_{R_1})}^{1-\theta}
\leq C\|{\sf w}\|_{W^{1,2}(B_{R_1})},
\end{align*}
where $\theta=n/2-n/p$ with $2\leq p\leq 2^*$. Combining with \eqref{1.23}, the proof is completed.
If $p>2^*$, then $\nabla {\sf w} \in L^{2^*}(B_{R_1})$ and
\begin{align}\label{1.24}
\|\nabla {\sf w}\|_{L^{ 2^*}(B_{R_1})}\leq
\frac{C}{R-R_1}\left([{\sf F}]_{\gamma,B_{3R/4}}+\|{\sf H}\|_{L^{\infty}(B_{3R/4})}+
\|{\sf w}\|_{W^{1,2}(B_{3R/4})}
\right).
\end{align}
By taking $R=R_1$, $\rho=R_2$ and $q=2^*$ in \eqref{1.23} and combining with \eqref{1.24}, one has
\begin{align*}
\|\nabla {\sf w}\|_{L^{ p\wedge 2^{**}}(B_{R_2})}
\leq \frac{C}{(R-R_1)(R_1-R_2)}
\left(
[{\sf F}]_{\gamma,B_{3R/4}}
+\|{\sf H}\|_{L^{\infty}(B_{3R/4})}+
+\|{\sf w}\|_{W^{1, 2}(B_{3R/4})}
\right).
\end{align*}
Similarly, if $p\leq 2^{**}$, using above formula and  interpolation inequality (see \cite[Theorem 5.8]{ada}), we have completed the proof of the theorem.

If $p>2^{**}$, continuing the above argument within finite steps,  with the
help of interpolation inequality (see \cite[Theorem 5.8]{ada}), we obtain
\begin{equation}\label{1.34}
\| {\sf w}\|_{W^{1,p}(B_{R/2})}\leq C\left(
[{\sf F}]_{\gamma,B_{3R/4}}
+\|{\sf H}\|_{L^{\infty}(B_{3R/4})}
+\|{\sf w}\|_{W^{1,2}(B_{3R/4})} \right),
\end{equation}
where $C$ depends on $n, \lam,  p, \|{\sf A}\|_{C^{\gamma}(Q)}, \dist(B_R, \ptl Q))$.
Similar to the proof of \eqref{1.33}, we can obtain
\begin{align*}
\int_{B_{3R/4}}|\nabla{\sf w}|^2\,dx
\leq C\left(\|{\sf w}\|^2_{L^{2}(B_R)}
+[{\sf F}]^2_{\gamma, B_{R}}+\|{\sf H}\|^2_{L^{2,\mu}(B_{R})} \right).
\end{align*}
This, combining with \eqref{1.34}, we can obtain
\begin{align}\label{1.40}
\| {\sf w}\|_{W^{1,p}(B_{R/2})}\leq C\left(\|{\sf w}\|_{L^{2}(B_{R})}
+[{\sf F}]_{\gamma,B_{R}}
+\|{\sf H}\|_{L^{\infty}(B_{R})}
\right).
\end{align}

Now, we prove the $W^{1, p}$ estimates near boundary $\Gamma$ by using the technology of locally flattening the boundary, which is the same to the proof in Theorem \ref{C1Gamma}. For simplicity, we use the same notation. Hence, we have
that $\mathcal{W}(y):={\sf w}(\Psi ^{-1}(y)) \in W^{1,2}(\mathcal{B}_R^+\subset \R^{n}, \R^m)$ satisfies
\begin{align*}
\sum\limits_{\alp,\bt,i,j}\int_{\mathcal{B}_R^+}\left(\mathcal{A}_{ij}^{\alpha\beta}\partial_{\beta}\mathcal{W}^{(j)}
+\mathcal{B}_{ij}^{\alpha}\mathcal{W}^{(j)}\right)
\ptl_{\alp} \phi^{(i)}
&+\mathcal{C}_{ij}^{\beta}\partial_{\beta}\mathcal{W}^{(j)}\phi^{(i)}
+\mathcal{D}_{ij}\mathcal{W}^{(j)}\phi^{(i)}\,dy\nonumber\\
=&\sum\limits_{\alp,i} \int_{\mathcal{B}_R^+}\mathcal{H}^{(i)}\phi^{(i)}+\mathcal{F}_{i}^{(\alpha)}\partial_{\alpha} \phi^{(i)}\,dy,
\end{align*}
for any $\phi\in W_{0}^{1,2}(\mathcal{B}_R^+, \R^m).$
In this special case, we can obtain the boundary estimate of the upper half space by using the above method of proving the interior  estimate \eqref{1.40},
thus, for $n\leq p<\infty$ we have
\begin{align*}
\|\mathcal{W}\|_{W^{1,p}(\mathcal{B}^+_{R/2})}\leq C\left(  \| \mathcal{W}\|_{L^2(\mathcal{B}_{R}^+)}+[\mathcal{F}]_{\gamma, \mathcal{B}_{R}^+}
+\|\mathcal{H}\|_{L^{\infty}(\mathcal{B}_{R}^+)} \right),
\end{align*}
where $C$ depends on $\lambda, \Lambda, p, R, \Psi$.
Then, changing back to the original variable $x$, we obtain
\begin{align*}
\|{\sf w}\|_{W^{1,p}(\mathcal{B}^+_{R/2})}\leq C\left(  \| {\sf w}\|_{L^2(\mathcal{B}_{R}^+)}+[{\sf F}]_{\gamma, \mathcal{B}_{R}^+}
+\|{\sf H}\|_{L^{\infty}(\mathcal{B}_{R}^+)} \right),
\end{align*}
where $\mathcal{N}^\prime=\Psi^{-1}(\mathcal{B}^+_{R/2})$, $\mathcal{N}=\Psi ^{-1}(\mathcal{B}^+_{R})$ and $C=C(\lam, \mu, p, R, \Psi)$. Furthermore, there exists a constant $0<\sigma<1$, independent on $R$, such that $ B_{\sigma R}(x_0)\cap Q \subset \mathcal{N}^\prime$.
Therefore, for any $x_0 \in Q'\cap \Gamma $, there exists $R_0:=R_0(x_0)>0$ such that,
\begin{align}\label{1.41}
\|\nabla {\sf w}\|_{W^{1,p}(B_{\sigma R_0}(x_0)\cap Q')}\leq C(\|{\sf w}\|_{L^2(Q)}+[{\sf F}]_{\gamma, Q}
+\|{\sf H}\|_{L^{\infty}(Q)}),
\end{align}
where $C$ depends on $\lambda, \Lambda, p, x_{0}, R$. Combining \eqref{1.40} and \eqref{1.41} and  making use of the finite covering theorem, We have completed the proof of the Theorem \ref{them2}.
Refer to the proof of  \cite[Theorem 2.4]{CL} for more details.
\end{proof}

\bigskip

\noindent Yan Li (Corresponding author)

\smallskip

\noindent
School of Mathematical Sciences, Beijing Normal University,
Beijing 100875, China.

\smallskip

\noindent {\it E-mail}: \texttt{yanli@mail.bnu.edu.cn} (Y. Li)


\begin{thebibliography}{99}
\bibitem{ada}
R.A. Adams; J.J.F. Fournier, Sobolev spaces. Second edition. Pure and Applied Mathematics (Amsterdam), 140. Elsevier/Academic Press, Amsterdam, 2003. xiv+305 pp.

\vspace{-0.3cm}

\bibitem{ackly} H. Ammari; G. Ciraolo; H. Kang; H. Lee; K. Yun, Spectral analysis of the Neumann-Poincar\'{e} operator and characterization of the stress concentration in anti-plane elasticity.
     Arch. Ration. Mech. Anal. 208 (2013) 275--304.



\vspace{-0.3cm}

\bibitem{adkl}  H. Ammari; G. Dassios; H. Kang; M. Lim, Estimates for the electric field in the presence of adjacent perfectly conducting spheres. Quat. Appl. Math. 65 (2007) 339--355.

\vspace{-0.3cm}

\bibitem{aklll}  H. Ammari; H. Kang; H. Lee; J. Lee; M. Lim, Optimal estimates for the electrical
field in two dimensions. J. Math. Pures Appl. 88 (2007) 307--324.

\vspace{-0.3cm}

\bibitem{akllz} H. Ammari; H. Kang; H. Lee; M. Lim; H. Zribi, Decomposition theorems and fine estimates for electrical
fields in the presence of closely located circular inclusions. J.
Differential Equations 247 (2009) 2897--2912.

\vspace{-0.3cm}

\bibitem{akl} H. Ammari; H. Kang; M. Lim, Gradient estimates to the conductivity problem. Math.
Ann. 332 (2005) 277--286.

\vspace{-0.3cm}

\bibitem{basl} I. Babu\u{s}ka; B. Andersson; P. Smith; K.
Levin, Damage analysis of fiber composites. I. Statistical analysis
on fiber scale. Comput. Methods Appl. Mech. Engrg. 172 (1999) 27--77.

\vspace{-0.3cm}

\bibitem{bjl}
 J.G. Bao; H.J. Ju; H.G. Li, Optimal boundary gradient estimates for Lam\'e systems with partially infinite coefficients. Adv. Math. 314 (2017) 583--629.

\vspace{-0.3cm}

\bibitem{bll} J.G. Bao; H.G. Li; Y.Y. Li, Gradient estimates for solutions of the Lam\'{e} system with partially infinite coefficients. Arch. Ration. Mech. Anal. 215 (2015) 307--351.

\vspace{-0.3cm}

\bibitem{bll2} J.G. Bao; H.G. Li; Y.Y. Li, Gradient estimates for solutions of the Lam\'{e} system with partially infinite coefficients  in dimensions greater than two. Adv. Math. 305 (2017) 298--338.

\vspace{-0.3cm}

\bibitem{bly1} E. Bao; Y.Y. Li; B. Yin, Gradient estimates for the perfect conductivity problem. Arch. Ration. Mech. Anal. 193 (2009) 195--226.

\vspace{-0.3cm}

\bibitem{bly2} E. Bao; Y.Y. Li; B. Yin, Gradient estimates for the perfect and insulated conductivity problems with multiple inclusions. Comm. Partial Differential Equations 35 (2010) 1982-2006.

\vspace{-0.3cm}

\bibitem{bt2}E. Bonnetier; F. Triki, On the spectrum of the Poincar\'{e} variational problem for two close-to-touching inclusions in 2D.
Arch. Ration. Mech. Anal. 209 (2013)  541--567.


\vspace{-0.3cm}

\bibitem{bv} E. Bonnetier; M. Vogelius, An elliptic regularity result for a composite medium
with ``touching'' fibers of circular cross-section.
SIAM J. Math. Anal. 31 (2000) 651--677.



\vspace{-0.3cm}

\bibitem{CL} Y. Chen; H.G. Li, Estimates and asymptotics for the stress concentration between closely spaced stiff $C^{1,\gamma}$ inclusions in linear elasticity. J. Funct. Anal. 281 (2021).

\vspace{-0.3cm}

\bibitem{dong} H.J. Dong, Gradient estimates for parabolic and elliptic systems from linear laminates.
Arch. Ration. Mech. Anal. 205 (2012) 119--149.


\vspace{-0.3cm}

\bibitem{dongzhang} H.J. Dong; H. Zhang, On an elliptic equation arising from composite materials. Arch. Rational Mech. Anal. 222 (2016)  47--89.

\vspace{-0.3cm}

\bibitem{JJ}J.E. Flaherty and J.B. Keller,
Elastic behavior of composite media.
Comm. Pure Appl. Math. 26 (1973) 565--580.

\vspace{-0.3cm}

\bibitem{Gia} M. Giaquinta; L. Martinazzi, An introduction to the regularity theory for elliptic systems, harmonic maps and minimal graphs. Springer
    Science  Business Media, 2013.

\vspace{-0.3cm}

\bibitem{gt} D. Gilbarg; N. S. Trudinger, Elliptic partial differential equations of second order. Springer 1998.

\vspace{-0.3cm}

\bibitem{GoBe1} Y. Gorb; L. Berlyand, Asymptotics of the effective conductivity of composites with closely spaced inclusions of optimal shape. Quart. J. Mech. Appl. Math. 58 (2005) 84--106.
\vspace{-0.3cm}

\bibitem{GoBe2} Y. Gorb; L. Berlyand,  The effective conductivity of densely packed high contrast composites with inclusions of optimal shape. Continuum models and discrete systems, 63--74,
NATO Sci. Ser. II Math. Phys. Chem., 158, Kluwer Acad. Publ., Dordrecht, 2004.

\vspace{-0.3cm}

\bibitem{jlx} H.J. Ju; H.G. Li; L.J. Xu,
Estimates for elliptic systems in a narrow region arising from composite materials.
Quart. Appl. Math. 77 (2019) 177--199.


\vspace{-0.3cm}

\bibitem{kly0} H. Kang; H. Lee; K. Yun,
Optimal estimates and asymptotics for the stress concentration between closely located stiff inclusions.
Math. Ann. 363 (2015) 1281--1306.


\vspace{-0.3cm}

\bibitem{ky} H. Kang; S. Yu,
Quantitative characterization of stress concentration in the presence of closely spaced hard inclusions in two-dimensional linear elasticity.
Arch. Ration. Mech. Anal. 232 (2019) 121--196.

\vspace{-0.3cm}

\bibitem{LHLY} H.G. Li; Y. Li, An extended Flaherty-Keller formula for an elastic composite with densely packed convex inclusions. arXiv:1912.13261.

\vspace{-0.3cm}

\bibitem{llby} H.G. Li; Y.Y. Li; E.S. Bao; B. Yin,
Derivative estimates of solutions of elliptic systems in narrow regions. Quart. Appl. Math. 72 (2014)
589--596.


\vspace{-0.3cm}

\bibitem{ln} Y.Y. Li; L. Nirenberg,
Estimates for elliptic systems from composite material.
 Comm. Pure Appl. Math. 56 (2003) 892--925.

\vspace{-0.3cm}

\bibitem{lv} Y.Y. Li; M. Vogelius,
 Gradient estimates for solutions to divergence form elliptic equations with discontinuous coefficients.
 Arch. Ration. Mech. Anal. 153 (2000) 91--151.

\vspace{-0.3cm}

\bibitem{ly2} M. Lim; K. Yun, Blow-up of electric fields between closely spaced spherical perfect conductors. Comm. Partial Differential Equations. 34 (2009) 1287--1315.

\vspace{-0.3cm}



\bibitem{osy} O.A. Oleinik; A.S. Shamaev; G.A. Yosifian, Mathematical problems in elasticity and homogenization,
Studies in Mathematics and its Applications, 26.
 North-Holland Publishing Co., Amsterdam, 1992. xiv+398 pp.

\vspace{-0.3cm}

\bibitem{V1}
 S.B. Vigdergauz,  On a case of the inverse problem of the two-dimensional theory of elasticity.
 Prikl. Mat. Mekh. 41 (1977) 902--908


\vspace{-0.3cm}

\bibitem{y1} K. Yun, Estimates for electric fields blown up between closely adjacent conductors with arbitrary shape. SIAM J. Appl. Math. 67 (2007)  714--730.

\end{thebibliography}
\end{document}